\documentclass[12pt,reqno]{amsart}
\usepackage{amssymb,amsmath,dsfont}
\usepackage{thmtools}
\usepackage{pgfplots}
\pgfplotsset{compat=1.14}
\usepackage{enumitem}
\usepackage{calc}
\usepackage{graphicx}
\usepackage{float}
\usepackage{stackrel}
\usepackage{amsmath}
\usepackage{bbm}
\usepackage{hyperref}
\usepackage{mathtools}
\newcommand{\subscript}[2]{$#1 #2$}
\usepackage{amsthm}
\usepackage{setspace}
\usepackage[utf8]{inputenc}
\usepackage{tikz}
\usepackage{caption}
\usepackage{subcaption}

\newtheorem{thm}{Theorem}[section]
\newtheorem{lemma}[thm]{Lemma}
\newtheorem{obser}[thm]{Observation}
\newtheorem*{notation}{Notation}
\newtheorem{prop}[thm]{Proposition}
\newtheorem*{definition}{Definition}
\newtheorem{cor}[thm]{Corollary}
\newtheorem{claim}[thm]{Claim}

\newtheorem{fact}[thm]{Fact}
\newtheorem*{remark}{Remark}

\theoremstyle{definition}

\def\E{\mathbb{E}}

\DeclareMathOperator\codim{codim}
\DeclareMathOperator\comp{comp}
\DeclareMathOperator\Ber{Ber}
\title{The upper tail problem for induced 4-cycles in sparse random graphs}
\author{Asaf Cohen Antonir}

\begin{document}
\maketitle
\begin{abstract}
    Building on the techniques from the breakthrough paper of Harel, Mousset and Samotij, which solved the upper tail problem for cliques, we compute the asymptotics of the upper tail for the number of induced copies of the 4-cycle in the binomial random graph $G_{n,p}$. We observe a new phenomenon in the theory of large deviations of subgraph counts. This phenomenon is that, in a certain (large) range of $p$, the upper tail of the induced 4-cycle does not admit a naive mean-field approximation.  
\end{abstract}
\section{Introduction}

The study of the random variable counting the number of copies of a given graph in the binomial random graph $G_{n,p}$ has a very long history and many things are known about it. 
In particular, for every graph $H$, when the expected number of copies of $H$ in $G_{n,p}$ tends to infinity a `law of large numbers', see Bollob\'as \cite{Bol81}, and a central limit theorem are known, see Ruci\'nski \cite{Ruc88}.

After establishing such results, it is natural to ask what is the probability of the event that this random variable differs from its expectation by a significant amount.
In this paper, we will consider only the `upper tail' problem: what is the probability of a random variable exceeding its expectation by a multiplicative factor $1+\delta$, where $\delta$ is some positive real.

We continue by recalling some definitions to make our discussion more rigorous.
For every non-negative integer $n$ and a sequence $p=p(n)\in (0,1)$, we let $G_{n,p}$ be the binomial random graph with $n$ vertices and density $p$. Furthermore, for any graph $H$, let $X^{H}_{n,p}$ be the random variable counting the number of (labeled) copies of $H$ in $G_{n,p}$. Further, for any $\delta$ a positive real and a random variable $Y$ we write $\mathbf{UT}(Y,\delta)$ for the event $\{Y\geq (1+\delta)\E[Y]\}$.

The work on the problem of estimating (the logarithm of) the upper tail probability of $X_{n,p}^{H}$ was initiated by the famous works of Kim and Vu \cite{KimVu04}, Vu \cite{Vu01}, Janson and Ruci\'nski \cite{JanRuc02}.
This problem turns out to be so difficult\footnote{This problem was called `the infamous upper tail' by Janson and Ruci\'nski \cite{JanRuc02}.} because when $H$ is a connected graph with at least two edges $X_{n,p}^{H}$ is \emph{not} a linear function of independent Bernoulli random variables. In the case of linear functions, life is easier and much is known about its large deviation properties, see \cite{DemZei02}. 

In a famous paper \cite{JanOleRuc04}, Janson, Oleszkiewicz, and Ruci\'nski estimated the logarithm of the upper tail probability for every graph $H$ up to a multiplicative factor of $O(\log(1/p))$. Seven years later, Chatterjee \cite{Cha12} and DeMarco--Kahn \cite{DeMKah10} closed this gap up to a multiplicative factor of $O(1)$ when $H$ is a triangle and DeMarco--Kahn \cite{DeMKah12} generalized this for a clique of arbitrary size.

Next, one would like to obtain a first order approximation of the logarithm of the upper tail probability. The first to do that were Chatterjee and Varadhan \cite{ChaVar11}.
They obtained a first order approximation of the upper tail probability for any graph, under the assumption that $p$ is a constant. Their proof relied on the regularity method and therefore extends only to $p$ tending to zero very slowly, for more discussion about this see \cite{LubZha15}.

The general strategy for estimating the logarithm of the upper tail probability, used by Chatterjee and Varadhan as well as all later works on this subject, is to establish a `large deviation principle'.
That is to prove that the logarithm of the upper tail probability is asymptotically equal to the solution to a minimization problem over a {non-trivial} set of product probability measures.
These minimization problems are called `variational problems'. After achieving such large deviation principle one is then left with solving the variational problem.

In a breakthrough paper of Chatterjee–Dembo \cite{ChaDem16}, the authors established a `large deviation principle' not only for $X_{n,p}^H$ when $p=\Omega(n^{-\varepsilon})$, but also for a large class of functions of $N$ independent Bernoulli random variables with mean $p=\Omega(N^{-\varepsilon})$. They proved that for any `smooth enough' function $f\colon \{0,1\}^N \to \mathbb{R}$ and a sequence of independent Bernoulli random variables $Y=(Y_1,Y_2,\ldots,Y_N)\in \{0,1\}^N$, all with mean $p$, we have the following:
\begin{equation}\label{eq_var_problem}
    -\log \mathbb{P}\left(\mathbf{UT}(f(Y),\delta)\right)=(1+o(1))\inf \left \{\sum _{i=1}^N I_p(q_i):\E[f(\Tilde{Y})]\geq (1+\delta)\E[f(Y)]\right\},
\end{equation}
where $I_p(q)=q\log\left(\frac{q}{p}\right)+(1-q)\log\left(\frac{1-q}{1-p}\right)$ is the relative entropy (the Kullback--Leibler divergence) between $\Ber(q)$ and $\Ber(p)$, and $\Tilde{Y}=(\Tilde{Y}_1,\Tilde{Y}_2,\ldots,\Tilde{Y}_N)\in \{0,1\}^N$ is a sequence of independent Bernoulli random variables with $\E[\Tilde{Y}_i]=q_i$ for each $i$.

This was further developed by Eldan \cite{Eld18} and Augeri \cite{Aug18,Aug19}.
These methods are completely different from the ones used in the dense regime. 

Revisiting the ideas from Chatergee--Varhadan \cite{ChaVar11}, Cook--Dembo \cite{CooDem20} developed a decomposition theorem similar to Szemer\'edi's regularity lemma and a corresponding counting lemma which are suitable for sparse graphs. Using these they extended the range of sparsity of $p$ were the variational approximation \eqref{eq_var_problem} holds (for every graph $H$). 
Generalizing this method, Cook--Dembo--Pham \cite{CooDemPha2021} pushed the bounds even further in the case of subgraph count. They also obtained an approximation of the logarithm of the upper tail probability for the count of uniform sub-hypergraphs in the binomial random uniform hypergraph model, and the induced count of graphs in $G_{n,p}$. Their result combined with the solution to the corresponding variational problem for uniform hypergraph cliques, and one more non-trivial 3-uniform hypergraph which were given by Liu--Zhao in \cite{LiuZha2021}, yields estimations for the sub-hypergraph count in the binomial random 
hypergraph model in the sparse regime.
We also note that to the best of our knowledge the solution to the `induced variational problem' is not known apart from the case where $H$ is a clique. The case of the induced subgraph count will be of main interest in this paper, and will be discussed later. 


To estimate the asymptotics of the logarithm of the upper tail probability, one also needs to solve the variational problem in the right-hand side of \eqref{eq_var_problem}. 
In the case of the random variable $X_{n,p}^H$, Lubetzky and Zhao \cite{LubZha17}, and Bhattacharya, Ganguly, Lubetzky, and Zhao \cite{BhaGanLubZha17} solved this variational problem for all $H,n$ and $p$ satisfying $n^{-1/\Delta_H}\ll p \ll 1$. This then leads to a first order approximation of the logarithm of the upper tail probability for $X_{n,p}^{H}$ for every graph $H$ and $p\geq n^{-c_H}$. We wish to emphasize that, in all known cases, when $f$ counts the number of copies of a given $H$ in $G_{n,p}$,  the solution to the corresponding variational problem (the right hand side of \eqref{eq_var_problem}) always satisfies $q_i\in \{p,1\}$ when $p=o(1)$.

A recent surprising result due to Gunby \cite{Gun20} studied the upper tail probability for subgraph counts in a random regular graph $G_{n,d}$, where he also proved a large deviation principle. However, for a particular graph, and certain range of $d$ (the regularity) the answer for this variational problem is achieved with three possible values of $q_i$ and not two as in the binomial random graph model. 

In a breakthrough paper of Harel, Mousset and Samotij \cite{HarMouSam19}, using a combinatorial approach, the authors managed to extend the range where the variational problem (the same one as before) bounds from above the logarithm of the upper tail probability for the count of cliques to the optimal range of $p$, which is $p^{\frac{r-1}{2}}\gg n^{-1}(\log(n))^{\frac{1}{r-2}}$ for a clique on $r$ vertices\footnote{That is because below this threshold they also showed a Poisson approximation.}. 
We wish to emphasize that their approach for solving the upper tail problem is completely different from any previous techniques in the papers mentioned above.
This, together with the solution to the variational problem, settled the problem of the first order approximation of the logarithm of the upper tail probability for cliques. 
Moreover, they established the correct range of $p$, up to a polylogarithmic factor, where the variational problem bounds from above the logarithm of the upper tail probability for non-bipartite regular graphs\footnote{For bipartite regular graphs their result was not optimal, but also a lot of progress was made.}.

Building on the combinatorial ideas of Harel, Mousset and Samotij, Basak and Basu \cite{BasBas19} extended the range of $p$ to the optimal range for all regular graphs, including bipartite graphs. Again, this with the solution to the variational problem settled the problem of estimating the logarithm of the upper tail probability for regular graphs.
We summarize this discussion with the following first order approximation of the upper tail problem for regular graphs in the `localized regime'. 
\begin{thm}[Due to \cite{HarMouSam19,BasBas19}]\label{thm_main_thm_HMS_BB}
    For any $\Delta\geq 2$ and a $\Delta$-regular connected graph $H$ the following holds: 
    \begin{equation*}
        \frac{-\log\mathbb{P}\left(X_{n,p}^{H}\geq(1+\delta)\E\left[X_{n,p}^{H}\right]\right)}{n^2p^{\Delta}\log(1/p)}=(1+o(1))
            \begin{cases}
                \min\{\theta_H,\frac{1}{2}\delta^{2/v_H}\}   &\text{ if }\frac{1}{\sqrt{n}}\ll p^{\Delta/2} \ll 1,\\
                \frac{1}{2}\delta^{2/v_H} &\text{ if } \frac{\log^{\frac{1}{v_H-2}}(n)}{n}\ll p^{\Delta/2} \ll \frac{1}{\sqrt{n}},
            \end{cases}
    \end{equation*}
where $\theta_H$ is the unique solution to the equation $P_H(\theta)=1+\delta$ where $P_H$ is the independence polynomial of $H$\,\footnote{The independent polynomial of $H$ is $P_H(X)=\sum_{k=0}^{\alpha(H)}s_k X^k$ where $s_k$ is the number of independent sets in $H$ of size $k$.}.
\end{thm}

In this paper we estimate the logarithm of the upper tail probability for the count of \emph{induced} copies of $C_4$ in $G_{n,p}$ in the range $\frac{\log^9(n)}{n}\ll p\ll 1$, which is optimal up to $\log^9(n)$ in the lower bound.
To the best of our knowledge, this is the first exact result for the upper tail probability of a random variable counting the number of induced copies of a given non-complete graph in $G_{n,p}$.
Our proofs rely heavily on the combinatorial approach of Harel, Mousset and Samotij. 

We now present the main result of this paper. For this, from now on let $X$ be the random variable counting the number of \emph{induced} copies of $C_4$ in $G_{n,p}$.

\begin{restatable}{thm}{thebesttheorem}
    \label{thm_main_result}
    There is an explicit sequence $0=c_1<c_2<\ldots\leq 1/3$ such that the following holds for $p\gg \frac{\log^9(n)}{n}$:
    \[
        \frac{-\log \mathbb{P}(X\geq (1+\delta)\E[X])}{n^2p^2\log(1/p)}=(1+o(1))
        \begin{cases}
            \rho_k(n,p) \sqrt{\frac{\delta}{2}}
            &\text{ if } n^{-1+c_{k-1}}\leq p\leq n^{-1+c_{k}},\\
            \sqrt{\frac{\delta}{2}} &\text{ if }n^{-2/3}\leq p\ll \frac{1}{\sqrt{n}\log(n)},\\
            \sqrt{\frac{\delta}{2}+\frac{1}{{128}}}-\frac{1}{\sqrt{128}} &\text{ if }n^{-1/2}\ll p \ll 1,
        \end{cases}
    \]
    where $\rho_k(n,p)=\sqrt{\frac{k}{k-1}}\left(1-\frac{2}{k}+\frac{\log(n)}{k\log(p)}\right)<1$.
\end{restatable}

The sequence above is explicitly defined in Section \ref{sec_6}.

There are several interesting phenomena which distinguish Theorem \ref{thm_main_result} from previous results concerning the upper tail problem for subgraph counts.

The first, and most noticeable difference is the infinite number of phase transitions. To the best of our knowledge, there is no earlier example of an upper tail problem exhibiting infinitely many phase transitions in its first order approximation. To understand the reasons for this phenomenon let us first discuss further Theorem \ref{thm_main_thm_HMS_BB}. 
A strategy to bound from below the upper tail probability is to observe that, conditioned on the event $C \subseteq G_{n,p}$ for some $C$ satisfying $\E\left[X_{n,p}^{H}\mid C\subseteq G_{n,p}\right]\geq (1+\delta)\E\left[X_{n,p}^{H}\right]$, the upper tail event holds with ‘decent’ probability. This allows us to bound the lower tail probability from below by the probability of $C \subseteq G_{n,p}$ (which is $p^{e_C}$) times the ‘decent’ probability above (which is $p^{o(e_C)}$). We refer to this as `planting a copy of $C$'. Of course one would like to consider the smallest such $C$, to increase the lower bound as much as possible.

In the previous works mentioned earlier, it was shown that for $\Delta$-regular graphs there are two natural candidates for $C$.
These candidates are:
a clique on $\Theta(np^{\Delta/2})$ vertices or a spanning complete bipartite graph with the smaller side of size $\Theta(np^{\Delta})$. The second construction is often called a \emph{Hub}.
When $p\ll n^{-1/\Delta}$, the small side of the Hub construction needs to be smaller than one and hence at that point the Hub construction is no longer valid, and we are left only with the clique construction as a lower bound for the upper tail probability. This is the reason for the two different regimes in Theorem \ref{thm_main_thm_HMS_BB}.

Even though $C_4$ is a regular graph, the presence of a clique in $G_{n,p}$ does not boost the expected number of \emph{induced} 4-cycles by a significant amount. Therefore, the first construction is no longer valid.
The analog of this construction in our case is a balanced and complete bipartite graph with $\Theta(np)$ vertices. Surprisingly, although amongst all graphs with a given number of edges the complete and balanced bipartite graph maximizes the number of induced copies of $C_4$ (see \cite{BolNarTac86,DanDanBalMat20} and Lemma~\ref{lemma_inducibility_like}) planting it does not always give the strongest lower bound on the upper tail probability. 

This is because, in a range of densities $p$, there is a large family $\mathcal{K}_p$ of subgraphs of $K_n$, where
each $C\in \mathcal{K}_p$ satisfies $\E[X\mid C\subseteq G_{n,p}]\geq (1+\delta)\E[X]$ and has only slightly suboptimal size (amongst all $C$ with this property), that is so `large' that
\[
  \log\mathbb{P}(C\subseteq G_{n,p}\text{ for some } C \in \mathcal{K}_p) \ge - (1-\Omega(1)) \min_C e(C) \log(1/p),
\] 
where the minimum ranges over all graphs $C$ with $\E[X\mid C\subseteq G_{n,p}]\geq (1+\delta)\E[X]$; moreover, conditioned on the event that $C \subseteq G_{n,p}$ for some $C \in \mathcal{K}_p$, the upper tail event still holds with `decent' probability.  Note that this is very different from the case of non-induced regular graphs, as here we condition on the appearance of \emph{some} graph from $\mathcal{K}_p$ rather than a single graph (which corresponds to tilting the measure of $G_{n,p}$ to $G_{n, \bar{q}}$ for some $\bar{q} \in \{p, 1\}^{\binom{n}{2}}$).  In fact, we will show (in Section~\ref{sec_7}) that the lower bound obtained with the use of $\mathcal{K}_p$ is stronger than any bound corresponding to the more general tilting $G_{n,p}$ to $G_{n,\bar{q}}$ for some $\bar{q} \in [0,1]^{\binom{n}{2}}$.

Let us elaborate on the structure of the graphs in $\mathcal{K}_p$.
The family $\mathcal{K}_p$ depends on $p$ in the following way: provided $n^{-1+c_{k-1}}\leq p\leq n^{-1+c_k}$ for some integer $k \ge 2$, the set $\mathcal{K}_p$ comprises all complete bipartite graphs with sides $k$ and $\ell=2\sqrt{\frac{\delta \E[X]}{k(k-1)}}$.
Note that, for any fixed $k$ and large $n$, the graph $K_{k,\ell}$ has more edges than the complete and balanced bipartite graph with the same number of induced copies of $C_4$.  However, the key point here is that we are no longer planting a \emph{single} copy of this graph.  Since $\ell$ is large, there are many embeddings of $K_{k,\ell}$ in $K_n$ and the probability that $G_{n,p}$ contains \emph{some} such copy is affected by the combinatorial factor of the number of possible embeddings. It turns out that the lower bound obtained by considering these families is actually tight in the range where $\frac{\log^9(n)}{n}\ll p\ll n^{-2/3-o(1)}$.  Further, between $\frac{\log^9(n)}{n}$ and $n^{-2/3-o(1)}$ the family $\mathcal{K}_p$ changes infinitely many times depending on $p$, which explains the infinite number of phase transitions. This will be explained in more details in Section~\ref{sec_7}.

\subsection*{Organisation of the paper} In Section \ref{sec_2}, we prove a general large deviation principle for nonnegative functions of independents Bernoulli variables which is a version of Theorem~9.1 in \cite{HarMouSam19}, which was stated in \cite{HarMouSam19} but was not proven. In Section \ref{sec_3}, we continue by connecting the general large deviation principle to our large deviation problem. To this end we define special classes of graphs called core graphs. Then we modify the general result proved in Section \ref{sec_2} using this notion of core graphs. 

Section \ref{sec_4} is independent of the others and gives lower bounds for the upper tail probability by planting graphs in the denser regimes and families of graphs in the sparser regimes.

In Section \ref{sec_5}, we solve the variational problem presented in Section \ref{sec_3}. Furthermore, we make a connection between the number of core graphs with $m$ edges and the maximum number of vertices a core graph with $m$ edges might have.  

Section \ref{sec_6} uses the results from Section \ref{sec_5} to provide upper bounds for the logarithm of the upper tail probability. This section is divided into three parts. First is the dense regime $n^{-1/2}\ll p\ll 1$, second is the sparse regime $\frac{\log^9(n)}{n}\ll p\ll \frac{1}{\sqrt{n}\log(n)}$ which we also split into two cases: the dense case in the sparse regime $n^{-2/3}\ll p\ll \frac{1}{\sqrt{n}\log(n)}$ and the sparser cases in the sparse regime $n^{-1+c_{k-1}}\ll p\ll n^{-1+c_{k}}$. Before the second split, we develop some tools used in both cases. 

In Section \ref{sec_7} we solve the naive mean field variational problem, showing that it is different the solution to the variational problem we used.

Appendix \ref{app_A} was added for completeness, where we reprove results of \cite{LubZha17, BhaGanLubZha17} in a language of graphs rather than graphons.

\subsection*{Acknowledgments} The author thanks his advisor Wojciech Samotij for his guidance through out the write-up of this paper, as well as helpful and fruitful discussions.
The author also thanks Arnon Chor and Dor Elboim for helpful discussions. Lastly, the author thanks Eden Kuperwasser for fruitful discussions and for creating the figure in this paper.

\section{Main tools - polynomials in the hypercube}\label{sec_2}
We start with some notations and definitions which can also be found in \cite{HarMouSam19}. We work within the probability space $(\{0,1\}^N,\Ber({p})^N)$.
Suppose $I\subseteq [N]$ and $x\in \{0,1\}^N$. We say that
\[
F(I,x)=\{y\in\{0,1\}^N: y_n=x_n \text{ for all } n\in I \}
\]
is a subcube centered at $(I,x)$ with \emph{codimension} $|I|$ denoted by $\codim(F)=|I|$. Note that every subcube is centered at some pair $(I,x)$. Moreover, if a subcube $F$ is centered at $(I,x)$ and it is also centered at $(J,y)$ then $I=J$. Hence, the codimension is well defined. For every subcube $F$ centered at $(I,x)$ we define the \emph{one-supcube} and the \emph{zero-supcube} of $F$ to be $F^{(1)}=F(I_1,x)$ and $F^{(0)}=F(I_0,x)$ where $I_i=\{n\in I:x_n=i\}$ for $i=0,1$. Moreover, for every subcube $F$ we say that the \emph{one-codimension} of it is the codimension of $F^{(1)}$ and the \emph{zero-codimension} of it is the codimension of $F^{(0)}$ denoted by $\codim_i(F)$ for $i=1$ and $i=0$ respectively. A simple observation is that any subcube $F$ always satisfies $F=F^{(0)}\cap F^{(1)}$ and $\codim(F)=\codim(F^{(0)})+\codim(F^{(1)})$. 
For every $X\colon \{0,1\}^N \rightarrow  \mathbb{R}_{\geq 0}$ we define the \emph{complexity} of $X$ denoted by $\comp(X)$ to be the smallest integer $d$ for which it is possible to represent $X$ as a linear combination with nonnegative coefficients of indicator functions of subcubes with codimension at most $d$. Note that the complexity is well defined as for every such function $X$ we can write \[X=\sum\limits_{x\in \{0,1\}^N}X(x)\mathbbm {1}_{F([N],x)}.\] Moreover, this shows that for any $X$ we have $\comp(X)\leq N$.
Assume now that $Y$ is a random variable taking values in $\{0,1\}^N$ and that $X=X(Y)$. Given a subcube $F\subseteq \{0,1\}^N$, we write $\mathbb{E}_F[X]=\E[X\mid Y\in F]$ for the expectation of $X$ conditioned on $Y\in F$.
We further define $\Phi_X\colon \mathbb {R}_{\geq 0 }\rightarrow \mathbb {R}_{\geq 0 }\cup \{\infty \}$ by
\begin{align*}
    \Phi_X(\delta)=\min \{&-\log \mathbb{P}(Y\in F):F\subseteq \{0,1\}^N\text{ is a subcube with }\mathbb{E}_F[X]\geq (1+\delta)\mathbb {E}[X]\}.
\end{align*}
Our main tool is \cite[Theorem~9.1]{HarMouSam19}, which was stated but not proved. 
We prove a slightly different version of the theorem, but the essence remains the same.

\begin{thm}\label{thm:9.1}
For every positive integer $d$ and all positive real numbers $\varepsilon$ and $\delta$ with $\varepsilon <\frac{1}{2}$, there is a positive constant $K=K(\varepsilon,\delta,d)$ such that the following holds. Let $Y$ be a sequence of $N$ independent $\Ber(p)$ random variables for some $p\in (0,1/2)$ and assume that $X=X(Y)$ is nonnegative with complexity at most $d$ and satisfies $\Phi_X(\delta -\varepsilon)\geq K\log(\frac{1}{p})$. Denote by $\mathcal{F}$ the collection of all subcubes $F\subseteq \{0,1\}^N$ satisfying
\begin{enumerate}[label=(\subscript{F}{{\arabic*}})]
    \item\label{cond_1} $\mathbb{E}_F[X]\geq (1+\delta -\varepsilon)\E [X]$,
    \item\label{cond_2} $\codim (F)\leq K\cdot\Phi _X(\delta+\varepsilon)$.
\end{enumerate}
Then,
\begin{equation*} \label{eq_Main1}
    \mathbb{P}(X\geq (1+\delta)\E[X])\leq (1+\varepsilon) \mathbb{P}(Y\in F \text{ for some } F\in \mathcal{F}).
\end{equation*}
\end{thm}

To prove the theorem we state and prove the following lemmas.

\begin{lemma}\label{lemma_upper}
Let $Y$ be a random variable taking values in $\{0,1\}^N$ and let $X=X(Y)$ be a real-valued function of $Y$. Suppose that $\E[X]>0$ and that $X\leq M$ always. Then for all positive $\varepsilon$ and $\delta$,
$$-\log\mathbb{P}(X\geq (1+\delta)\E [X])\leq \Phi_X(\delta+\varepsilon)+\log\bigg (\frac{M}{\varepsilon\E[X]}\bigg)$$
\end{lemma}

\begin{proof}
Let $t=(1+\delta)\E[X]$. If $\Phi_X(\delta+\varepsilon)=\infty$ then the claim is vacuously true.
Otherwise, there exists a subcube $F$ such that $-\log\mathbb{P}(Y\in F)=\Phi_X(\delta+\varepsilon)$ and $\E_F[X]\geq (1+\delta+\varepsilon)\E[X]=t+\varepsilon \E[X]$. By $X\leq M$ we have $\E_F[X]\leq M\cdot  \mathbb{P}(X\geq t \mid Y\in F)+t$. Thus, $\mathbb{P}(X\geq t \mid Y\in F)\geq\frac{\varepsilon \E[X]}{M}$ and hence,
\[
\mathbb{P}(X\geq t)\geq \mathbb{P}(X\geq t\mid Y\in F)\mathbb{P}(Y\in F )\geq\frac{\varepsilon \E[X]}{M}\mathbb{P}(Y\in F ).
\]
Now taking the negative logarithm gives the assertion of the lemma.
\end{proof}

\begin{fact}\label{claim_intersection_of_subcubes}
Suppose $F_1,\ldots,F_k$ are subcubes of $\{0,1\}^N$. If $F_1\cap\cdots\cap F_k$ is nonempty, then it is also a subcube of $\{0,1\}^N$ and, moreover, $\codim(F_1\cap\cdots\cap F_k)\leq \sum_{i=1}^k \codim(F_i)$.
\end{fact}

\begin{proof}
We prove the statement for $k=2$; the case $k>2$ follows by a simple inductive argument. Let $I_1,I_2$ and $x_1,x_2$ be such that $F_i$ is a subcube centered at $(I_i,x_i)$ for $i=1,2$. Define $x\in \{0,1\}^N$ in the following way: for all $i\in I_1$ and $j\in I_2$ put $(x)_i=(x_1)_i$ and $(x)_j=(x_2)_j$ and for all $i\not\in I_1\cup I_2$ put $(x)_i=0$.
This is indeed a well defined element in $\{0,1\}^N$ as for all $i\in I_1\cap I_2$ we have, $(x_1)_i=(x_2)_i$, otherwise it will contradict our assumption that $F_1\cap F_2\neq \emptyset$. It follows from the definition of $x$ that $F_1\cap F_2=F(I_1\cup I_2,x)$ and hence $F_1\cap F_2$ is a subcube. Furthermore, we have
\[
  \codim(F_1\cap F_2)=|I_1\cup I_2|\leq |I_1|+|I_2|=\codim(F_1)+\codim(F_2).\qedhere
\]
\end{proof}

\begin{lemma}\label{lemma_second}
Let $Y$ be a random variable taking values in $\{0,1\}^N$ and let $X=X(Y)$ be a nonnegative real-valued function with complexity bounded by $d$.
Then for every positive integer $\ell$ and all positive real numbers $\varepsilon$ and $\delta$ with $\varepsilon<1+\delta$,
$$\mathbb{P}(X\geq (1+\delta)\E [X]\text{ and }Y \not \in F\text{ for all }F\in \mathcal{F})\leq {\bigg(\frac{1+\delta-\varepsilon}{1+\delta}\bigg)}^\ell,$$
where $$\mathcal{F}=\{F\in \{0,1\}^N:F\text{ is a subcube and }\codim(F)\leq d\ell\text{ and }\E_F[X]\geq(1+\delta-\varepsilon)\E[X]\}.$$
\end{lemma}

\begin{proof}
Given $S\subseteq \{0,1\}^N$ a subcube, let $Z_S$ be the indicator random variable of the event that $Y\not \in F$ for all $F\in \mathcal{F}$ with $F\supseteq S$. Note that $S\subseteq S'$ implies $Z_S\leq Z_{S'}$ and let $Z=Z_{\emptyset}$. Since $XZ\geq 0$ and $Z^\ell=Z$, Markov's inequality gives 
\begin{equation}\label{eq_Markov}
    \mathbb{P}(X\geq (1+\delta)\E[X]\text{ and } Z=1) = \mathbb{P}(XZ\geq (1+\delta)\E[X])\leq \frac{\E[X^\ell Z]}{((1+\delta)\E[X])^\ell}.
\end{equation}
To simplify the notation, for every $S\subseteq \{0,1\}^N$ we write  $\mathbbm{1}_S$ for $\mathbbm{1}_{\{Y\in S\}}$.
Write $X=\sum_{F}\alpha _F\mathbbm {1}_{F}$, where the sum ranges over all subcubes of $\{0,1\}^N$, each coefficient $\alpha_F$ is nonnegative, and $\alpha_F=0$ for all $F$ with $codim(F)>d$. Then for every $k\in [\ell]$, 
\begin{align*}
    \E[X^kZ]=&\sum\limits_{F_1,\ldots,F_k}\alpha_{F_1}\cdots\alpha_{F_k}\E[\mathbbm {1}_{F_1}\cdots\mathbbm {1}_{F_k}\cdot Z]\\
    \leq&\sum\limits_{F_1,\ldots,F_k}\alpha_{F_1}\cdots\alpha_{F_k}\E[\mathbbm {1}_{F_1}\cdots\mathbbm {1}_{F_k}\cdot Z_{F_1\cap \cdots \cap F_k}]\\
    \leq&\sum\limits_{F_1,\ldots,F_{k-1}}\alpha_{F_1}\cdots\alpha_{F_{k-1}}\E[\mathbbm {1}_{F_1}\cdots\mathbbm {1}_{F_{k-1}}\cdot Z_{F_1\cap \cdots \cap F_{k-1}}]\\
    &\qquad\qquad\qquad\qquad\; \;\cdot\E[X\mid \mathbbm {1}_{F_1}\cdots\mathbbm {1}_{F_{k-1}}\cdot Z_{F_1\cap \cdots \cap F_{k-1}}=1],
\end{align*}
where we may let the third sum range only over sequences $F_1,\ldots, F_{k-1}$ for which the event $\{\mathbbm {1}_{F_1}\cdots\mathbbm {1}_{F_{k-1}}\cdot Z_{F_1\cap \cdots \cap F_{k-1}}=1\}$ has a positive probability of occurring.
\begin{claim}
    For any such sequence, 
    \[F_1\cap \ldots \cap F_{k-1}\not \in \mathcal{F}\; \text{ and } \; \mathbbm {1}_{F_1}\cdots\mathbbm {1}_{F_{k-1}}\cdot Z_{F_1\cap \cdots \cap F_{k-1}}=\mathbbm {1}_{F_1}\cdots\mathbbm {1}_{F_{k-1}}.\]
\end{claim}

\begin{proof}
To see this, note that if 
$\{\mathbbm {1}_{F_1}\cdots\mathbbm {1}_{F_{k-1}}\cdot Z_{F_1\cap \cdots \cap F_{k-1}}=1\}$ has positive probability of occurring, then there exists a $y$ such that $y\in F_1\cap \ldots \cap F_{k-1}$ and $Z_{F_1\cap \ldots \cap F_{k-1}}(y)=1$. That means $y\in F_1\cap \ldots \cap F_{k-1}$ and $y\not \in F$ for any subcube $F\in \mathcal{F}$ such that $F\supseteq F_1\cap \ldots \cap F_{k-1}$. Hence, $F_1\cap \ldots \cap F_{k-1}\not \in \mathcal{F}$. For the second part assume towards a contradiction that $\mathbbm {1}_{F_1}\cdots\mathbbm {1}_{F_{k-1}}\cdot Z_{F_1\cap \cdots \cap F_{k-1}} < \mathbbm {1}_{F_1}\cdots\mathbbm {1}_{F_{k-1}}$, meaning that there is $y'\in F_1\cap \ldots \cap F_{k-1}$ such that $Z_{ F_1\cap \ldots \cap F_{k-1}}(y')=0$. Therefore, there exists a subcube $F^*\in \mathcal {F}$ such that, $F^*\supseteq  F_1\cap \ldots \cap F_{k-1}$. This is a contradiction as now, $y\in  F_1\cap \ldots \cap F_{k-1}\subseteq F^*$ but by our assumption, $y\not \in F^*$ as $F^*\in \mathcal{F}$.
\end{proof}

Since $\{\mathbbm{1}_{F_1\cap\cdots\cap F_{k-1}}=1\}$ has a positive probability of occurring, Fact \ref{claim_intersection_of_subcubes} asserts that $F_1\cap\cdots\cap F_{k-1}$ is a subcube, and $\codim({F_1\cap \cdots \cap F_{k-1}})\leq d(k-1)\leq d\ell$. Therefore
\[
\E[X\mid \mathbbm {1}_{F_1}\cdots\mathbbm {1}_{F_{k-1}}\cdot Z_{F_1\cap\cdots\cap F_{k-1}}=1]=\E_{F_1\cap \cdots \cap F_{k-1}}[X]<(1+\delta-\varepsilon)\E[X],
\]
as otherwise $F_1\cap \cdots \cap F_{k-1}$ would belong to $\mathcal{F}$. It follows that
\begin{align*}
    \sum\limits_{F_1,\ldots,F_k}\alpha_{F_1}\cdots\alpha_{F_k}\E[\mathbbm {1}_{F_1}\cdots\mathbbm {1}_{F_k}\cdot Z_{F_1\cap \cdots \cap F_k}]&\\
    <(1+\delta-\varepsilon)\E[X]\sum\limits_{F_1,\ldots,F_{k-1}}&\alpha_{F_1}\cdots\alpha_{F_{k-1}}\E[\mathbbm {1}_{F_1}\cdots\mathbbm {1}_{F_{k-1}}\cdot Z_{F_1\cap \cdots \cap F_{k-1}}].
\end{align*}
By induction, we see that $\E[X^\ell Z]<((1+\delta-\varepsilon)\E[X])^\ell$. Substituting this inequality into \eqref{eq_Markov} completes the proof.
\end{proof}

Now we use the previous lemmas to prove the theorem.

\begin{proof}[Proof of Theorem \ref{thm:9.1}]
Let $K$ be a large constant that may depend on $\varepsilon,\delta$ and $d$. Furthermore, let $t=(1+\delta)\E[X]$ and $\ell =\left\lfloor{K/d\cdot \Phi_X(\delta+\varepsilon)}\right\rfloor$. Define
\begin{equation*}
    \mathcal{F}=\{F\subseteq\{0,1\}^N:F\text{ is a subcube with }\codim(F)\leq d\ell \text{ and }\E_F[X]\geq (1+\delta-\varepsilon)\E[X]\}.
\end{equation*}
It follows from Lemma \ref{lemma_second} that 
\begin{equation*}
    \mathbb{P}(X\geq t\text{ and }Y\not \in F\text{ for all }F\in\mathcal{F})\leq \bigg(1-\frac{\varepsilon}{1+\delta}\bigg)^\ell.
\end{equation*}
As $\comp(X)\leq d$, we can write
\[
X=\sum\limits_{j\in J}\alpha_j\mathbbm {1}_{\{Y \in F_j\}}
\]
for some $J$ so that, for all $j\in J$, the coefficient $\alpha_j$ is nonnegative and $F_j$ is a subcube with $\codim(F_j)\leq d$.
Put $M=\sum \alpha _i$ and note that $X\leq M$ always. Applying Lemma \ref{lemma_upper} gives
\begin{align*}
    -\log\mathbb{P}(X\geq t)\leq \Phi_X(\delta+\varepsilon)+\log\bigg (\frac{M}{\varepsilon\E[X]}\bigg).
\end{align*}
Note also that $\E[X]\geq M p^d$ as $\comp(X)\leq d$ and therefore $\mathbb{P} (Y\in F_j)\geq \min\{p,1-p\}^d=p^d$ for every $j\in J$. Therefore,
\begin{align*}
    -\log\mathbb{P}(X\geq t)\leq \Phi_X(\delta+\varepsilon)+\log\bigg (\frac{1}{\varepsilon p^d}\bigg).
\end{align*}
Thus, provided $K$ is sufficiently large we also have
\begin{align*}
    -\log\mathbb{P}(X\geq t)\leq (1+\varepsilon)\Phi_X(\delta+\varepsilon),
\end{align*}
as we assumed that $\Phi _X(\delta +\varepsilon)\geq \Phi _X(\delta -\varepsilon)\geq K\log (1/p)$. 
Putting all of this together gives that for sufficiently large $K$, we have 
\begin{align*}\label{eq_no_seeds}
    \mathbb{P}(X\geq t\text{ and }Y\not\in F\text{ for all }F\in\mathcal{F})&\leq \bigg(1-\frac{\varepsilon}{1+\delta}\bigg)^\ell\leq e^{-\frac{\varepsilon}{1+\delta}\left\lfloor{K/d\cdot \Phi_X(\delta+\varepsilon)}\right\rfloor}\\
    & \leq (\varepsilon/2)\cdot \mathbb{P}(X\geq t).
\end{align*}
The assertion of the theorem now follows: 
\begin{equation*}
    \mathbb{P}(X\geq t)\leq{(1-\varepsilon/2)}^{-1}\cdot \mathbb{P}(Y\in F \text{ for some }F\in \mathcal{F})\leq (1+\varepsilon )\mathbb{P}(Y\in F \text{ for some }F\in \mathcal{F}) 
\end{equation*}
where the last inequality holds for $\varepsilon < 1/2$.
\end{proof}

When $p\ll 1$ and $\comp(X)=O(1)$, the majority of the contribution to both $\E_F[X]$ and $\mathbb{P}(Y\in F)$ comes from the one-supcube $F^{(1)}\supseteq F$. We prove a straightforward corollary of Theorem \ref{thm:9.1} which will be more convenient to work with. We start by proving the following lemma.

\begin{lemma}\label{lemma_third}
The following holds for every positive integer $d$ and all positive real numbers $\varepsilon$ and $\delta$ with $\varepsilon<1$. Let $Y$ be a sequence of $N$ independent $\Ber(p)$ random variables for some $p< 1-\big({\frac{1+\delta-\varepsilon}{1+\delta-\varepsilon/2}\big)^{1/d}}$ and assume that $X=X(Y)$ has complexity at most $d$. Let $F$ be a subcube satisfying $\mathbb{E}_F[X]\geq (1+\delta -\varepsilon/2)\E [X]$. Then, $\E_{F^{(1)}}[X]\geq (1+\delta-\varepsilon)\E[X]$ where $F^{(1)}$ is the one-supcube of $F$.
\end{lemma}

\begin{proof}
As $X$ has complexity $d$ we can write 
\[
X=\sum\limits_{j\in J}\alpha_j\mathbbm {1}_{\{Y \in F_j\}}
\] 
for some $J$ so that, for all $j\in J$, the coefficient $\alpha_j$ is nonnegative and $F_j$ is a subcube with $\codim(F_j)\leq d$. Moreover, by our assumptions on $F_j$, we have $\codim_0(F_j)\leq \codim(F_j)\leq d$. Thus, we have the following inequality,
\begin{equation*}
    \mathbb{P}(Y\in F_j\mid Y\in F^{(1)})\geq (1-p)^d \cdot \mathbb{P}(Y\in F_j\mid Y\in F).
\end{equation*}
It follows that 
\begin{align}
    \E_{F^{(1)}}[X]&=\sum\limits_{j\in J}\alpha_j\E_{F^{(1)}}[\mathbbm {1}_{\{Y\in F_j\}}]\geq \sum\limits_{j\in J}\alpha_j(1-p)^d\cdot\mathbb{E}_{F}[\mathbbm {1}_{\{Y\in F_j\}}]\nonumber\\
    &\geq \sum\limits_{j\in J}\alpha_j\Big(\frac{1+\delta-\varepsilon}{1+\delta-\varepsilon/2}\Big)\mathbb{E}_{F}[\mathbbm {1}_{F_j}]=\Big(\frac{1+\delta-\varepsilon}{1+\delta-\varepsilon/2}\Big)\E_F[X]\geq (1+\delta-\varepsilon)\E[X]\label{lemma_line_2},
\end{align}
where the first inequality in \eqref{lemma_line_2} follows from the assumption on $p$ and the second inequality in \eqref{lemma_line_2} follows from the assumption on $\E_F[X]$.
\end{proof}

Using this lemma we now state and prove a straightforward corollary of Theorem \ref{thm:9.1}.

\begin{cor}\label{Main_tool}
For every positive integer $d$ and all positive real numbers $\varepsilon$ and $\delta$ with $\varepsilon <{1}$, there is a positive constant $K=K(\varepsilon,\delta,d)$ such that the following holds. Let $Y$ be a sequence of $N$ independent $\Ber(p)$ random variables for some $p<\min \{1-\big({\frac{1+\delta-\varepsilon}{1+\delta-\varepsilon/2}\big)^{1/d}}, 1/2\}$ and assume that $X=X(Y)$ has complexity at most $d$ and satisfies $\Phi_X(\delta -\varepsilon)\geq K\log(\frac{1}{p})$. Denote by $\mathcal{F}_1$ the collection of all subcubes $F\subseteq \{0,1\}^N$ satisfying
\begin{enumerate}[label=(\subscript{H}{{\arabic*}})]
    \item \label{thm_2_cond_1} $\mathbb{E}_F[X]\geq (1+\delta -\varepsilon)\E [X]$,
    \item \label{thm_2_cond_2} $\codim F\leq K\cdot\Phi _X(\delta+\varepsilon)$,
    \item\label{thm_2_cond_4} $\codim_1(F)=\codim(F)$.
\end{enumerate}
Then
\begin{equation*}
    \mathbb{P}(X\geq (1+\delta)\E[X])\leq (1+\varepsilon)\mathbb{P}(Y\in F \text{ for some }F\in \mathcal{F}_1).
\end{equation*}
\end{cor}

\begin{proof}
Applying Theorem \ref{thm:9.1} with $\varepsilon$ replaced by $\varepsilon/2$ we obtain 
\begin{equation}\label{eq_main_cor_left}
    \mathbb{P}(X\geq (1+\delta)\E[X])\leq (1+\varepsilon) \mathbb{P}(Y\in F \text{ for some } F\in \mathcal{F}),
\end{equation}
where $\mathcal{F}$ is the collection of all subcubes $F\subseteq \{0,1\}^N$ satisfying \ref{cond_1} and \ref{cond_2} where $\varepsilon$ is replaced with $\varepsilon/2$. Letting $\mathcal{F}_{1}=\{F^{(1)}:F\in \mathcal{F}\}$, we obtain \ref{thm_2_cond_2} and \ref{thm_2_cond_4} for every subcube $F\in \mathcal{F}_1$ due to \ref{cond_2} and the definition of one-supcubes. Noting that for every subcube $F$ we have $F\subseteq F^{(1)}$ we obtain also that \begin{equation}\label{eq_main_cor_right}
    \mathbb{P}(Y\in F \text{ for some } F\in \mathcal{F})\leq \mathbb{P}(Y\in F \text{ for some } F\in \mathcal{F}_1).
\end{equation}
Combining \eqref{eq_main_cor_left} and \eqref{eq_main_cor_right} we obtain that
\begin{equation*}
    \mathbb{P}(X\geq (1+\delta)\E[X])\leq (1+\varepsilon) \mathbb{P}(Y\in F \text{ for some } F\in \mathcal{F}_1).
\end{equation*}
Furthermore, Lemma \ref{lemma_third} implies that $\mathcal{F}_1$ also satisfies \ref{thm_2_cond_1}.
\end{proof}

We call a subcube $F$ a \emph{seed} if $F$ is an element of $\mathcal{F}_1$ defined in Corollary \ref{Main_tool}. A formal definition is given in the following section.
A general way one would use Corollary \ref{Main_tool} to bound upper tails of counts of induced subgraphs is to define a special type of seeds called \emph{structured seeds}. These structured seeds are subcubes in $\mathcal{F}_1$ from Corollary \ref{Main_tool}, with a stronger condition than \ref{thm_2_cond_1}. This condition is a combinatorial condition involving various supcubes of the subcubes in $\mathcal{F}_1$ (this will be explained further in Section \ref{sec_3}). Then one would define a \emph{core subcube} to be a subcube containing a structured seed such that `every coordinate counts'. In the following section we define these special subcubes. We will also show that Corollary \ref{Main_tool} can be modified to bound the upper tail probability via the probability of $Y$ being an element in a core subcube.

\section{From seeds to modified cores}\label{sec_3}

Note that for any subcube $F\subseteq \{0,1\}^{E(K_n)}$, $F^{(1)}$ the one-supcube of $F$ corresponds to all subgraphs of $K_n$ containing a specific subgraph of $K_n$ with $\codim_1({F})$ edges. Therefore, from this point onwards we think of one-supcubes as a family of subgraphs of $K_n$ containing a specific subgraph. In particular, instead of writing $\E_F[X]$ for some one-supcube (of some subcube) we will write $\E_G[X]$ for the graph corresponding to $F$. The subgraphs corresponding to the members of $\mathcal{F}_1$ from Corollary \ref{Main_tool} will be called seeds.
We start with a definition which will be useful in this section as well as the next ones.

\begin{definition} \label{def_N(n,m,H)}
Suppose $H$ and $G$ are graphs and let $e$ be an edge of $G$. We define:
\begin{itemize}
  \item
  $N_{ind}(H,G)$ is the number of induced copies of $H$ in $G$.
  \item
  $N_{ind}(e,H,G)$ is the number of induced copies of $H$ in $G$ that contain $e$.
\end{itemize}
\end{definition}

Our main concern in this paper is the random variable counting the number of induced copies of $C_4$ in $G_{n,p}$. Therefore, we let $X=N_{ind}(C_4,G_{n,p})$. We will use this notation from this point onward. We now continue by defining seed graphs.

\begin{definition}
Let $\varepsilon,\delta,K$ be positive reals. In addition let $p\in (0,1)$. Then we define $\mathcal{S}(\varepsilon,\delta,K)$ to be the collection of all spanning subgraphs $G\subseteq K_n$ satisfying:
\begin{enumerate}[label=(\subscript{S}{{\arabic*}})]
    \item \label{cond_2_seeds}$e(G)\leq K\cdot\Phi _X(\delta+\varepsilon)$ and
    \item \label{cond_1_seeds}$\mathbb{E}_G[X]\geq (1+\delta -\varepsilon)\E [X]$.
\end{enumerate}
We call the graphs in $\mathcal{S}(\varepsilon,\delta,K)$ seeds. Furthermore, for every positive integer $m$ we define $\mathcal{S}_m(\varepsilon,\delta,K)$ to be the set of all seeds with $m$ edges.
\end{definition}

Since $\E_G[X]$ is determined by the number of induced copies of various subgraphs of $C_4$ in $G$, it will be convenient to restate \ref{cond_1_seeds} in the above definition with a graph theoretic condition, giving rise to the notion of \emph{structured seeds}.
In the definition we use the graph obtained from $K_{1,2}$ by adding an isolated vertex; we denote this graph by $K_{1,2}\sqcup K_{1}$. Now we define the structured seeds.

\begin{definition}
Let $\varepsilon,\delta,K$ be positive reals. In addition let $p\in (0,1)$. Then we define $\mathcal{S}'(\varepsilon,\delta,K)$ to be the collection of all spanning subgraphs $G\subseteq K_n$ satisfying:
\begin{enumerate}[label=(\subscript{S'}{{\arabic*}})]
    \item $e(G)\leq K\cdot\Phi _X(\delta+\varepsilon)$ and
    \item $N_{ind}(C_4,G)+N_{ind}(K_{1,2}\sqcup K_{1},G)p^2\geq (\delta-\varepsilon)\E[X]$.
\end{enumerate}
We call the graphs in $\mathcal{S}'(\varepsilon,\delta,K)$ structured seeds. Furthermore, for every positive integer $m$ we define $\mathcal{S}'_m(\varepsilon,\delta,K)$ to be the set of all structured seeds with $m$ edges.
\end{definition}

The following is a lemma relating the seeds and the structured seeds, which we prove later.

\begin{lemma}\label{claim_F_becomes_G}
Let $\varepsilon,\delta,K$ be positive reals and suppose also that $\varepsilon\leq \delta/2$ and $p\ll 1$. Then, there exists a positive constant $C=C(\varepsilon,\delta,K)$ such that for any $m\leq K\Phi_X(\delta-\varepsilon)$ and large enough $n$ we have,
\[
\mathcal{S}_m(\varepsilon,\delta,C)\subseteq \mathcal{S}'_m(2\varepsilon,\delta,C).
\]
\end{lemma}

Now we are ready to define what a core graph is. This is given in the following definition.

\begin{definition}
Let $\varepsilon,\delta,K$ be positive reals. In addition let $p\in (0,1)$. Then we define $\mathcal{C}(\varepsilon,\delta,K)$ to be collection of all structured seeds $G\in \mathcal{S}'(\varepsilon,\delta,K)$ satisfying the following:\\
For all $e\in E(G)$
    
\[
    N_{ind}(e,C_4,G)+N_{ind}(e,K_{1,2}\sqcup K_{1},G)p^2\geq \varepsilon \E[X]/(2K\cdot \Phi _X(\delta+\varepsilon)).
\]
We call the graphs in $\mathcal{C}(\varepsilon,\delta,K)$ core graphs. Furthermore, for every positive integer $m$ we define $\mathcal{C}_m(\varepsilon,\delta,K)$ to be the set of all cores with $m$ edges.
\end{definition}
In cases where the parameters $\varepsilon,\delta$ and $K$ can be understood from the context we omit them.
The main aim of this section is to prove the following theorem which we derive using Corollary~\ref{Main_tool} and several lemmas.

\begin{thm}\label{thm:main_cores_1}
For all positive real numbers $\varepsilon,\delta,p$ with $\varepsilon <{\delta}$ and $p<\min \{1-\big({\frac{1+\delta-\varepsilon}{1+\delta-\varepsilon/2}\big)^{1/6}}, 1/2\}$, there is a positive constant $K=K(\varepsilon,\delta)$ such that following holds:
\begin{equation}
    \mathbb{P}(X\geq (1+\delta)\E[X])\leq (1+\varepsilon)\mathbb{P}(G \subseteq G_{n,p} \text{ for some }G\in \mathcal{C}(\varepsilon,\delta,K)).
\end{equation}
In particular,
\[
    \mathbb{P}(X\geq (1+\delta)\E[X])\leq (1+\varepsilon)\sum \limits_m |\mathcal{C}_m(\varepsilon,\delta,K)|p^m.
\]
\end{thm}

We start by proving Lemma \ref{claim_F_becomes_G} relating seeds and structured seeds.

\begin{proof}[Proof of Lemma \ref{claim_F_becomes_G}]
Suppose $G\in \mathcal{S}_m(\varepsilon,\delta,K)$.
By \ref{cond_1_seeds} we have
\begin{align*}
   (\delta-\varepsilon)\E[X] \leq \E_G[X]-\E[X] \leq \sum\limits_{\emptyset \neq H{\subseteq}^* C_4}N_{ind}(H,G)p^{4-e(H)},
\end{align*}
where ${\subseteq}^*$ stands for spanning subgraphs.
Observe that 
\[
    N_{ind}(H,G)\leq 
        \begin{cases}
            mn^2 &\text{ if } H=K_2\sqcup 2K_1 \\
            m^2&\text{ if } H=M_2\text{ or }P_4,
        \end{cases}
\]
where $K_2\sqcup2K_1$ is the graph on four vertices and one edge, $M_2$ is a matching of size two, and $P_4$ is the path with 3 edges; this indeed holds as two edges of $G$ span at most one induced $M_2$ or $P_4$. Therefore, we obtain
\begin{align}\label{eq_very_long_1_1}
    (\delta-\varepsilon)\E[X]\leq& N_{ind}(C_4,G)+N_{ind}(K_{1,2}\sqcup K_{1},G)p^2 +\underbrace{m^2p}_{N_{ind}(P_4,G)}+\underbrace{m^2p^2}_{N_{ind}(M_2,G)}+\underbrace{mn^2p^3}_{N_{ind}(K_2\sqcup 2K_1,G)} \nonumber\\
    \leq& N_{ind}(C_4,G)+N_{ind}(K_{1,2}\sqcup K_{1},G)p^2+2pm^2+mn^2p^3.
\end{align}
In Section \ref{sec_5} (Claim~\ref{claim_evalutation_of_Phi}) we prove that $\Phi _X (\delta)=O_{\delta}(\sqrt{\E[X]}\log(1/p))$. Therefore, as $\E[X]=\Theta (n^4p^4)$, we have $m = O_{\varepsilon,\delta}(n^2p^2\log(1/p))$. Further, as $p\ll 1$ we get from \eqref{eq_very_long_1_1} the following inequality for large enough $n$,
\begin{align*}
    (\delta-\varepsilon)\E[X]&\leq N_{ind}(C_4,G)+N_{ind}(K_{1,2}\sqcup K_{1},G)p^2+O(p\log^2 (1/p)\E[X])\\
                            &\leq N_{ind}(C_4,G)+N_{ind}(K_{1,2}\sqcup K_{1},G)p^2+\varepsilon\E[X].
\end{align*}
Therefore, we obtain $N_{ind}(C_4,G)+N_{ind}(K_{1,2} \sqcup K_{1},G)p^2\geq (\delta-2\varepsilon)\E[X]$ for large enough $n$. This is the assertion of the lemma.
\end{proof}
Informally, the next claim is that every structured seed contains a core. Formally this is given in the following claim.
\begin{lemma}\label{lemma_forth}
Suppose $\varepsilon,\delta,K$ are positive reals. Then for every structured seed $G\in \mathcal{S}'$ and every nonnegative real $s$, there exists a subgraph $G^*\subseteq G$ such that:
\begin{enumerate}[label=(\subscript{C}{{\arabic*}})]
    \item\label{Not a lot of edges in a core} $e(G)\leq K\cdot\Phi _X(\delta+\varepsilon)$,
    \item\label{N(G) larger} $N_{ind}(C_4,G^*)+N_{ind}(K_{1,2}\sqcup K_{1},G^*)p^2\geq (\delta-\varepsilon)\E[X]-s$, and
    \item \label{lemma_forth_cond_2}$N_{ind}(e,C_4,G^*)+N_{ind}(e,K_{1,2}\sqcup K_{1},G^*)p^2\geq \frac{s}{K\Phi _X(\delta+\varepsilon)}$ for every edge $e\in E(G^*)$.
\end{enumerate}
\end{lemma}

\begin{proof}
In this proof, it would be convenient to let 
\[
    N(G)=N_{ind}(C_4,G)+N_{ind}(K_{1,2}\sqcup K_{1},G)p^2.
\]
This is because then, \ref{N(G) larger} is equivalent to $N(G^*)\geq (\delta-\varepsilon)\E[X]-s$ and \ref{lemma_forth_cond_2} is equivalent to $N(G^*)-N(G^*\setminus e)\geq \frac{s}{K\Phi _X(\delta+\varepsilon)}$ for every $e$ an edge in $G^*$. 

Define the sequences $G=G_0\supseteq G_1\supseteq \cdots \supseteq G_r=G^*$ and $e_1,e_2,\ldots,e_r\in G$ by repeatedly setting $G_{k+1}$ to be a subgraph of $G_k$ obtained by the deletion of an edge $e_k$ such that
\[
N(G_k)-N(G_{k+1})<\frac{s}{e(G)},
\]
as long as such edge exists. The graph $G^*$ satisfies \ref{Not a lot of edges in a core} as it is a subgraph of a structured seed. We also claim that, the subgraph $G^*$ satisfies \ref{lemma_forth_cond_2}. That is because, if there is an edge $e$ in $G^*$ with $N(G^*)-N(G^*\setminus e)<\frac{s}{e(G)}$ the process would have continued by deleting this edge $-$ a contradiction. Recalling that $G$ is a structured seed we find that, $e(G)\leq K\Phi_X(\delta+\varepsilon)$ and thus,
\[
    N(G^*)\geq \frac{s}{e(G)}\geq \frac{s}{K\Phi_X(\delta+\varepsilon)}. 
\]
Finally, since $r\leq e(G)$, we have
\begin{align*}
    N(G)-N(G^*)=\sum \limits_{k=0}^{r-1}N(G_k)-N(G_{k+1})\leq \frac{rs}{e(G)}\leq s.   
\end{align*}
Rearranging this inequality and recalling that $G$ is a structured seed we obtain the assertion of the lemma,
\[
    N(G^*)=N_{ind}(C_4,G^*)+N_{ind}(K_{1,2}\sqcup K_{1},G^*)p^2\geq (\delta-\varepsilon)\E[X]-s. \qedhere
\]
\end{proof}

Applying Lemma \ref{lemma_forth} invoked with $\varepsilon$ replaced by $\varepsilon/2$ and $s=\varepsilon \E[X]/2$ yields the following corollary.

\begin{cor}\label{cor:seeds_contains_cores}
Suppose $\varepsilon,\delta,K$ are positive reals. Suppose further that $G\in \mathcal{S}'(\varepsilon/2,\delta,K)$ is a structured seed. Then, there exists a core $G^* \in \mathcal{C}(\varepsilon,\delta,K)$ such that $G^*\subseteq G$.
\end{cor}

Now we derive Theorem \ref{thm:main_cores_1} from the above claims and Corollary \ref{Main_tool}.

\begin{proof}[Proof of Theorem \ref{thm:main_cores_1}]
Applying Lemma \ref{claim_F_becomes_G} with $\varepsilon$ replaced by $\varepsilon/4$ and Corollary \ref{cor:seeds_contains_cores} with $\varepsilon$ replaced by $\varepsilon/2$ we find that 
\begin{align*}
    \mathbb{P}(G \subseteq G_{n,p} \text{ for some }G\in \mathcal{S}(\varepsilon/4,\delta,K)) &\leq \mathbb{P}(G \subseteq G_{n,p} \text{ for some }G\in \mathcal{S}'(\varepsilon/2,\delta,K))\\
    &\leq \mathbb{P}(G \subseteq G_{n,p} \text{ for some }G\in \mathcal{C}(\varepsilon,\delta,K)).
\end{align*}
Applying Corollary \ref{Main_tool} with $\varepsilon$ replaced with $\varepsilon/2$ we obtain
\[
    \mathbb{P}(X\geq (1+\delta)\E[X])\leq (1+\varepsilon)\mathbb{P}(G \subseteq G_{n,p} \text{ for some }G\in \mathcal{S}(\varepsilon/4,\delta,K)).
\]
Combining the above inequalities we obtain the assertion of the theorem, that is 
\[
    \mathbb{P}(X\geq (1+\delta)\E[X])\leq (1+\varepsilon)\mathbb{P}(G \subseteq G_{n,p} \text{ for some }G\in \mathcal{C}(\varepsilon,\delta,K)).\qedhere
\]
\end{proof}

\section{Lower bounds}\label{sec_4}

The aim of this section is to give lower bounds for $\mathbb{P}(X\geq (1+\delta)\E [X])$ for every positive $\delta$. We do so by presenting a family of graphs such that `planting' them in $G_{n,p}$ increases the expectation of $X$ by a multiplicative factor of $1+\delta$.
More formally, when we say `planting' we mean changing the probability measure on the hypercube $\{0,1\}^{\binom{n}{2}}$ from the usual product measure of $G_{n,p}$ to the probability measure of $G_{n,p}$ conditioned on the existence of a subgraph from a predetermined family of graphs. The graphs in the families that we will consider will satisfy $\E_G[X]\geq (1+\delta)\E[X]$. We choose these families as such because, on the event that $G\subseteq G_{n,p}$ for $G\subseteq K_n$ such that $\E_G[X]\geq (1+\delta)\E[X]$ the probability of $X\geq (1+\delta)\E[X]$ is pretty `large'.

Since for every labeled graph $G\subseteq K_n$ the probability of $G_{n,p}$ containing $G$ is $p^{e(G)}$ it makes sense to consider such graphs with the smallest $e(G)$. In contrast, in our case there is a wide range of $p$ where the strongest lower bound is obtained by planting \emph{some} member of a large family of graphs; each individual graph in the family is sub-optimal in terms of the number of edges, but the size of the family compensates for the difference in the number of edges between the optimal graph and the graphs in our family. This family is, roughly speaking all embeddings of an unbalanced complete bipartite graph into $K_n$.





More precisely, denoting the complete bipartite graph with sides of size $s$ and $t$ by $K_{s,t}$, we will define integers $m_0,m_2,m_3,\ldots $ and $m_*$ such that $k|m_k$ for $k\neq 0$ and $n|m_*$ and $m_k\approx 2\sqrt{k/(k-1)\delta \E[X]}$ for $k\neq 0$, $m_0\approx 2\sqrt{\delta \E[X]}$, and $m_*\approx (\sqrt{1+2\delta}-1)\sqrt{2\E[X]}$.

Note that provided that $n^{-1} \ll p\ll n^{-1/2}$ the constructions $K_{k,m_k/k}$ and $K_{\sqrt{m_0},\sqrt{m_0}}$ contain $\delta \E[X]$ induced copies of $C_4$, up to lower order terms. In addition, we will show later that $H=K_{2m_*/n,n/2}$ admits $\E_{H}[X]\geq (1+\delta) \E[X]$.

Denote by $\mathcal{E}_{k}$ the set of all copies of $K_{k,m_k/k}$ in $K_n$ when $k\neq 0,1$. Denote by $\mathcal{E}_{0}$ the set of all $K_{\sqrt{m_0},\sqrt{m_0}}$ in $K_n$ and and by $\mathcal{E}_*$ denote the set of all copies of $H$ in $K_n$. Planting one of $\mathcal{E}_k,\mathcal{E}_0$ or $\mathcal{E}_*$ yields a lower bound for the probability of the upper tail event which is valid for all values of $p$. As was mentioned in the introduction the significant different between this work and previous ones is the need to plant a large family of sub-optimal graphs and not a single optimal graph. This is true only for the families $\mathcal{E}_k$ where $k\neq 0$. In the case of $\mathcal{E}_0,\mathcal{E}_*$ we could as well plant a single graph from these sets as $|\mathcal{E}_0|$ and $|\mathcal{E}_*|$ are negligible.

One can compare these bounds, and see that the best one depends on $p$ in the following way. There exists an increasing sequence $\{c_k\}_{k=1}^\infty$ with $c_1=0$ and $\lim_{k\rightarrow \infty}c_k=1/3$ so that,
provided $n^{-1+c_{k-1}}\ll p\ll n^{-1+c_{k}}$ for some integer $k\geq 2$ we obtain the strongest lower bound on the upper tail probability by planting $\mathcal{E}_k$. The best family to plant when $n^{-2/3}\ll p \ll n^{-1/2}$ is $\mathcal{E}_0$ which should be thought of as the `limit' of $\mathcal{E}_k$ when $k$ goes to infinity. Lastly, $\mathcal{E}^*$ is the best family to plant when $n^{-1/2} \ll p\ll 1$.

The main result of this section is a formalisation of the above discussion. In order to make things rigorous from now on we fix $\delta,\varepsilon$ to be positive reals and let
\[
r_k=\Bigg\{
\begin{array}{@{}l@{\thinspace}l}
        2 &\text{ for } k=0, \\
        2\sqrt{k/(k-1)} &\text{ for }k\geq 2,\\
\end{array}
\]
when $k$ is some non negative integer. Moreover, note the following:
For all $n^{-1}\ll p\ll 1$ we have $\E[X]=\Omega(n^4p^4)$. Therefore, for clarity of the presentation we assume that $r_k\sqrt{(\delta+\varepsilon)\E[X]}$ is an integer divisible by $k$ for any $2\leq k\leq O(np)$. Furthermore, if $p\gg n^{-1/2}$ we have $\sqrt{\E[X]}/n =\Omega(np^2)$, hence for clarity of the presentation we assume $\sqrt{C(\varepsilon,\delta)\E[X]}/n$ is an integer where $C(\varepsilon,\delta)$ will be specified later.
The lower bounds given by the following theorem should be thought of as the probability of the appearance of $K_{k,m_k/k}$ in $G_{n,p}$ for some fixed integer $2\leq k\leq O(np)$ (provided $p\ll n^{-1/2}$) and the appearance of $H$ in $G_{n,p}$. Where we define:
\begin{align*}
    m_k&=r_k\sqrt{(\delta+\varepsilon) \E[X]},\\
    m_*&=C(\varepsilon,\delta)\sqrt{\E[X]}/2,
\end{align*}
where $C(\varepsilon,\delta)=\sqrt{r+d^2}-d$ and $r=16(\delta+3\varepsilon/2)$ and $d=\sqrt{2}/(1+\varepsilon)$. Now we are ready to state the main result of this section. We wish to emphasize that both $m_k$ and $m_*$ depend on $\delta$ and $\varepsilon$.


\begin{thm} \label{thm_main_upper_2}
Let $\varepsilon,\delta,C$ be positive real numbers and let $2\leq k\leq Cnp$ be a positive integer. Then the following holds
\[
\mathbb{P}(X\geq(1+\delta)\E[X])\geq 
\begin{cases}
             \left( p^{m_k} \binom{n}{m_k/k}\right)^{1+\varepsilon}  &\text{ for } n^{-1}\ll p\ll n^{-1/2} \\
             p^{({1+\varepsilon})m_*} &\text{ for } n^{-1/2}\ll p\ll 1\\
\end{cases}
\]
for large enough $n$.
\end{thm}

In the following lemma we claim that the expected number of induced copies of $C_4$ conditioned on one labeled copy of $K_{k,m_k/k}$ or $H$ being a subgraph of $G_{n,p}$ in a suitable range of $p$ is at least $(1+\delta +\varepsilon/2)\E[X]$ provided $n$ is large enough. To this end let us introduce the following notations.
From now on we assume that the vertex set of $G_{n,p}$ is $[n]$. For every positive integer $k\leq n$ and $A\subset [n]\setminus[k]$ with $|A|=m_k/k$ define the following events:
\begin{enumerate}[label=(\Roman*)]
    \item 
    $F_{A,k}$ is the event that $\cap_{i=1}^{k}N(i)=A$ and there are no edges between any $i,j\in [k]$,
    \item 
    $F_k=\cup_{A\in \binom{[n]\setminus[k]}{m_k/k}}F_{A,k}$.
\end{enumerate}
Now we are ready to state the lemma.

\begin{lemma} \label{lemma_estimation_of_conditioned_expectation}
    Suppose $\varepsilon,\delta,C$ are positive reals, $2\leq k\leq Cnp$ is some positive integer and $A\subseteq [n]\setminus [k]$ with $|A|=m_k/k$. Then the following holds:
    \begin{enumerate}
        \item\label{lower_bound_k} Suppose $n^{-1}\ll p\ll n^{-1/2}$. Then, for large enough $n$ we have 
        \[
            \E[X\mid F_{A,k}]\geq (1+\delta+3\varepsilon/4)\E[X].
        \]
        \item\label{lower_bound_H} Suppose $n^{-1/2}\ll p\ll 1$. Then, for large enough $n$ we have 
        \[
            \E_{H}[X]\geq (1+\delta+\varepsilon)\E[X].
        \]
    \end{enumerate}
\end{lemma}

\begin{proof}

We start with the first item in the lemma. Assume $n^{-1}\ll p\ll n^{-1/2}$.
First, note that 
\[
    E[X\mid F_{A,k}]\geq N_{ind}(C_4,K_{k,m_k/k})(1-p)+\E[N_{ind}(C_4,G_{n-k,p})].
\]
Let us compute these two quantities.

Since $K_{s,t}$ contains exactly $\binom{s}{2}\binom{t}{2}$ induced copies of $C_4$ we obtain that $K_{k,m_k/k}$ contains 
\[
    (k-1)m_k^2/4k-O(m_k)=(\delta+\varepsilon)\E[X]-O(\E[X]^{1/2})
\]
induced copies of $C_4$ as $r_k=2\sqrt{k/(k-1)}$.
Further, as $k=O(np)$ and for all nonnegative integers $a,b,c$ we have $\binom{a-c}{b}/\binom{a}{b}\geq \left(\frac{a-b-c}{a-b}\right)^b$, we deduce
\[
    \E[N_{ind}(C_4,G_{n-k,p})]=\frac{\binom{n-k}{4}}{\binom{n}{4}}\E[X]\geq \left(1-\frac{k}{n-4}\right)^4\E[X]=(1-o(1))\E[X].
\]
Combining the above we obtain, 
\[
    \E[X\mid F_{A,k}]\geq (1-o(1))(\delta+\varepsilon)\E[X]-O(\E[X]^{1/2})+(1-o(1))\E[X]\geq (1+\delta+3\varepsilon/4)\E[X].
\]

For the second item assume $n^{-1/2}\ll p\ll 1$. Similar to the first case we have
\[
    \E_H[X]\geq (N_{ind}(C_4,H)+N_{ind}(K_{1,2}\sqcup K_{1},H)p^2)(1-p)^2+\E[N_{ind}(C_4,G_{n-2m_*/n,p})].
\]
We now compute these quantities. 

Since $K_{s,t}$ contains exactly $\binom{s}{2}\binom{t}{2}$ induced copies of $C_4$, we obtain that $H$ contains ${m_*^2}/{4}+O(m_*n)$ induced copies of $C_4$. Moreover, thinking of $H$ as a spanning subgraph of $K_n$ by adding isolated vertices, we see that $H$ contains at least $m_*n^2/8+O(m_*^2)$ induced copies of $K_{1,2}\sqcup K_{1}$ as we can choose one vertex from one side, two form the other side and another isolated vertex. Furthermore, as for all nonnegative integers $a,b,c$ we have $\binom{a-c}{b}/\binom{a}{b}\geq \left(\frac{a-b-c}{a-b}\right)^b$ and $2m_*/n=O(np^2)=o(n)$, we deduce
\[
    \E[N_{ind}(C_4,G_{n-2m_*/n,p})]=\frac{\binom{n-2m_*/n}{4}}{\binom{n}{4}}\E[X]\geq \left( 1-\frac{2m_*}{n(n-4)}\right)^4\E[X]\geq (1-o(1))\E[X],
\]
taking $n$ large enough we have, $\E[N_{ind}(C_4,G_{n-2m_*/n,p})]\geq (1-\varepsilon/4)\E[X]$. 
Recalling that $p\gg n^{-1/2}$ and combining all the above bounds we obtain,
\[
    \E_H[X]\geq {m_*^2}/{4}+m_*n^2p^2/8+(1-\varepsilon/4)\E[X]+o(m_*^2).
\]
By this inequality and the definition of $m_*$ one can check that for large enough $n$,
\[
    \E_H[X]\geq (\delta+3\varepsilon/2)\E[X]+(1-\varepsilon/4)\E[X]-\varepsilon/4\E[X]=(1+\delta+\varepsilon)\E[X].
\]
This completes the proof.
\end{proof}

Now we are ready to prove Theorem \ref{thm_main_upper_2}. Before proving the theorem we make the following remark. Suppose $\gamma>0$ is a real number, $k=np$ and $p\gg n^{-1}$. Then $m_0\leq m_k\leq (1+\gamma)m_0$ provided $n$ is sufficiently large. Therefore, Theorem \ref{thm_main_upper_2} invoked with $\varepsilon=\gamma$ implies:
\[
    \mathbb{P}(X\geq (1+\delta)\E[X])\geq \left(p^{m_k}\binom{n}{m_k/np}\right)^{(1+\gamma)}\geq p^{(1+\gamma)^2m_0}
\]
for large enough $n$. Thus, by setting $\gamma=\sqrt{1+\varepsilon}-1$, and letting $\{c_k\}_{k=2}^\infty$ be any increasing sequence satisfying $c_2=0$ and $\lim_{k\rightarrow \infty}c_k=1/3$, we have the following corollary of Theorem~\ref{thm_main_upper_2}, which will be shown to be tight in the next sections for some specific sequence $c_k$.

\begin{cor} \label{thm_main_upper_1}
Let $\varepsilon$ and $\delta$ be positive real numbers and let $k\geq 2$ be positive integer. Then the following holds
\[
\log \mathbb{P}(X\geq (1+\delta)\E[X])\geq 
\begin{cases}
             (1+\varepsilon)(m_k\log p+\log \binom{n}{m_k/k} )&\text{ for } n^{-1+c_{k-1}}\ll p\ll n^{-1+c_{k}},\\
             (1+\varepsilon)m_0\log(p)  &\text{ for } n^{-2/3}\ll p\ll n^{-1/2}, \\
             {(1+\varepsilon)m_*}\log(p) &\text{ for } n^{-1/2}\ll p\ll 1.\\
\end{cases}
\]
for large enough $n$.
\end{cor}

Note that in the proof of Theorem \ref{thm_main_upper_2} we use a similar method as was used in \cite{HarMouSam19}.

\begin{proof}[Proof of Theorem \ref{thm_main_upper_2}]
Through out the proof we assume that the vertex set of $G_{n,p}$ is $[n]$.

We start with the first item. To this end fix an integer $2\leq k\leq Cnp$ and $\varepsilon$ some positive real and assume that $n^{-1}\ll p\ll n^{-1/2}$. Let $A\subset [n]\setminus[k]$ with $|A|=m_k/k$ and recall the definitions of the events $F_{A,k}$ and $F_k$:
\begin{enumerate}[label=(\Roman*)]
    \item 
    $F_{A,k}$ is the event that $\cap_{i=1}^{k}N(i)=A$ and there are no edges between any $i,j\in [k]$,
    \item 
    $F_k=\cup_{A\in \binom{[n]\setminus[k]}{m_k/k}}F_{A,k}$.
\end{enumerate}
Note that for any $A,B\subseteq [n]\setminus[k]$ such that $A\neq B$ with $|A|=|B|=m_k/k$ we have $F_{A,k}\cap F_{B,k}=\emptyset$. 
Therefore, we have the following,
\begin{align}\label{eq_main_thm2}
    \mathbb{P}(X\geq(1+\delta)\E[X])&\geq\mathbb{P}(F_k\text{ and }X\geq(1+\delta)\E[X])\nonumber\\
    &=\sum\limits_{A\in\binom{[n]\setminus[k]}{m_k/k}}\mathbb{P}(X\geq(1+\delta)\E[X]\mid F_{A,k})\cdot \mathbb{P}(F_{A,k})\nonumber \\
    &\geq \sum\limits_{A\in\binom{[n]\setminus[k]}{m_k/k}}\mathbb{P}(X\geq(1+\delta)\E[X]\mid F_{A,k})\cdot p^{m_k}(1-p)^{k^2}(1-p^k)^{(n-k)}.
\end{align}
Lemma \ref{lemma_estimation_of_conditioned_expectation} asserts that provided  $n^{-1}\ll p\ll n^{-1/2}$ we have the following for all $A\subset [n]\setminus[k]$ with $|A|=m_k/k$:
\[
    \E[X\mid F_{A,k}]\geq (1+\delta+\varepsilon/2)\E[X].
\] 
Note that $X\leq n^4$ always, we can bound $\E[X\mid F_{A,k}]$ from above (similar to the proof of Lemma \ref{lemma_upper}) as follows:
\begin{equation*}
    \E[X\mid F_{A,k}]\leq (1+\delta)\E[X]+\mathbb{P}(X\geq(1+\delta)\E[X]\mid F_{A,k})\cdot n^4.
\end{equation*} 
Combining the two inequalities we obtain, 
\begin{equation*} 
    \mathbb{P}(X\geq(1+\delta)\E[X]\mid F_{A,k})\geq \frac{\varepsilon\cdot \E[X]}{2n^4}.
\end{equation*}
Therefore, for large enough $n$ we also have,
\begin{equation}\label{eq_lower_bound_thm2}
        \mathbb{P}(X\geq(1+\delta)\E[X]\mid F_{A,k})\geq \varepsilon p^4(1-p)^2/(2\cdot 4^4) \geq p^5.
\end{equation}
Substituting \eqref{eq_lower_bound_thm2} into \eqref{eq_main_thm2} gives, 
\begin{align}\label{eq_lowerbound_upper_tail}
     \mathbb{P}(X\geq(1+\delta)\E[X])&\geq \sum\limits_{A\in\binom{[n]\setminus[k]}{m_k/k}} p^5 p^{m_k}(1-p)^{k^2}(1-p^k)^{n-k}\nonumber\\
     &= \binom{n-k}{m_k/k}p^{m_k+5}(1-p)^{k^2}(1-p^k)^{n-k}.
\end{align}
We now show that $p^5(1-p)^{k^2}(1-p^k)^{n-k}\geq p^{o(m_k)}$ and $\binom{n-k}{m_k/k}\geq \binom{n}{m_k/k}^{1+o(1)}p^{o(m_k)}$.

First, as $p\ll n^{-1/2}, m_k\gg 1$ and $2\leq k\leq Cnp$, we have the following for large $n$:
\begin{align}\label{eq_bounding_prob}
    p^5(1-p)^{k^2}(1-p^k)^{(n-k)}&\geq \exp(5\log(p)-k^2(p+p^2)-(n-k)(p^k+p^{2k})) \nonumber\\
                                 &\geq \exp(5\log(p)-C^2n^2p^3(1+p)-np^2(1+p^k))\nonumber \\
                                 &= \exp(-o(n^2p^2))\nonumber\\
                                 &\geq p^{\varepsilon m_k/2},
\end{align}
where the first inequality follows as $1-x\geq \exp(-x-x^2)$ for all small enough $x$. 
Second, for all nonnegative integers $a,b,c$ we have $\binom{a-c}{b}/\binom{a}{b}\geq \left(\frac{a-b-c}{a-b}\right)^b$, and thus we obtain that
\[
\binom{n-k}{m_k/k}/\binom{n}{m_k/k}\geq\left({1-\frac{k}{n-m_k/k}}\right)^{m_k/k}=e^{-O(m_k/n)}.
\]
In addition, as $2\leq k\leq Cnp$ and $n^{-1}\ll p\ll n^{-1/2}$, we have
\begin{align*}
   \binom{n}{m_k/k}p^{m_k}&\leq \left(\frac{enk}{m_k}\right)^{m_k/k}p^{m_k}\leq \left(\frac{eCn^2p}{m_k}\right)^{m_k/k}p^{m_k}\\
   &\leq \exp\left(\left(1-\frac{1}{k}+O\left(\frac{1}{k\log(p)}\right)\right)m_k\log(p)\right)\\
   &\leq \exp((1+o(1))m_k\log(p)). 
\end{align*}

Furthermore, as $2\leq k\leq Cnp$ and $n^{-1}\ll p \ll n^{-1/2}$ we have $m_k/n=o(m_k\log(1/p))$ and hence, for sufficiently large $n$ we have,
\begin{equation}\label{eq_bounding_binoms}
    \binom{n-k}{m_k/k}/\binom{n}{m_k/k}\geq \binom{n}{m_k/k}^\varepsilon p^{\varepsilon m_k/2}.
\end{equation}
Combining \eqref{eq_lowerbound_upper_tail}, \eqref{eq_bounding_prob}, and \eqref{eq_bounding_binoms} gives, 
\[
    \mathbb{P}(X\geq(1+\delta)\E[X])\geq \left(\binom{n}{m_k/k} p^{m_k}\right)^{1+\varepsilon}.
\]
This finishes the first part of the proof.

For the second item, define the event $F^*$ to be the event that $G_{n,p}$ contains $K_{2m_*/n,n/2}$ as a subgraph on the vertex set $[n/2+2m_*/n]$ with sides $[2m_*/n]$ and $[n/2+2m_*/n]\setminus[2m_*/n]$. Note that $X\leq n^4$ always. Further, Lemma \ref{lemma_estimation_of_conditioned_expectation} asserts that provided $n^{-1/2}\ll p \ll 1$ we have 
\[
    \E[X\mid F^*]\geq (1+\delta+\varepsilon)\E[X],
\] 
and thus, $\Phi_X(\delta+\varepsilon)\leq -\log \mathbb{P}(F^*)=-\log(p^{m_*})$.
Therefore, applying Lemma \ref{lemma_upper} gives,
\begin{align*}
    -\log \mathbb{P}(X\geq (1+\delta)\E[X])\leq& \Phi_X(\delta+\varepsilon) + \log\left(\frac{n^4}{\varepsilon\E[X]}\right)\leq -\log\left(\frac{p^{m_*} n^4}{\varepsilon\E[X]}\right).
\end{align*}
Moreover, for sufficiently large $n$ we have, 
$\frac{\varepsilon\E[X]}{n^4}\geq p^5$. Thus, taking negative logarithms we obtain the following for large enough $n$: 
\begin{equation*}
    \mathbb{P}(X\geq (1+\delta)\E[X])\geq p^{m_*+5}\geq p^{(1+\varepsilon)m_*}.
    \end{equation*}
This is as claimed.
\end{proof}

\section{Counting the number of cores}\label{sec_5}

Recall that $X$ is the random variable counting the number of induced copies of $C_4$ in $G_{n,p}$. In this section we prove a general upper bound on the logarithmic probability of the upper tail event of $X$. The main tool we use in this section is Theorem~\ref{Main_tool} which is a variation of \cite[Theorem~9.1]{}. There are two major parts in this section. The first is an evaluation of $\Phi_{X}(\delta)$ in different regimes of $p$. The second is an evaluation of the entropic term $|\mathcal{C}_m|$. The main results of this section are the following lemmas.
\begin{lemma}\label{claim_evalutation_of_Phi}
    Suppose $\varepsilon,\delta$ are positive real numbers with $\varepsilon$ being small as a function of $\delta$. Then the following hold:
    \begin{enumerate}[label=(\roman*)]
        \item \label{eq_evaluation_of_Phi_1}
        If $p \ll n^{-1/2}$ then for large enough $n$ we have,
        \[
            2(1-\varepsilon)\sqrt{\delta \E[X]}\leq \Phi _X(\delta)/\log (1/p)\leq 2(1+\varepsilon)\sqrt{\delta \E[X]}.
        \]
        \item \label{eq_evaluation_of_Phi_2} 
        If $n^{-1/2}\ll p \ll 1$ then for large enough $n$ we have,
        \[
            (1-\varepsilon)\left(\sqrt{\frac{n^4p^4}{16}+4\delta \E[X]}-\frac{n^2p^2}{4}\right)\leq \frac{\Phi _X(\delta)}{\log (1/p)}\leq (1+\varepsilon)\left(\sqrt{\frac{n^4p^4}{16}+4\delta \E[X]}-\frac{n^2p^2}{4}\right).
        \]
    \end{enumerate} 
\end{lemma}

Before we present the second lemma, let us remind the reader the definition of $\mathcal{C}_m$.

\begin{definition}
Let $\varepsilon,\delta,K$ be positive reals. In addition let $p\in (0,1)$. Then we define $\mathcal{C}(\varepsilon,\delta,K)$ to be collection of all $G\subset K_n$ spanning subgraphs satisfying the following:
\begin{enumerate}[label=(\subscript{C}{{\arabic*}})]
\item $e(G)\leq K\cdot\Phi _X(\delta+\varepsilon)$,
    \item $N_{ind}(C_4,G)+N_{ind}(K_{1,2}\sqcup K_{1},G)p^2\geq (\delta-\varepsilon)\E[X]$, and
    \item For all $e\in E(G)$ \[
                                    N_{ind}(e,C_4,G)+N_{ind}(e,K_{1,2}\sqcup K_{1},G)p^2\geq \varepsilon \E[X]/(2K\cdot \Phi _X(\delta+\varepsilon)).
                              \]
\end{enumerate}
We call the graphs in $\mathcal{C}(\varepsilon,\delta,K)$ core graphs. Furthermore, for every positive integer $m$ we define $\mathcal{C}_m(\varepsilon,\delta,K)$ be the set of all cores with $m$ edges.
\end{definition}

\begin{lemma}\label{claim_general_bound_3}
        Suppose $\varepsilon,\delta,C,K$ are positive reals with $\varepsilon<1$. Furthermore, suppose $n^{-1}\ll p\ll 1$ as $n$ tends to infinity. Then, there exist $D$ and $n_0$ such that the following holds for all $n>n_0$:\\
        Let $m$ be a positive integer with $Cn^2p^2 \leq m\leq K\Phi_X(\delta+\varepsilon)$. Furthermore, let 
        \[
        v_m=\max \{|\{v\in V(G):\deg(v)\neq 0\}|:G\in \mathcal{C}_m(\varepsilon,\delta,K)\}.
        \] 
        Then,
        \begin{equation*}
            |\mathcal{C}_m|\leq \log(1/p)^{Dm}\binom{n}{v_m}.
        \end{equation*}
\end{lemma}

We start with the first lemma. Let us give some motivation and history of the problem.

\begin{definition}
For every graph $H$ and any positive integer $m$ define, $$N(m,H)=\max|\{T\subset G : T\cong H \}|$$ where the maximum ranges over all graphs with $m$ edges.  
\end{definition}

This definition was first presented by Erd\H{o}s and Hanani in \cite{ErdHan62}. In their paper they computed the asymptotic of this function where $H$ is a clique. In \cite{Alo81} Alon generalized this and computed the asymptotic of this function for all $H$. Later, in \cite{FriKah98} Friedgut and Kahn reproved Alon's result using entropy methods which they were also able to generalize for the case of hypergraphs. 

In \cite{JanOleRuc04} Janson, Oleszkiewicz, and Ruci\'nski found a relation between $N(m,H)$ and a related parameter $N(n,m,H)$ and the probability that the random variable counting the number of copies of a fixed graph $H$ in $G_{n,p}$ exceeds its expectation by a multiplicative factor. This result used and generalized the machinery developed by Friedgut and Kahn. This led us to generalize the definition of $N(m,H)$ to fit our setting. Let us recall some definitions from Section \ref{sec_3} and introduce some new ones.

\begin{definition} \label{def_N(n,m,H)}
Suppose $H$ and $G$ are graphs. Let $n,m$ be positive integers and let $e$ be an edge of $G$. We define:
\begin{itemize}
  \item
  $N_{ind}(H,G)$ is the number of induced copies of $H$ in $G$.
    \item
  $N_{ind}(e,H,G)$ is the number of induced copies of $H$ in $G$ that contain $e$.
  \item
  ${N}_{ind}(n,m,H)$ is the maximum of $N_{ind}(H,G)$ over all graphs $G$ such that the number of vertices of $G$ is at most $n$ and the number of edges of $G$ is $m$.
\end{itemize}
\end{definition}

Note that similar generalizations have been studied in \cite{BolNarTac86} and \cite{DanDanBalMat20}. The following is a simple corollary of Lemma \ref{lemma_inducibility_like} which we prove later in Section \ref{sec_6}.

\begin{cor}\label{cor_inducibility_like_2}
For every $n,m$ positive integers such that $3<m\leq n$ we have,
$${N}_{ind}(n,m,C_4) \leq \frac{m (m-n+1)}{4}\leq \frac{m^2}{4}.$$ 
\end{cor}

In \cite{BolNarTac86} the authors determined the asymptotic behaviour of $\max \{N_{ind}(n,m,K_{s,s}):n\in \mathbb{N}\}$ for any $s\in \mathbb{N}$ (and actually obtained optimal bounds), and gave tight bounds in some ranges of $m$ (in the above we consider only $K_{2,2}$).
In \cite{DanDanBalMat20} the authors considered the problem of determining the asymptotic in $n$ and $m$ of the maximum number of copies of a fixed bipartite graph $H$ in a bipartite host graph with $n$ vertices and $m$ edges. They solved this problem for some class of graphs which includes all complete bipartite graphs $K_{s,t}$. Here we prove a similar statement to the ones in \cite{BolNarTac86,DanDanBalMat20}.

We wish to emphasize that there is a significant difference in the behaviour of the problem depending on whether $n^{-1} \ll p\ll n^{-1/2}$ or $n^{-1/2} \ll p\ll 1$. This can be seen for example in the first result of this section which we now start to prove. To this end we start with an observation. Recall that we denote by $K_{1,2}\sqcup K_1$ the disjoint union of a star with two leaves and an isolated vertex.

\begin{obser}\label{obs_K_1,2_cup_K1}
    Suppose $n,m$ are positive integers with $m\leq \binom{n}{2}$. Then,
    \[
        N_{ind}(n,m,K_{1,2}\sqcup K_{1})\leq mn^2/8.
    \]
\end{obser}

\begin{proof}
Let $G$ be a graph achieving the maximum in the definition of $N_{ind}(n,m,K_{1,2}\sqcup K_{1})$. Let $uv=e\in E(G)$ and denote $x=|N(u)\cup N(v)\setminus \{u,v\}|$. Then,
\begin{align*}
    N_{ind}(e,K_{1,2}\sqcup K_{1},G)\leq x(n-x)\leq n^{2}/4.
\end{align*}
Moreover we have,
\[
    N_{ind}(K_{1,2}\sqcup K_{1},{G})\leq \frac{1}{2}\sum_{e\in E({G})}N_{ind}(e,K_{1,2}\sqcup K_{1},G)\leq \frac{mn^2}{8}.\qedhere
\]
\end{proof}

\begin{proof}[Proof of Lemma \ref{claim_evalutation_of_Phi}]
We start with the upper bounds both in \ref{eq_evaluation_of_Phi_1} and in \ref{eq_evaluation_of_Phi_2}.

For the upper bound in \ref{eq_evaluation_of_Phi_1}, assume that $n^{-1}\ll p\ll n^{-1/2}$. Since $\lim_{n\rightarrow \infty}\E[X]=\infty$ we may treat $m'=\sqrt[4]{4(1+\varepsilon)\delta \E[X]}$ as an integer and let $H=K_{m',m'}$. Note that $H$ has $2\sqrt{(1+\varepsilon)\delta\E[X]}$ edges and provided $n$ is large enough $H$ contains at least $(1+\varepsilon/2)\delta \E[X]>\delta \E[X]$ induced copies of $C_4$. Note further that, 
\[
    \E_H[X]\geq N_{ind}(C_4,H)(1-p)^2+\E[ N_{ind}(C_4,G_{n-2m',p})]\geq (1+\varepsilon/2) \delta \E[X]+\frac{\binom{n-2m'}{4}}{\binom{n}{4}}\E[X].
\]
As $m'\ll n$ we also have the following for sufficiently large $n$,
\[
    \binom{n-2m'}{4}/{\binom{n}{4}}\geq \left(\frac{n-2m'-3}{n}\right)^4\geq 1-\delta \varepsilon /2.
\]
Combining the above inequalities we find the following for large enough $n$,
\[
    \E_H[X]\geq (1+\varepsilon/2)\delta\E[X]+(1-\varepsilon\delta/2)\E[X]=(1+\delta)\E[X].
\]
Therefore, for large enough $n$ we have
\[
    \E_H[X]-\E[X]\geq \delta \E[X]. 
\]
Hence, by the definition of $\Phi_X(\delta)$,
\[
    \Phi_X(\delta)\leq -\log (p^{e(H)})=2\sqrt{(1+\varepsilon)\delta\E[X]}\log (1/p)\leq 2(1+\varepsilon)\sqrt{\delta \E[X]}\log (1/p).
\]

For the upper bound in \ref{eq_evaluation_of_Phi_2}, assume $n^{-1/2} \ll p\ll 1$. 
Letting 
\[
    \Tilde{m}=(1+\varepsilon)\left(\sqrt{\frac{n^4p^4}{16}+4\delta \E[X]}-\frac{n^2p^2}{4}\right),
\]
and recalling that $\lim_{n\rightarrow \infty}\sqrt{\E[X]}/n=\infty$ we may treat $\Tilde{m}/n$ as an integer and consider the following graph. Let $H$ be a graph on the vertex set $[n]$ and such that $H\left[[2\Tilde{m}/n+n/2]\right]\cong K_{2\Tilde{m}/n,n/2}$ and all vertices in $[n]\setminus [2\Tilde{m}/n+n/2]$ are isolated. Note that $H$ contains $\Tilde{m}$ edges. Let $\xi$ be a small positive real. Assuming $n$ is large enough, $H$ contains at least $(1-\xi)\Tilde{m}^2/4$ induced copies of $C_4$. Furthermore, if $n$ is large enough, $H$ contains at least $(1-\xi)\frac{\Tilde{m}n^2}{8}$ induced copies of $K_{1,2}\sqcup K_{1}$. Let $\eta$ be a small positive real (that might depend on $\xi$). Then provided $n$ is large enough,
\begin{align*}
    \E_{H}[X]\geq& \left(N_{ind}(C_4,H)+N_{ind}(K_{1,2}\sqcup K_{1},H)p^2\right)(1-p)^2+\E[N_{ind}(C_4,G_{n-2\Tilde{m}/n,p})]\\
    \geq&(1-\eta)(1-\xi)\left(\frac{\Tilde{m}^2}{4}+\frac{\Tilde{m}n^2p^2}{8} \right)+\frac{\binom{n-2\Tilde{m}/n}{4}}{\binom{n}{4}}\E[X].
\end{align*}
Where the first inequality holds as the first term is at most the expected number of induced copies $C_4$ with at least one vertex in $[2\Tilde{m}/n]$, and the second term is the expected number of induced copies of $C_4$ with no vertices in $[2\Tilde{m}/n]$.

By the definition of $\Tilde{m}$ and the choices of $\xi,\eta$ being small enough we obtain,
\begin{align*}
    (1-\eta)(1-\xi)\left(\frac{\Tilde{m}^2}{4}+\frac{\Tilde{m}n^2p^2}{8} \right)\geq  (1+\varepsilon/2)\delta \E[X]. 
\end{align*}
Further, for large enough $n$ we have,
\[
    \binom{n-2\Tilde{m}/n}{4}/\binom{n}{4}\geq \left(\frac{n-2\Tilde{m}/n-3}{n}\right)^4\geq {(1-\delta \varepsilon/2)}.
\]
Combining all of these inequalities we have,
\[
    \E_{H}[X]-\E[X]\geq \delta \E[X].
\]
Hence, by the definition of $\Phi_X(\delta)$ we have,
\[
    \Phi_X(\delta)\leq -\log (p^{e(H)})=(1+\varepsilon)\left(\sqrt{\frac{n^4p^4}{16}+4\delta \E[X]}-\frac{n^2p^2}{4}\right)\log (1/p).
\]

For the lower bounds of both \ref{eq_evaluation_of_Phi_1} and \ref{eq_evaluation_of_Phi_2} we start with a claim. 

\begin{claim}\label{very_long_eq_fixed}
    Suppose $G$ is a spanning subgraph of $K_n$ achieving the minimum in the definition of $\Phi _X (\delta)$ and let $m=e(G)$. Then for any small enough $\gamma>0$ and large enough $n$ 
    \begin{align*}
    \delta \E[X] \leq m^2/4+\min \{mn^2p^2/8,m^2np^2/2\}+\gamma n^4p^4.
\end{align*}
\end{claim}

\begin{proof}
Let $G$ be a spanning subgraph of $K_n$ achieving the minimum in the definition of $\Phi _X (\delta)$ and put $m=e(G)$. By the definition of $\Phi_X (\delta)$ we have $\E_G[X]-\E [X]\geq \delta \E[X]$. Note that we can bound $\E _G [X]-\E [X]$ from above in the following way:
\begin{align}\label{claim_Phi_difference}
    \E_{{G}}[X]-\E[X] \leq \sum \limits _{ \emptyset \neq H\subseteq^* C_4}N_{ind}(H,{G})p^{4-e(H)},
\end{align}
where $\subseteq ^*$ stands for spanning subgraphs. Let $P_4$ be the path with three edges, $M_2$ be a matching of size two and $K_2\sqcup I_2$ be a disjoint union of an edge and an independent set of size two.
Since any matching of size two span at most one induced copy of $P_4$ or $M_2$ we have
\[
    N_{ind}(P_4,G),N_{ind}(M_2,G)\leq m^2. 
\]
Therefore we obtain
\begin{align*}
    \E _G [X]-\E [X] \leq& N_{ind}(C_4,{G})+N_{ind}(K_{1,2}\sqcup K_{1},{G})p^2\\
    &+\underbrace{m^2p}_{N_{ind}(P_4,{G})}+\underbrace{m^2p^2}_{N_{ind}(M_2,{G})}+\underbrace{mn^2p^3}_{N_{ind}(K_2\sqcup I_2,{G})} \\
    \leq&  N_{ind}(C_4,{G})+N_{ind}(K_{1,2}\sqcup K_{1},{G})p^2+mp(m+mp+n^2p^2).
\end{align*}
By Observation \ref{obs_K_1,2_cup_K1} we have
\[
    N_{ind}(K_{1,2}\sqcup K_{1},{G})\leq N_{ind}(m,n,K_{1,2}\sqcup K_{1}) \leq \frac{mn^2}{8}.
\]
Furthermore we always have,
\[
    N_{ind}(K_{1,2}\sqcup K_{1},{G})\leq \binom{m}{2}\cdot n\leq \frac{m^2n}{2}.
\]
Combining the two we find that,
\begin{align*}
    N_{ind}(K_{1,2}\sqcup K_{1},{G})p^2\leq \min \{mn^2p^2/8,m^2np^2/2\}.
\end{align*}
Hence, by the definition of $\Phi_X(\delta)$ and Corollary \ref{cor_inducibility_like_2}, for large enough $n$ we have,
\begin{align}\label{eq_very_long_1}
    \delta \E[X]&\leq N_{ind}(C_4,{G})+N_{ind}(K_{1,2}\sqcup K_{1},{G})p^2+mp(m+mp+n^2p^2) \nonumber\\
                &\leq m^2/4+\min \{mn^2p^2/8,m^2np^2/2\}+mp(m+mp+n^2p^2).
\end{align}
As $p\ll 1$ and the upper bound achieved before, we have $m=O(n^2p^2)$. Therefore, 
for any $\gamma>0$ and sufficiently large $n$ we obtain from \eqref{eq_very_long_1} the following inequality
\[
    \delta \E[X] \leq m^2/4+\min \{mn^2p^2/8,m^2np^2/2\}+\gamma n^4p^4. \qedhere
\]
\end{proof}

We now prove the lower bound in \ref{eq_evaluation_of_Phi_1}. For this assume that $n^{-1}\ll p\ll n^{-1/2}$ and note that this implies that $\min\{mn^2p^2/8,m^2np^2/2\}=o(m^2)$. Thus, for any $\gamma>0$ and large enough $n$ we deduce the following from Claim \ref{very_long_eq_fixed},
\[
    \delta \E[X]\leq m^2/4+\min \{mn^2p^2/8,m^2np^2/2\}+\gamma n^4p^4\leq (1/4+\gamma)m^2+\gamma n^4p^4.
\]
Taking $\gamma$ sufficiently small we obtain the following bound on $m$,
\[
    (1-\varepsilon)\delta \E[X] \leq \frac{1}{4(1-\varepsilon)}m^2,
\]
implying 
\[
    m\geq 2(1-\varepsilon)\sqrt{\delta \E[X]}.
\]
This proves the lower bound in \ref{eq_evaluation_of_Phi_1} as for sufficiently large $n$ we obtain,
\[
    \Phi_X(\delta)=-\log(p^m)\geq 2(1-\varepsilon)\sqrt{\delta \E[X]}\log (1/p).
\]

Lastly, we prove the lower bound in \ref{eq_evaluation_of_Phi_2}. For this assume $n^{-1/2}\ll p\ll1$. Let $\gamma>0$ be some small constant, then for large enough $n$ Claim \ref{very_long_eq_fixed} implies,
\begin{align*}
    \delta \E[X] \leq m^2/4+mn^2p^2/8+\gamma \E[X].
\end{align*}
We conclude the following as the above is a quadratic inequality in $m$ and the fact that $\E[X]=(1+o(1))\frac{n^4p^4}{8}$,
\[
    m\geq \sqrt{\frac{n^4p^4}{16}+4(\delta-\gamma) \E[X]}-\frac{n^2p^2}{4} \geq (1-\varepsilon)\left(\sqrt{\frac{n^4p^4}{16}+4\delta \E[X]}-\frac{n^2p^2}{4}\right) .
\]
This implies \ref{eq_evaluation_of_Phi_2} as follows,
\[
    \Phi_X(\delta)=-\log(p^m)\geq (1-\varepsilon)\left(\sqrt{\frac{n^4p^4}{16}+4\delta \E[X]}-\frac{n^2p^2}{4}\right)\log (1/p). \qedhere
\]
\end{proof}

This finishes the first evaluation we prove in this section. The second evaluation is the evaluation of the entropic term $|\mathcal{C}_m|$ given by Lemma \ref{claim_general_bound_3}. Roughly speaking Lemma \ref{claim_general_bound_3} shows that the number of core graphs with $m$ edges is determined by $v_m$ the maximum number of non-isolated vertices over all graphs in $\mathcal{C}_m$. As seen already, there is a big difference in the behaviour of the problem depending on whether $p\ll n^{-1/2}$ or $p\gg n^{-1/2}$. Later, we will derive two corollaries from Lemma \ref{claim_general_bound_3}, corresponding to these regimes, these corollaries will be very important in Section \ref{sec_6}.





\begin{proof}[Proof of Lemma \ref{claim_general_bound_3}]
Let $G\in \mathcal{C}_m$ be a core graph with $m$ edges and let $uv$ be an edge in $G$. Then by the definition of a core graph we have,
\[
    N_{ind}(uv,C_4,G)+N_{ind}(uv,K_{1,2},G)np^2\geq \varepsilon \E[X]/(2K\cdot \Phi _X(\delta+\varepsilon)).
\]
Therefore, 
\begin{equation}\label{eq_condition_1_in_(*)}
    \deg(u)\deg(v)\geq N_{ind}(uv,C_4,G)\geq \varepsilon \E[X]/(4K\cdot \Phi _X(\delta+\varepsilon)),
\end{equation}
or
\begin{equation}\label{eq_condition_2_in_(*)}
    \deg(u)+\deg(v)\geq N_{ind}(uv,K_{1,2},G)\geq \varepsilon \E[X]/(4K\cdot \Phi _X(\delta+\varepsilon)np^2),
\end{equation}
call this property $(*)$.
To bound $|\mathcal{C}_m|$ it is enough to bound the number of subgraphs of $K_n$ with $m$ edges satisfying property $(*)$. This can be bounded from above by multiplying the following quantities:
\begin{enumerate}
    \item The number of possible ways to choose the set of non-isolated vertices of such graph. 
    \item The number of possible choices of the edges of the graph such that property $(*)$ is being satisfied.
\end{enumerate}
The first item can be bounded from above by the number of ways to choose a set of at most $v_m$ vertices which is at most $2^{v_m}\binom{n}{v_m}$. 

We now bound the second item. 
Let $H$ be a graph with $v$ vertices and $m$ edges that satisfies property $(*)$. For every integer $0\leq t\leq \log_2(m)$ define $V_t(H)=\{v\in V(H):2^t\leq\deg(v)<2^{t+1}\}$. Furthermore, for every integer $0\leq t\leq \log_2(m)$ define $U_t(H)= \cup_{\ell \geq t}V_\ell$. Using a standard double counting argument and the fact that $U_{t}(H)\subseteq V(H)$ we have 
\[
    |V_t(H)|\leq |U_t(H)|\leq \min\left\{\frac{m}{2^{t-1}},v_m\right\}
\]
for all $t\geq 0$.
By property $(*)$ edges can be placed between the sets $V_t$ in two ways.

The first option is when the edges satisfy \eqref{eq_condition_1_in_(*)}. Since, the degree of each vertex is always bounded from above by $n$, we obtain that the degrees of the endpoints of each edge satisfying \eqref{eq_condition_1_in_(*)} are bounded from below by $d^*$ defined as 
\[
    d^* = \frac{\varepsilon \E[X]}{4K\cdot n\Phi_X(\delta+\varepsilon)}=\frac{C_0np^2}{\log(1/p)}
\]
where $C_0$ is some positive real that might depend on $\varepsilon,\delta$ and $K$. Thus, such edges can only connect a vertex in $V_t$ and a vertex in $U_{\lfloor \log_2 \ell (t)\rfloor}$ for integers $\lfloor \log_2(d^*)\rfloor \leq t\leq \log_2(n)$ and $\ell (t) =\varepsilon \E[X]/(4K\cdot \Phi _X(\delta+\varepsilon)2^{t})$. We denote the number of such options by $T_m$. 

The second option is that the edge has an endpoint in $U_{\lfloor\log_2(r)\rfloor}(H)$ where $r$ is defined as $r=\varepsilon \E[X]/\left(8K\cdot \Phi _X(\delta+\varepsilon)np^2\right)$. That happens when the edge satisfies \eqref{eq_condition_2_in_(*)}. We denote the number of such options by $S_m$.

Furthermore, note that provided $n$ is large enough there are at most
\[
    {\left(\log_2(n)-\log_2(d^*)+2\right)^{v_m}}={\left(\log_2\left(\frac{\log(1/p)}{C_0p^2}\right)+2\right)^{v_m}}\leq 3^{v_m}\log(1/p)^{v_m}
\]
partitions of the non-isolated vertices of our graph into sets $V_t$ where $\lfloor\log_2(d^*)\rfloor\leq t\leq \log_2(n)$ is an integer and $\cup_{t=0}^{\lfloor\log_2(d^*)\rfloor-1}V_t$. We will think of the vertices in $V_t$ as vertices which will have degrees between $2^t$ and $2^{t+1}$ where $\lfloor\log_2(d^*)\rfloor\leq t\leq \log _2(n)$ is an integer.

We conclude that given the vertex set and a partition as explained above, there are at most $T_m+S_m$ pairs of vertices that can be edges of $H$.
Hence, the number of graphs with property $(*)$ is bounded from above by 
\[
    {\left(3\log(1/p)\right)^{v_m}} 2^{v_m}\binom{n}{v_m} \binom{S_m+T_m}{m}\leq (6\log(1/p))^{v_m}\binom{n}{v_m} \binom{S_m+T_m}{m}.
\]
Let us now estimate $S_m$ and $T_m$. 
\begin{align*}
    T_m=&\sum\limits_{t=\lfloor\log_2(d^*)\rfloor}^{\log_2(n)}|V_t||U_{\lfloor{\log_2(\ell (t))}\rfloor}|\leq \sum\limits_{t=\lfloor\log_2(d^*)\rfloor}^{\log_2(n)}\frac{m}{2^{t-1}}\frac{4m}{\ell (t)}=\sum\limits_{t=\lfloor\log_2(d^*)\rfloor}^{\log_2(n)}\frac{m^2}{2^{t-3}}\frac{4K\cdot \Phi _X(\delta+\varepsilon)2^{t}}{\varepsilon \E[X]}\\
    \leq & (\log_2(n)-\lfloor\log_2(d^*)\rfloor+1)\frac{64K\Phi_X(\delta+\varepsilon)m^2}{\varepsilon \E[X]}\leq\frac{192 K\Phi_X(\delta+\varepsilon)m^2\log(1/p)}{\varepsilon \E[X]} .
\end{align*}
Recalling the assumption that $m\leq K\Phi_X(\delta+\varepsilon)$ we obtain,
\[
    T_m\leq \frac{192 K^3\Phi^3_X(\delta+\varepsilon)\log(1/p)}{\varepsilon \E[X]}.
\]
Applying Lemma \ref{claim_evalutation_of_Phi} we find that there is $C_1>0$ such that,
\[
    \Phi_X(\delta+\varepsilon)\leq \frac{C_1 \E[X]}{n^2p^2}\log(1/p).
\]
By our assumptions, $Cn^2p^2\leq m\leq \Phi_X(\delta+\varepsilon)= O(n^2p^2\log(1/p))$. This implies there are $C_2,C_3>0$ such that for large enough $n_0$ we have,
\begin{align*}
    T_m\leq& C_2n^2p^2\log^4(1/p)\\
       \leq& C_3m\log^5(1/p).
\end{align*}
For $S_m$ recall that there are at most $4m/r$ vertices in $U_{\lfloor \log_2(r)\rfloor}(H)$ and further recall that $\Phi_X(\delta+\varepsilon)=O(n^2p^2\log(1/p))$. Hence,
\[
    S_m\leq n\frac{4m\log(1/p)}{r}\leq C_4m\log(1/p),
\]
where $C_4=C_4(\varepsilon,\delta,K)>0$. 
This implies that for large enough $n_0$ we have $S_m, T_m =O(m\log^5(1/p))$.
Putting it all together there exists $C_5>0$ such that provided $n_0$ is large enough we have,

\begin{align*}
    |C_m|/\binom{n}{v_m}\leq&{(6\log(1/p))}^{v_m}\binom{T_m+S_m}{m}\\
    \leq& {(6\log_2(1/p))}^{2m}\left(\frac{C_5m\log^5(1/p)}{m}\right)^m\\
    \leq& {\log(1/p)}^{O(m)}.
\end{align*}
This proves the lemma.
\end{proof}

We now deduce two corollaries and discuss how one would use them in order to obtain an upper bound on the upper tail probability of $X$. The first corollary will be used when $n^{-1}\ll p\ll n^{-1/2}$ and the second when $n^{-1/2}\ll p\ll1$.

\begin{cor}\label{claim_general_bound}
        Suppose $\varepsilon,\delta,C,K$ are positive reals with $\varepsilon<1$. Furthermore, suppose $n^{-1} \ll p\ll n^{-1/2}$ as $n$ tends to infinity. Then, there exist $D>0$ and $n_0$ such that the following holds for all $n>n_0$:\\
        Let $m$ be a positive integer with $Cn^2p^2 \leq m\leq K\Phi_X(\delta+\varepsilon)$. Furthermore, let 
        \[
        v_m=\max \{|\{v\in V(G):\deg(v)\neq 0\}|:G\in \mathcal{C}_m(\varepsilon,\delta,K)\}.
        \] 
        Then,
    \begin{equation*}
        \log \left(|\mathcal{C}_m|\right)\leq-v_m\log\left(np^2\right)+\eta m\log(n).
    \end{equation*}  
\end{cor}

\begin{proof}
By Lemma \ref{claim_general_bound_3} and the assumption that $p\gg n^{-1}$ we have the following provided $n_0$ is sufficiently large,
\[
    \log(|\mathcal{C}_m|)\leq \log\binom{n}{v_m}+Dm\log\log(1/p)\leq v_m\log(en/v_m)+\eta m\log(n)/3.
\]
Since, $v_m\leq 2m$ we deduce that
\[
    \log(|\mathcal{C}_m|)\leq v_m\log(n/v_m)+2\eta m \log(n)/3.
\]
We now split into two cases. First, assume $v_m\leq \eta m/3$. This assumption implies that
\[
    v_m\log(n/v_m)\leq \eta m\log(n)/3.
\]
On the other hand, if we assume $v_m\geq \eta m/3$ we obtain
\[
    v_m\log(n/v_m)\leq v_m\log(3n/\eta m)=-v_m\log(np^2)+v_m\log( O(1)),
\]
where the equality is due to the assumption that $m=\Omega(n^2p^2)$. Since $v_m\leq 2m$, in both cases we obtain that, provided $n$ is large enough, the following holds:
\begin{equation*}
    \log \left(|\mathcal{C}_m|\right)\leq -v_m\log\left(np^2\right)+\eta m\log(n). \qedhere
\end{equation*}
\end{proof}

Corollary \ref{claim_general_bound} suggests the following `plan of attack' which we implement in Section \ref{sec_6}. Assuming $p\ll n^{-1/2}$ the term $\log(np^2)$ is negative. Therefore in order to estimate $|\mathcal{C}_m|$ estimate the maximum number of vertices that a core with $m$ edges can have for each $m$. This estimation and Corollary \ref{claim_general_bound} then yield an upper bound on the number of cores with $m$ edges, denote this bound by $\beta_m$. Then we compare $p^m\beta_m$ for every $m_0\leq m\leq K\Phi_X (\delta)$, where $m_0$ is the minimum number of edges in a core graph and denote this maximum by $\beta^*$. This implies,
\[
\sum \limits _{m=m_0}^{K\Phi _\delta (X)} p^m |\mathcal {C}_m|\leq K\Phi_{X}(\delta)\beta ^*.
\]
We also show that $K\Phi_{X}(\delta)$ is negligible compared to $\beta^*$ and therefore we obtain,
\[
\log\left(\sum \limits _{m=m_0}^{K\Phi_X (\delta)} p^m |\mathcal {C}_m|\right)\leq (1-\eta)\log(\beta ^*).
\]
Then Theorem \ref{thm:main_cores_1} will imply
\[
    \log \mathbb{P}(X\geq (1+\delta)\E[X])\leq (1-\varepsilon)\log (\beta^*).
\]
Note that $\beta^*$ depends on $p$. We will also show that this upper bound is matched to the lower bounds that we obtained in Section \ref{sec_4}. 

Next we present a corollary of Lemma \ref{claim_general_bound_3} which will be used in the range where $p\gg n^{-1/2}$.

\begin{cor}\label{claim_general_bound_2}
        Suppose $\varepsilon,\delta,C,K$ are positive reals with $\varepsilon<1$. Furthermore, suppose $n^{-1/2}\ll p\ll 1$ as $n$ tends to infinity. Then, there exists $n_0$ such that the following holds for all $n>n_0$:\\
        Let $m$ be a positive integer with $Cn^2p^2 \leq m\leq K\Phi_X(\delta+\varepsilon)$. Then,
        \begin{equation*}
            |\mathcal{C}_m|\leq \left(\frac{1}{p}\right)^{\varepsilon m}.
        \end{equation*}   
\end{cor}

This is also an immediate corollary of Lemma \ref{claim_general_bound_3}. To see this note that in this range of $p$ we have $n \ll m$ and thus $\binom{n}{v_m}\leq 2^n\leq \log(1/p)^{m}$. Hence Lemma \ref{claim_general_bound_3} implies 
\[
  \log|\mathcal{C}_m|\leq O(m\log\log(1/p))=o(m\log(1/p)).  
\]

Similar to before $\Phi_X(\delta+\varepsilon)$ will be shown to be negligible and therefore Corollary \ref{claim_general_bound_2} will yield the following provided $p\gg n^{-1/2}$,
\[
    \log\left(\sum_{m=m_*}^{K\Phi_X(\delta)}p^m|\mathcal{C}_m|\right)\leq(1-\varepsilon)m_*\log\left(p\right) 
\]
where $m_*$ is the minimum number of edges in a core graph. Which then by Theorem \ref{thm:main_cores_1} implies
\[
    \log \mathbb{P}(X\geq (1+\delta)\E[X])\leq (1-\varepsilon)m_*\log (p).
\]
In the next section we will compute this $m_*$ and show that the above matches the lower bounds given in Section \ref{sec_4}.

\section{Upper bounds}\label{sec_6}
As can be seen in previous sections there is a big difference in the behaviour of the problem depending on the regime $p$ lies at. As explained in Section \ref{sec_5}, in order to obtain quantitative bounds on the upper tail probability we need to bound the entropic term $|\mathcal{C}_m|$ when $\frac{\log^9(n)}{n}\ll p\ll \frac{1}{\sqrt{n}\log(n)}$, and when $p\gg n^{-1/2}$ we only need to compute the minimum number of edges in a core graph. Actually, the situation is more involved in the sparse regime where we see a surprising change in the behaviour of the problem. To present the main result of this section let us define a sequence $c_k$ for $k\geq 2$:
\[
    c_k=\Bigg\{
    \begin{array}{@{}l@{\thinspace}l}
        0 &\text{ for } k=1, \\
        \frac{1}{2 + \sqrt{\frac{k+1}{k-1}}} &\text{ for }k\geq 2.
    \end{array}
\]
We note that $c_k$ is an increasing sequence and $\lim _{k\rightarrow \infty}c_k=1/3$. This is the sequence promised in the our main theorem \ref{thm_main_result} and in Section \ref{sec_4}. Furthermore, recall the definitions of $m_k$ and $m_*$ which were given at the beginning of Section \ref{sec_4}. For the convenience of the reader we give these definitions here as well. First, we define $m_k=r_k\sqrt{(\delta+\varepsilon)\E[X]}$, where
\[
r_k=\Bigg\{
\begin{array}{@{}l@{\thinspace}l}
        2 &\text{ for } k=0, \\
        2\sqrt{k/(k-1)} &\text{ for }k\geq 2.\\
\end{array}
\]
Second, we define
\[
    m_*=(\sqrt{r+d^2}-d)\sqrt{\E[X]}/2,
\]
where $r=16(\delta+3\varepsilon/2)$ and $d=\sqrt{2}/(1+\varepsilon)$.
We are now ready to state the main theorem of this section which is the following.


\begin{thm}\label{thm_main_theorem}
    Suppose $k$ is an integer greater than $1$ and suppose that $\varepsilon,\delta$ are positive reals with $\varepsilon$ small enough. Then there exists $n_0$ such that for any $n>n_0$ the following holds:
    \begin{enumerate}
        \item \label{item_1} If $ \max \left\{ \frac{\log^9(n)}{n},n^{-1+c_{k-1}}\right\}\leq p\leq n^{-1+c_{k}}$, then
        \[
            \log \mathbb{P}(X\geq (1+\delta)\E[X]) \leq {(1-\varepsilon)}\log\left(p^{m_{k}}\binom{n}{m_{k}/{k}}\right).
        \] 
        \item \label{item_2} If $n^{-2/3}\leq p\leq {\frac{1}{\sqrt{n}\log(n)}}$, then
        \[
            \log\mathbb{P}(X\geq (1+\delta)\E[X]) \leq {(1-\varepsilon)}m_0\log\left(p\right).
        \]
        \item \label{item_3} If $n^{-1/2}\ll p\ll 1$, then
        \[
            \log \mathbb{P}(X\geq (1+\delta)\E[X])\leq {(1-\varepsilon)}m_*\log (p).
        \]
    \end{enumerate}
\end{thm}

The proof naturally splits into three parts. We will prove each part in a different subsection. Before we split into cases let us prove a lemma which makes a connection between the number of induced copies of $C_4$ in a graph to the numbers of edges and vertices of that graph. This lemma will be important for us both in the sparse regime and the dense regime.

\begin{lemma}\label{lemma_inducibility_like}
Suppose $n,m$ are positive integers such that $m>3$ and let $\gamma\geq 0$ be real. Let $G$ be a graph with $n$ vertices and $m$ edges and assume $\delta(G)\geq 2$. Then,
$$N_{ind}(C_4,G) \leq \frac{m (m-n+1)}{4}.$$
\end{lemma}

\begin{remark} 
The bound from the lemma is sharp when $n=\frac{m}{k}+k$ for some $k\in \mathbb{N}$     as the graph $K_{k,\frac{m}{k}}$ has $n$ vertices and $m$ edges and contains             $\binom{k}{2}\binom{\frac{m}{k}}{2}=\frac{m}{4}(m-k-\frac{m}{k}+1)$ induced copies of $C_4$.
\end{remark}

\begin{proof}[Proof of Lemma \ref{lemma_inducibility_like}]
Since an $n$-vertex graph with minimum degree at least $2$ has at least $n$ and at most $\binom{n}{2}$ edges we may assume that $n \leq m \leq \binom{n}{2}$. 
Since, 
$$N_{ind}(C_4,G)=\frac{1}{4}\sum_{e\in E(G)}N_{ind}(e,C_4,G) \leq \frac{m}{4}\max _{e\in E(G)} N_{ind}(e,C_4,G).$$
We only need to prove that $N_{ind}(uv,C_4,G)\leq m-n+1$ for all $uv\in E(G)$. Note that every $C_4$ containing $uv$ is determined by the `parallel' edge to $uv$ meaning the edge between the two other vertices in that $C_4$. This is because we are considering induced copies. In other words, the number of induced copies containing $uv$ is bounded by the number of edges between $N(u)\setminus\{v\}$ and $N(v)\setminus\{u\}$. Let $X$ be the set of vertices not in $N(u)\cup N(v)$ and note that $|X|\geq n-\deg(u)-\deg(v)$. The number of edges between $N(u)\setminus\{v\}$ and $N(v)\setminus\{u\}$ is bounded from above by $m-e(X)-(\deg(u)+\deg(v)-1)$ where $e(X)$ is the number of edges with an endpoint in $X$.
Since $\delta(G)\geq 2$ we may estimate $e(X)$ as follows,
$$e(X)\geq \frac{1}{2}\sum\limits_{x\in X}\deg(x)\geq |X|.$$
Thus, the number of edges between $N(u)\setminus\{v\}$ and $N(v)\setminus\{u\}$ is bounded from above by 
\[
m-(|X|+\deg(v)+\deg(u)-1)\leq m-n+1
\]
and the lemma follows.
\end{proof}

In order to prove Theorem \ref{thm_main_theorem} we now split into cases, each dealing with each item of the theorem.

\subsection{The dense regime}
Assume that $n^{-1/2} \ll p \ll 1$ and assume that $\varepsilon$ is any positive real which is small enough as a function of $\delta$.
Corollary \ref{claim_general_bound_2} asserts that in this regime of $p$ the number of core graphs is negligible. Therefore, as explained at the end of Section \ref{sec_5}, to bound the upper tail probability using Theorem \ref{thm:main_cores_1} we are only left with showing that $m_*$ is a lower bound on the minimum number of edges in a core. This is given in the following lemma.

\begin{lemma}\label{claim_last_claim}
    Suppose $\varepsilon,\delta$ are positive real numbers with $\varepsilon$ small enough as a function of $\delta$. Further suppose $p\gg n^{-1/2}$. Then, there exists $n_0$ such that for all $n>n_0$ 
    \[
        \min\{e(G):G\in \mathcal{C}\}\geq m_*.
    \]
\end{lemma}

Before proving the claim let us recall Observation \ref{obs_K_1,2_cup_K1}. For the convenience of the reader we state it here as well:
\begin{obser}\label{obs_K_1,2_cup_K1_second}
    Suppose $n,m$ are positive integers with $m\leq \binom{n}{2}$. Then,
    \[
        N_{ind}(m,n,K_{1,2}\sqcup K_{1})\leq mn^2/8.
    \]
\end{obser}

\begin{proof}[Proof of Lemma \ref{claim_last_claim}]
Suppose $G\in \mathcal{C}_m$ is a core graph with $m$ edges. Then $G$ is also a structured seed and thus
\[
    N_{ind}(C_4,G)+N_{ind}(K_{1,2}\sqcup K_{1},G)p^2\geq (\delta -\varepsilon )\E[X].
\]
Lemma \ref{lemma_inducibility_like} and Observation \ref{obs_K_1,2_cup_K1_second} assert that 
\[
    N_{ind}(C_4,G)\leq N_{ind}(m,C_4)\leq m^2/4,
\]
and
\[
    N_{ind}(K_{1,2}\sqcup K_{1},G)\leq N_{ind}(m,n,K_{1,2}\sqcup K_{1})\leq mn^2/8.
\]
Therefore, 
\[
    m^2+mn^2p^2/2-4(\delta-\varepsilon)\E[X]\geq 0.
\]
Taking $n_0$ large enough we have $\E[X] \geq (1-\varepsilon)^2n^4p^4/8 $ and thus,
\[
    m^2+\frac{m\sqrt{2\E[X]}}{(1-\varepsilon)}-4(\delta-\varepsilon)\E[X]\geq 0.
\]
This implies that for large enough $n_0$ and small enough $\varepsilon$ we have
\[
    m\geq \left(\sqrt{\frac{2}{(1-\varepsilon)^2}+16(\delta -\varepsilon)}-\frac{\sqrt{2}}{(1-\varepsilon)}\right)\sqrt{\E[X]}/2 = m_*.\qedhere
\]
\end{proof}

As explained after Corollary \ref{claim_general_bound_2}, this implies the following corollary which is the third item in Theorem \ref{thm_main_theorem}.

\begin{cor}\label{cor_when_c>1/3}
    Suppose $\varepsilon,\delta$ are positive real numbers with $\varepsilon$ small enough as a function of $\delta$. Further suppose $n^{-1/2}\ll p\ll 1$. Then, for large enough $n$ we have
    \[
        \log \mathbb{P}(X\geq (1+\delta)\E[X])\leq {(1-\varepsilon)}m_*\log (p).
    \]
\end{cor}

\subsection{The sparse regime}
Throughout this subsection, assume $\frac{\log^9(n)}{n}\leq p\ll \frac{1}{\sqrt{n}\log(n)}$. As mentioned earlier, there is another change in the behaviour of the problem when $\frac{\log^9(n)}{n}\leq p \leq  n^{-2/3-
\varepsilon}$ and $n^{-2/3}\leq p \ll \frac{1}{\sqrt{n}\log(n)}$. Before splitting into these two cases we show a reduction and develop some tools which we use in both cases. 

We start with the following fact which plays a major role in the later proofs.

\begin{lemma}\label{obs:product_of_deg}
        Suppose $\varepsilon,\delta,s,K$ are reals and positive with $\varepsilon<1$. Furthermore, suppose $\frac{\log(n)}{n}\leq p\ll \frac{1}{\sqrt{n}\log(n)}$. Then, there exist $n_0$ and $\xi>0$ such that the following holds for all $n>n_0$:\\
        Suppose $G\in \mathcal{C}_m$ where $Cn^2p^2\leq m\leq K\Phi_X(\delta+\varepsilon)$ and $uv\in E(G)$. Then, 
        \[
            \deg(u)\deg(v)\geq \frac{ \xi n^2p^2}{\log(1/p)}\geq s,
        \]
        and furthermore, $\deg(u),\deg(v)\geq 2$.
\end{lemma}


\begin{proof}
Let $G\in \mathcal{C}_m$ be a core graph with $Cn^2p^2\leq m\leq K\Phi_X(\delta+\varepsilon)$ edges and let $uv\in E(G)$ be an edge of $G$. By the definition of $\mathcal{C}_m$ we have the following:
\[
    N_{ind}(uv,C_4,G)+N_{ind}(uv,K_{1,2},G)np^2\geq \frac{\varepsilon\E[X]}{2K\Phi_X(\delta+\varepsilon)}.
\]
Note that $N_{ind}(uv,K_{1,2},G)\leq m\leq K\Phi_X(\delta+\varepsilon)$. Lemma \ref{claim_evalutation_of_Phi} asserts that there exists $D>0$ such that provided $n_0$ is large enough we have
\[
   \Phi_{X}(\delta+\varepsilon)\leq D \sqrt{\E[X]}\log(1/p).
\]
Hence as $p\ll \frac{1}{\sqrt{n}\log(n)}$ we deduce the following for large enough $n_0$:
\[
    N_{ind}(uv,K_{1,2},G)np^2\leq KD \sqrt{\E[X]} np^2 \log(1/p) \ll \sqrt{\E[X]}/\log(1/p),
\]
and 
\begin{equation}\label{equation - bound on N_ind(uv)}
    N_{ind}(uv,C_4,G)\geq \frac{\varepsilon  \sqrt{\E[X]}}{2DK\log(1/p)}-N_{ind}(uv,K_{1,2},G)np^2\geq \frac{\varepsilon  \sqrt{\E[X]}}{4DK\log(1/p)}.
\end{equation}
Note also that $\deg(u)\deg(v)\geq N_{ind}(C_4,uv,G)$ and thus the assertion follows for small enough $\xi$,
\[
    \deg(u)\deg(v)\geq \frac{\varepsilon \sqrt{\E[X]}}{4DK\log(1/p)}\geq \frac{\xi n^2p^2}{\log(1/p)}
\]
and if $\deg(u) = 1$ or $\deg(v) = 1$ then $N_{ind}(C_4, uv,G) = 0$ which is a contradiction to \eqref{equation - bound on N_ind(uv)}.
\end{proof}

The above lemma will be used many times throughout this subsection, mostly the fact that when $\frac{\log^9(n)}{n}\leq p\ll \frac{1}{\sqrt{n}\log(n)}$ the minimum degree of a core graph is at least $2$. We will omit the reference to this lemma and keep in mind that the minimum degree of all graphs considered is at least $2$.

The following is a corollary which bounds $v_m=\max |\{v\in V(G):G\in \mathcal{C}_m \text{ and }\deg(v)>0\}|$. This will be used afterwards to formalize the discussion after Corollary \ref{claim_general_bound}.

\begin{lemma}\label{lemma_weak_bound}
        Suppose $\varepsilon,\delta,K,\gamma$ are positive reals. Suppose further that $\frac{\log^9(n)}{n} \leq p \ll \frac{1}{\sqrt{n}\log(n)}$. Then, there exists an integer $n_0$ such that for all integers $n>n_0$ the following holds:\\
        Suppose $m$ is an integer with $m_0\leq m \leq K\Phi_{X}(\delta+\varepsilon)$. Then,
        \[
        v_m\leq \frac{m}{2}+m^{3/4}.
        \]
\end{lemma}

\begin{proof}
Fix an arbitrary $G\in \mathcal{C}_m$ and define, $A=\{v\in V(G):\deg(v)<\sqrt{m}/\log^2(n)\}$. We have
\[
    2m=\sum_{v\in V(G)} \deg(v)=\sum_{v\in A} \deg(v)+\sum_{v\in A^c} \deg(v) \geq |A^c|\sqrt{m}/\log^2(n).
\]
Thus, $|A^c|\leq 2\sqrt{m}\cdot \log^2(n)$.

By Theorem \ref{claim_evalutation_of_Phi} we have $\Phi_{X}(\delta+\varepsilon)=O(n^2p^2\log(1/p))$ and hence,
\[
    \left(\frac{\sqrt{m}}{\log^2(n)}\right)^2\leq \frac{O(n^2p^2\log(1/p))}{\log^4(n)}\ll \frac{n^2p^2}{\log(1/p)}.
\] 
Thus, by Lemma \ref{obs:product_of_deg} for sufficiently large $n_0$ no two vertices of $A$ can be connected by an edge. Therefore,
\[
    m \geq \sum_{v\in A} \deg(v).
\]
Since $\deg(v)\geq 2$ for all of the vertices of $G$, we obtain $|A|\leq \frac{m}{2}$.

Therefore for large enough $n_0$,
\[
|V(G)|=|A \sqcup A^c|=|A|+|A^c|\leq \frac{m}{2}+2\sqrt{m}\cdot \log^2(n).
\]
Provided $n_0$ is large enough, $2\sqrt{m}\log^2(n) \leq m^{3/4}$. Hence, for large enough $n_0$ the assertion of the lemma holds.
\end{proof}

We now formalize the discussion after Corollary \ref{claim_general_bound} in the following proposition.

\begin{prop}\label{claim_second_interval}
   Suppose $\delta$ is a positive real and $\varepsilon,\eta$ are sufficiently small positive reals. Furthermore, suppose $\frac{\log^9(n)}{n} \leq p \ll \frac{1}{\sqrt{n}\log(n)}$. Then, there exists $n_0$ such that for all $n>n_0$ we have,
    \[
        \log \mathbb{P}(X\geq (1+\delta)\E[X]) \leq (1-\eta )\max_{m_0\leq m \leq m_2}\{m\log(p)-v_m\log(np^2)\}.
    \]
\end{prop}

To prove this proposition we start with a lemma reducing the problem of estimating $m\log(p)-v_m\log(np^2)$ only when $m_0\leq m\leq K\Phi_X(\delta+\varepsilon)$.

\begin{lemma}\label{lem_all_interval}
    Suppose $\delta$ is a positive real and $\varepsilon,\eta$ are sufficiently small positive reals. Furthermore, suppose $n^{-1}\ll p\ll n^{-1/2}/\log(n)$. Then, there exists $K,n_0>0$ such that for all $n>n_0$ we have,
    \[
        \log \mathbb{P}(X\geq (1+\delta)\E[X]) \leq (1-\eta )\max_{m_0\leq m \leq K\Phi_X(\delta+\varepsilon)}\left\{m\log(p)-v_m\log(np^2)\right\}.
    \]
\end{lemma}

\begin{proof}
For all $\eta_1 ,K'>0$ Corollary \ref{claim_general_bound} asserts that provided $n_0$ is large enough we have the following for all $m_0\leq m\leq K'\Phi_X(\delta+\varepsilon)$:
\[
    \log(p^m|\mathcal{C}_m|)\leq m\log(p)-v_m\log(np^2)+\eta_1 m\log(n),
\]
where $v_m=\max |\{v\in V(G):G\in \mathcal{C}_m \text{ and }\deg(v)>0\}|$.

Since $n^{-1} \ll p\ll \frac{1}{\sqrt{n}\log(n)}$, provided $n$ is sufficiently large, we have $v_m\leq m/2+m^{3/4}$ for all $m_0 \leq m\leq K'\Phi_X(\delta+\varepsilon)$ by Lemma \ref{lemma_weak_bound}. Thus, 
\[
    m\log(p)-v_m\log(np^2)\leq -\Omega\left(m\log(n)\right).
\]

For any positive $\eta_2$ we may choose $\eta_1$ small enough such that we have the following for sufficiently large $n$,
\[
    m\log(p)-v_m\log(np^2)+\eta_1 m\log(n)\leq (1-\eta_2)(m\log(p)-v_m\log(np^2)).
\]

Therefore, provided $\eta_2$ is sufficiently small Theorem \ref{thm:main_cores_1} and Theorem \ref{claim_evalutation_of_Phi} then yield that there exist $K,n_0>0$ such that for all $n>n_0$: 
\[
    \log \mathbb{P}(X\geq(1+\delta)\E[X])\leq (1-\eta)\max_{m_0\leq m\leq K\Phi_X(\delta+\varepsilon)}\{m\log(p)-v_m\log(np^2)\}. \qedhere
\]
\end{proof}

We now use the above lemma to prove Proposition \ref{claim_second_interval}

\begin{proof}[Proof of Proposition \ref{claim_second_interval}]
By Lemma \ref{lem_all_interval} it is sufficient to prove that for all $m_2\leq m\leq K\Phi_X(\delta+\varepsilon)$ we have the following provided $\eta'$ is sufficiently small
\[
    m\log(p)-v_m\log(np^2)\leq (1-\eta')(m_2\log(p)-v_{m_2}\log(np^2)).
\]

Provided $n_0$ is large enough Lemma \ref{lemma_weak_bound} gives $v_m\leq m/2+m^{3/4}\leq m/2+\eta' m$ for all $m_2\leq m\leq K\Phi_X(\delta+\varepsilon)$. 
Therefore,
\begin{align*}
    m\log(p)-v_{m}\log(np^2)&\leq m\log(p)-(1/2+\eta')m\log\left(np^2\right)\\
    &\leq (-1/2-2\eta')m\log(n)\leq(-1/2-2\eta')m_2\log(n).
\end{align*}
Noting that each copy of $K_{m_2/2,2}$ in $K_n$ is a core graph with $m_2$ edges we find that $v_{m_2}\geq m_2/2$. Therefore,
\[
    m_2\log(p)-v_{m_2}\log(np^2)\geq m_2\log(p)-\frac{m_2}{2}\log(np^2)=-\frac{m_2}{2}\log(n).
\]
These inequalities imply the following for all $m_2\leq m\leq K\Phi_X(\delta+\varepsilon)$ and sufficiently small $\eta''$
\[
    m\log(p)-v_m\log(np^2) \leq (1-\eta'')(m_2\log(p)-v_{m_2}\log(np^2)). 
\]
This establish the proposition provided $\eta''$ is small enough.
\end{proof}

Proposition \ref{claim_second_interval} and Corollary \ref{claim_general_bound} show that to prove Theorem \ref{thm_main_theorem} when $\frac{\log^9(n)}{n} \leq p \ll \frac{1}{\sqrt{n}\log(n)}$ we only need to bound $m\log(p)-v_m\log(np^2)$ for $m_0\leq m\leq m_2$. Therefore, to prove the first item in Theorem \ref{thm_main_theorem} it is enough to prove that, when $n^{-1+c_{k-1}}\leq p\leq n^{-1+c_k}$ we have
\[
    m\log(p)-v_m\log(np^2)\leq (1-\eta)(m_k\log(p)-v_{m_k}\log(np^2)),
\]
for all $m_0\leq m\leq m_2$.
Moreover, to prove the second item in Theorem \ref{thm_main_theorem} it is enough to prove that, when $n^{-2/3}\leq p\ll \frac{1}{\sqrt{n}\log(n)}$ we have
\[
    m\log(p)-v_m\log(np^2)\leq (1-\eta)m_0\log(p),
\]
for all $m_0\leq m\leq m_2$.
We now split into these two cases mentioned above and prove them. 
First we deal with the denser case where $n^{-2/3}\leq  p \ll \frac{1}{\sqrt{n}\log(n)}$.

\subsubsection{The denser case in the sparse regime} Assume $n^{-2/3}\leq p \ll \frac{1}{\sqrt{n}\log(n)}$. Proposition \ref{claim_second_interval} implies that we need to evaluate the term $m\log(p)-v_m\log(np^2)$ only for $m_0\leq m\leq m_2$. To do this we use Lemma \ref{lemma_inducibility_like}  which implies a connection between the number of induced copies of $C_4$ in a graph and the number of edges and vertices in it. We start with a simple corollary of Lemma \ref{lemma_inducibility_like}.

\begin{cor} \label{claim_weak_bound}
Suppose $\varepsilon,\delta,K$ are positive reals. Suppose further that $\frac{\log(n)}{n} \leq p \ll 1$ and $m$ is an integer. Then, for any $G\in\mathcal{C}_m$ the number of vertices of $G$ is at most 
\[
    m-4(\delta -\varepsilon)\E[X]/m+1.
\]
\end{cor}

\begin{proof}
Since the minimum degree of $G$ is at least $2$, we may apply Lemma \ref{lemma_inducibility_like} and obtain,
\begin{equation}
    (\delta-\varepsilon)\E[X]\leq N(C_4,G)\leq N_{ind}(v(G),m,C_4)\leq \frac{m(m-v(G)+1)}{4}
\end{equation}
By algebraic manipulations we see that, $v(G)\leq m-4(\delta -\varepsilon)\E[X]/m+1$ which is the assertion of the corollary.
\end{proof}

Now we obtain the desired bound for $m\log(p)-v_m\log(np^2)$ in this regime of $p$.

\begin{lemma}\label{claim_when_c>1/3}
    Suppose that $\varepsilon,\delta,\gamma$ are positive reals with $\varepsilon$ small enough, and $\gamma$ small enough as a function of $\varepsilon$. Furthermore, suppose $n^{-2/3}\leq p\ll \frac{1}{\sqrt{n}\log(n)}$ and $K(\varepsilon,\delta)$ is some constant that might depends on $\varepsilon$ and $\delta$. Then there exists $n_0$ such that for any $n>n_0$ and any $m_0\leq m\leq m_2$ we have, 
        \[
            m\log(p)-v_m\log(np^2)\leq (1-\gamma)m_0\log(p).
        \] 
\end{lemma}

\begin{proof}
Let $c\in [1/3, 1/2]$ be such that $p=n^{-1+c}$ and recall that $m=\Theta (n^2p^2)$. Applying Corollary \ref{claim_weak_bound} we obtain the following for any $m_0\leq m\leq m_2$ and sufficiently large $n_0$:
\begin{align*}
    m\log(p)-v_m\log(np^2)&\leq (-1+c)m\log(n)+(1-2c)(m-4(\delta-\varepsilon)\E[X]/m+\eta m)\log(n)\\
    &\leq -cm\log(n)-4(1-2c)(\delta-\varepsilon)\E[X]\log(n)/m+{\eta m\log(n)}/{3}.
\end{align*}
Recalling that $m_0=2\sqrt{(\delta-\varepsilon)\E[X]}$ we may rewrite the above as follows provided $n_0$ is large enough:
\begin{align*}
    m\log(p)-v_{m}\log(np^2)&\leq (-cm-(1-2c)m_0^2/m))\log(n)+\eta m\log(n)/3. 
\end{align*}

We claim that $g_c(m)\coloneqq -cm-(1-2c)m_0^2/m\leq (-1+c)m_0$ for all $m_0\leq m\leq m_2$.
Indeed, $g'_c(m)=-c+(1-2c)(m_0/m)^2\leq 1-3c\leq 0$ for $m\geq m_0$ and $c\geq 1/3$.
This concludes the proof as now provided $\gamma$ is small enough we have the following for all $m_0\leq m\leq m_2$:
\begin{align*}
    m\log(p)-v_m\log(np^2)&\leq (-1+c)m_0\log(n)+\eta m\log(n)/3\\
    &= m_0\log(p)+\eta m\log(n)/3\leq  m_0\log(p)+\eta m_2\log(n)/3\\
    &\leq (1-\gamma)m_0\log(p). \qedhere
\end{align*}
\end{proof}
This together with Proposition \ref{claim_second_interval} conclude the proof for the second item in Theorem \ref{thm_main_theorem}.
\subsubsection{The sparser case in the sparse regime}

Assume $\frac{\log^9(n)}{n}\leq p\ll n^{-2/3-\varepsilon}$.
It was already seen that there is a big difference in the behaviour of the upper tail probability depending on the regime $p$ lies at. In this section we present a surprising change of the behaviour in the sparser regime. This phenomenon comes from the fact that the number of core graphs is significant and matters a lot in the evaluation of the upper tail probability. This is completely different from the previous cases where the number of cores was negligible. This will be explained in more details later on.

The essence of the proofs in this subsection is still to make use of the connection between the number of induced copies of $C_4$ in a core graph and the number of vertices in it. To state the claims in this section it is useful to introduce the following notations.

\begin{notation}\label{notation - vertices and edges}
    Suppose $G$ is a graph and $k$ is a nonnegative integer. Then we denote the following sets accordingly:
    \begin{itemize}
        \item Denote by $X_k(G)$ the set of all edges of $G$ with an endpoint of degree $k$ and denote its cardinality by $x_k(G)$.
        \item Denote by $X_{> k}(G)$ the set of all edges of $G$ whose both endpoints have degree greater than $k$; similarly we denote by $x_{>k}(G)$ the cardinality of $X_{>k}(G)$.
    \end{itemize}
     In most cases we omit `$G$' for brevity.
\end{notation}

The following lemma bounds from above the number of induced copies of $C_4$ in any core graph with $m$ edges where $p\ll \frac{1}{\sqrt{n}\log(n)}$.

\begin{lemma}\label{lem_upper_bound_C_4}
    Suppose $\varepsilon,\delta,R,K$ are positive reals and $\frac{\log(n)}{n}\leq p \ll \frac{1}{\sqrt{n}\log(n)}$. Then there exists $n_0$ such that for any $n>n_0$ the following holds:\\
    Let $G$ be a core graph with $m$ edges. Then 
    \begin{equation}\label{eq_main_bound}
        N_{ind}(C_4,G)\leq \sum ^{R}_{i=2}\frac{x_i}{i}\binom{i}{2}\left(\sum_{i<j}\frac{x_j}{j}+\frac{x_i}{2i}\right)+\sum ^{R}_{i=2}\frac{x_ie(G)(i-1)}{R}+\frac{x_{> R}^2}{4},
    \end{equation}  
    and
    \begin{equation}\label{eq_bound_number_vertices}
        v(G)\leq \sum _{i=2}^R \frac{x_i}{i}+\frac{2e(G)}{R}.
    \end{equation}
\end{lemma}

\begin{proof}
Let $G$ be a core graph. First, let us make some notations for the proof. For every nonnegative integer $k$ let $V_k(G)$ be set of all vertices of $G$ with degree $k$ and denote its cardinality by $v_k(G)$. Furthermore, let $V_{> k}(G)$ be the set of all vertices of $G$ with degree greater than $k$ and denote its cardinality by $v_{> k}(G)$. 

We now proceed with the proof. Note that taking $n_0$ large enough according to Lemma~\ref{obs:product_of_deg} applied with $s=R^2+1
$, we are guaranteed that no two vertices of degree less or equal to $R$ can be connected by an edge. This implies that $\{X_i:i\leq R\}\cup\{X_{> R}\}$ is a partition of the edges of $G$. Furthermore, $v_{>R}(G)\leq 2e(G)/R$ and for any $i\leq R$ we have $v_i=x_i/i$. This follows by a standard double counting argument and as the set of all vertices of degree at most $R$ is an independent set. We thus obtain \eqref{eq_bound_number_vertices} as
\[
    v(G)=\sum _{i=2}^{R}v_i(G)+v_{>R}(G)\leq \sum _{i=2}^R \frac{x_i}{i}+\frac{2e(G)}{R}.
\]

To prove \eqref{eq_main_bound} we note that there are three types of induced copies of $C_4$ in core graphs:
\begin{enumerate}[label=(\roman*)]
    \item Induced copies of $C_4$ with exactly two vertices in $\cup _{i=2}^{R}V_i$.
    \item Induced copies of $C_4$ with exactly one vertex in $\cup _{i=2}^{R}V_i$.
    \item Induced copies of $C_4$ with no vertices in $\cup _{i=2}^{R}V_i$ i.e.\ all vertices in $V_{> R}$.
\end{enumerate}
Let us bound from above the number of induced copies of $C_4$ of each type. We claim that the number of induced copies of $C_4$ of the first type is at most 
\begin{align*}
    \sum ^{R}_{i=2}v_i\binom{i}{2}\left(\sum_{i<j}v_j+\frac{v_i}{2}\right)=\sum ^{R}_{i=2}\frac{x_i}{i}\binom{i}{2}\left(\sum_{i<j}\frac{x_j}{j}+\frac{x_i}{2i}\right).
\end{align*}
Indeed, for every $i$, there are at most $\binom{v_i}{2}\binom{i}{2}$ copies of $C_4$ with two vertices in $v_i$ and, for all $i<j$ there are at most $v_iv_j\binom{i}{2}$ copies of $C_4$ with one vertex in $V_i$ and one vertex in $V_j$.

We now bound the number of induced copies of $C_4$ of the second kind. Note that each such induced copy of $C_4$ is determined by choosing its vertex of degree less or equal to $R$, two of its neighbours and another vertex of degree larger than $R$. Therefore, the number of induced copies of $C_4$ of the second kind it at most 
\[
\sum _{i=2}^{R} v_iv_{> R}\binom{i}{2}=\sum _{i=2}^{R}\frac{x_iv_{> R}(i-1)}{2}\leq \sum _{i=2}^{R}\frac{x_ie(G)(i-1)}{R}.
\]

To bound from above the number of induced copies of $C_4$ of the third kind we observe that each such induced copy of $C_4$ is determined by each of the two perfect matchings in it. Since each induced copy of $C_4$ of the third type contains only edges of $X_{>R}$, there are at most $\frac{1}{2}\binom{x_{>R}}{2}\leq \frac{x_{>R}^2}{4}$ such copies. Summing all of these bounds we obtain the assertion of the lemma. 
\end{proof}
Our only use of Lemma \ref{lem_upper_bound_C_4} is when $R=\lceil1/\varepsilon\rceil$. Therefore, form now on we set $R=\left\lceil 1/\varepsilon \right\rceil$. We will only consider core graphs $G$ with at most $m_2$ edges, thus we may bound from above the right-hand side of \eqref{eq_main_bound} by
\begin{align*}
    f(x_1,x_2,\ldots,x_{R},x_{> R})=&\sum ^{R}_{i=2}\binom{i}{2}\frac{x_i}{i}\left(\sum_{i<j}\frac{x_j}{j}+\frac{x_i}{2i}\right)+\varepsilon m_2\sum ^{R}_{i=2}{(i-1)x_i}+\frac{x_{> R}^2}{4}\\
    =&\sum ^{R}_{i=2}{(i-1)x_i}\left(\sum_{i<j\leq R}\frac{x_j}{2j}+\frac{x_i}{4i}+\varepsilon m_2\right)+\frac{x_{> R}^2}{4}.
\end{align*}
In particular, for every core graph $G$ with at most $m_2$ edges, 
\[
    (\delta-\varepsilon)\E [X]\leq N_{ind}(C_4,G)\leq  f(x_1(G),x_2(G),\ldots,x_{R}(G),x_{> R}(G)).
\]
Recall that in this range of $p$ we have $\log(np^2)<0$ and note that by Lemma \ref{lem_upper_bound_C_4} for any core graph $G$ the term $e(G)\log(p)-v(G)\log(np^2)$ can be bounded from above using $x_i(G)$ by
\begin{align*}
    \sum _{i=2}^{R}\left(x_i(G)\log(p)-x_i(G)\log(np^2)/i\right)+x_{>R}(G)\log(p)-2e(G)\log(np^2)/R.
\end{align*}
Furthermore, for $\varepsilon <\eta /2$ we have,
\[
    2e(G)/R\leq 2e(G)\varepsilon\leq \eta e(G).
\]
Therefore, letting $n_0$ sufficiently large we find that,
\begin{align}\label{eq_extended_bound_on_entropic_term}
    \max_{m_0\leq m\leq m_2}m\log(p)-v_m\log(np^2)\leq& \max_{G\in \bigcup_{m=m_0}^{m_2} \mathcal{C}_m} \sum _{i=2}^{R}\left(x_i(G)\log(p)-x_i(G)\log(np^2)/i\right) \nonumber \\
                                                        &+x_{>R}(G)\log(p)+\eta m_2\log(n).
\end{align}
We now restate this bound in a more compact way. To this end we introduce the following notations. For every graph $G$ define the following:
\begin{itemize}
    \item $x(G)\in \mathbb{R}^{R+1}$ is the vector defined by $x_i(G)$ in its $i$-th coordinate for $1\leq i\leq R$ and $x_{>R}(G)$ in the last coordinate.
    \item $u\in \mathbb{R}^{R+1}$ is defined by $0$ as the first coordinate, $\log(p)-\log(np^2)/i$ as the $i$-th coordinate for $2\leq i\leq R$, and $\log(p)$ in the last coordinate.
\end{itemize}
Using these notations we may rewrite \eqref{eq_extended_bound_on_entropic_term} as follows,
\begin{align} \label{eq_entropic_bound}
    \max_{m_0\leq m\leq m_2}m\log(p)-v_m\log(np^2)\leq \max_{G\in \cup _{m=m_0}^{m_2}\mathcal{C}_m}\langle x(G),u \rangle+\eta m_2\log(n).
\end{align}
From now on we let $t=(\delta-\varepsilon)\E[X]$. Since $t\leq  f(x(G))$ for every $G\in \bigcup_{m=m_0}^{m_2}\mathcal{C}_m$,
\begin{align*}
    \max_{G\in \cup _{m=m_0}^{m_2}\mathcal{C}_m} \langle x(G),u \rangle \leq \max\{\langle x,u \rangle : x\in \mathbb{R}_{\geq 0}^{R+1} \land f(x)\geq t\}.
\end{align*}

Note that in the above we dropped the condition that $x$ came from a graph. By this we only consider more possible cases and thus obtain an upper bound on the maximization problem.

Now we are ready to state the main technical proposition in order to bound from above the term $m\log(p)-v_m\log(np^2)$ when $\frac{\log^9(n)}{n}\leq p\leq n^{-2/3-\varepsilon}$.

\begin{prop}\label{claim_conclusion_c_less_1/3}
    Suppose $k>1$ is an integer and suppose that $\eta,\varepsilon,\delta$ are positive reals with $\varepsilon$ small enough as a function of $\eta$ and $k$. Then, there exists $n_0$ such that for any $n>n_0$ and any $n^{-1+c_{k-1}} \leq p \leq n^{-1+c_{k}}$ we have 
        \[
            \max\{\langle x,u \rangle : x\in \mathbb{R}_{\geq 0}^{R+1} \land f(x)\geq t\} \leq (1-\eta)m_{k} u_{k}.
        \]
\end{prop}

This proposition and Proposition \ref{claim_second_interval} imply the following corollary. This corollary yields the first item in Theorem \ref{thm_main_theorem} as explained before.

\begin{cor}\label{lemma_when_c<1/3}
        Suppose $k>1$ is an integer and suppose that $\eta,\varepsilon,\delta,K$ are positive reals with $\varepsilon$ small as a function of $\eta$ and $k$. Then, there exists $n_0$ such that for any $n>n_0$ and any $\max \left\{ \frac{\log^9(n)}{n},n^{-1+c_{k-1}}\right\}\leq p\leq n^{-1+c_{k}}$ we have the following: 
        \[
            m\log(p)-v_m\log(np^2) \leq (1-\eta)m_k u_{k}=(1-\eta)m_k(\log(p)-\log(np^2)/k)
        \]
        for every $m_0\leq m\leq K\Phi_X(\delta+\varepsilon)$.
\end{cor}

We now turn to the proof of Proposition \ref{claim_conclusion_c_less_1/3}. We show that for $\max\left \{\frac{\log^9(n)}{n},n^{-1+c_{k-1}}\right\} \leq p \leq  n^{-1+c_{k}}$ the maximum of $\langle x,u \rangle$ is achieved by a vector $x$ of the form $\alpha \cdot e_k$ where $e_k$ is the $k$-th element in the standard basis. The strategy of the proof is to think of the vector $x$ as a distribution vector and push its mass towards the `center of mass' while keeping $\langle x,u \rangle$ constant and ensuring that $f$ evaluated on the vector after the transformation is still greater than $t$. To be more precise, the center of mass will be the $k$-th coordinate such that $k$ satisfies  $ \max \left\{ \frac{\log^9(n)}{n},n^{-1+c_{k-1}}\right\}\leq p\leq n^{-1+c_{k}}$. 
This will be done iteratively by showing that if $x$ is a vector achieving the maximum of $\langle x,u \rangle$ and $x_i=0$ for all $i\leq j<k$ then we may define a vector $\Tilde{x}$ which also achieves this maximum and satisfies $\Tilde{x_i}=0$ for all $i\leq j+1$. A similar statement will be shown for the other direction i.e.\ pushing the mass from the right to the left.

Let us introduce the following notations of pushing mass. Let $2\leq k\leq R$ be some integer (should be thought of as the center of mass) and suppose that $x\in \mathbb{R}^{R+1}$. Then for any integers $2 \leq i<k< j\leq R+1$ define the following vectors,
\[
    x^{(i,k)}=x-x_ie_i+\frac{u_{i}x_i}{u_{i+1}}e_{i+1},
\]
\[
    x^{(j,k)}=x-x_je_j+\frac{u_{j}x_j}{u_{j-1}}e_{j-1}.
\]
One should think of these operations as pushing the mass towards the $k$-th coordinate while staying on the hyperplane defined by $\langle x,u \rangle=s$ for some constant $s$.

\begin{proof}[Proof of Proposition \ref{claim_conclusion_c_less_1/3}]
As explained earlier we prove this proposition iteratively. We do so in several claims. First, we show that we may assume that the distribution vector $x$ satisfies $x_{R+1}=0$. This is given in the following claim.

\begin{claim}\label{lem:no_last_coordinate}
    If $p \leq n^{-2/3-\eta}$ for some positive $\eta$, then provided $\varepsilon$ is small enough there exists $n_0$ such that for any $n>n_0$ the following holds:
    \[
        \max\{\langle x,u \rangle : x\in \mathbb{R}_{\geq 0}^{R+1} \land f(x)\geq t\}=\max\{\langle x,u \rangle : x\in \mathbb{R}_{\geq 0}^{R+1} \land x_{R+1}=0\land f(x)\geq t\}.
    \]
\end{claim}

\begin{proof}
Let 
\[
        s=\max\{\langle x,u \rangle : x\in \mathbb{R}_{\geq 0}^{R+1} \land f(x)\geq t\}.
\] 
Furthermore, let $x\in \mathbb{R}_{\geq 0}^{R+1}$ be such that $s=\langle x,u \rangle$ and $f(x)\geq t$. By the definition of $x^{(R+1,k)}$ we have, 
\[
    \langle x,u \rangle=\langle x^{(R+1,k)},u \rangle
\]
and therefore, $s=\langle x^{(R+1,k)},u \rangle$.
In order to prove the lemma it is sufficient to prove $f(x^{(R+1,k)})\geq f(x)$. To do so let 
\[
    f_1(x_1,x_2,\ldots,x_{R+1})=\sum ^{R-1}_{i=2}(i-1){x_i}\left(\sum_{i<j\leq R-1}\frac{x_j}{2j}+\frac{x_i}{4i}+\varepsilon m_2\right),
\]
which depends only on $x_1,x_2,\ldots,x_{R-1}$, and 
\[
    f_2(x_1,x_2,\ldots,x_{R+1})=\sum ^{R-1}_{i=2}\frac{(i-1)x_ix_R}{2R}+(R-1)x_R\left(\frac{x_R}{4R}+\varepsilon m_2\right)+\frac{x_{R+1}^2}{4}.
\]
Note that for all $y\in \mathbb{R}^{R+1}$ we have
\[
    f(y)=f_1(y)+f_2(y).
\]
Note also that for all $1\leq i\leq R-1$ we have $x^{(R+1,k)}_i=x_i$.
Therefore, we obtain that $f(x^{(R+1,k)})-f(x)=f_2(x^{(R+1,k)})-f_2(x)$.
As $x_R^{(R+1,k)}=x_R+\frac{u_{R+1}}{u_R}x_{R+1}$ and $x^{(R+1,k)}_{R+1}=0$ a straightforward computation gives
\begin{align*}\label{eq_difference_in_f}
    f(x^{(R+1,k)})-f(x)&=\left(\sum ^{R}_{i=2}\frac{(i-1)x_i}{2R}+\varepsilon m_2 (R-1)\right)\frac{u_{R+1}}{u_{R}}x_{R+1}\\
    &\quad+\left(\frac{(R-1)u^2_{R+1}}{Ru^2_{R}}-1\right)\frac{x_{R+1}^2}{4}.
\end{align*}
Note that $u_{R+1}=\log(p)<0$ and $u_{R}<0$. This holds as $u_R=\log(p)-\log(np^2)/R$ and we assume $\varepsilon$ is small enough so that $p^{R-2}\leq n$, as then the left-hand side decays to zero and the right-hand side approaches infinity. This implies that 
\[
    \frac{u_{R+1}}{u_{R}}>0.
\]
Therefore, 
\[
    f(x^{(R+1,k)})-f(x)\geq \left(\frac{(R-1)u^2_{R+1}}{Ru^2_{R}}-1\right)\frac{x^2_{R+1}}{4}.
\]
Proving that the above is nonnegative provided $n_0$ is sufficiently large will conclude the proof. Indeed, letting $y=\frac{\log(np^2)}{R\log(p)}$
\[
    \frac{(R-1)u^2_{R+1}}{Ru^2_{R}}=\frac{(R-1)\log^2(p)}{R(\log(p)-\log(np^2)/R)^2}\geq 1
\]
is equivalent to,
\[
    -y^2+2y-1/R\geq 0,
\]
which is also equivalent to
\[
    1-\sqrt{1-1/R}\leq y\leq 1+\sqrt{1-1/R}.
\]
 Since $1-\sqrt{1-x}\leq x/2+x^2$ and $1+\sqrt{1-x}\geq 2-x$ for all $0\leq x\leq1$, it is enough to prove that
\[
   2/R+1/R^2 \leq \frac{\log(np^2)}{R\log(p)}\leq 2-1/R.
\]
First, as $\log(p)<0$, the right inequality is equivalent to 
\[
    \log(np^2)\geq (2R-1)\log(p).
\]
Equivalently, $\log(n)\geq (2R-3)\log(p)$, which holds for $\varepsilon$ satisfying $2R-3\geq 0$ as then the left-hand side is positive for $n>1$ and the right-hand side tends to minus infinity as $p$ tends to infinity.
Second, as $np^2\ll 1$, the left inequality is equivalent to 
\[
    (1/2+1/R)\log(p)\geq \log(np^2).
\]
This holds if and only if $p\leq n^{-1/(3/2-1/R)}$. Finally, assuming that $\varepsilon$ is small enough so that $1/R\leq 3/2-1/(2/3+\eta)$ implying the claim.
\end{proof}
We now continue to the more technical part of the proof which is to push the mass to the $k$-th coordinate provided that the last coordinate is $0$, which we may assume due to the above claim.

\begin{claim}\label{claim_the_right_center}
    If  $n^{-1+c_{k-1}}\leq p\leq n^{-1+c_{k}} $ then there exists $n_0$ such that for any $n>n_0$ the following holds:
        \[
        \max\{\langle x,u \rangle : x\in \mathbb{R}_{\geq 0}^{R+1} \land f(x)\geq t\}=\max\{\alpha u_k:f(\alpha e_k)\geq t\}.
        \]
\end{claim}

\begin{proof}
Let
\[
    s=\max\{\langle x,u \rangle : x\in \mathbb{R}_{\geq 0}^{R+1} \land f(x)\geq t\}.
\] 
Furthermore, let $x\in \mathbb{R}_{\geq 0}^{R+1}$ be such that $s=\langle x,u \rangle$ and $f(x)\geq t$. By Claim \ref{lem:no_last_coordinate} we may assume $x_{R+1}=0$. As $x$ satisfies $\langle x,u \rangle=s$ we also have $\langle x^{(i,k)},u \rangle=\langle x^{(j,k)},u \rangle=s$ for any $2 \leq i<k<j\leq R$.

We will start with the push of the mass from the left to the right. To this end we prove iteratively the following:
Let $x\in \mathbb{R}_{\geq 0}^{R+1}$ and suppose that $x_r=0$ for all $1\leq r<i<k$. Then $f\left(x^{(i,k)}\right)-f(x)\geq 0$ provided $n_0$ is large enough. To see this let 

\begin{align*}
    f_1(x_{1},x_{2},\ldots,x_{R+1})=\sum ^{R}_{\ell =i+2}{(\ell-1)x_\ell}\left(\sum_{\ell<j\leq R}\frac{x_j}{2j}+\frac{x_\ell}{4\ell}+\varepsilon m_2\right),
\end{align*}
which depends only on $x_{i+2},x_{i+3},\ldots, x_{R}$, and let
\begin{align*}
    f_2(x_1,x_{2},\ldots,x_{R+1})=\sum ^{i+1}_{\ell=i}{(\ell-1)x_\ell}\left(\sum_{\ell<j\leq R}\frac{x_j}{2j}+\frac{x_\ell}{4\ell}+\varepsilon m_2\right),
\end{align*}
which depends only on $x_i,x_{i+1},\ldots,x_R$.
Note that for all $y$ with $y_r=0$ for $r<i$ we have $f(y)=f_1(y)+f_2(y)$. This implies that $f(x^{(i,k)})-f(x)=f_2(x^{(i,k)})-f_2(x)$. By the definition of $x^{(i,k)}$ we have
\begin{align*}
    f\left(x^{(i,k)}\right)-f(x)=&\left(\sum _{i+1 \leq j}\frac{x_ix_j}{2j}+\varepsilon m_2 x_i\right)\left(\frac{iu_i}{u_{i+1}}-(i-1)\right)+\frac{x_i^2}{4}\left(\frac{iu^2_i}{(i+1)u^2_{i+1}}-\frac{i-1}{i}\right).
\end{align*}
Thus, to show that $f\left(x^{(i,k)}\right)-f(x)\geq 0$ it is suffices to show the following:
\begin{align}\label{eq_the_c_k_term_1}
    \frac{iu_i}{u_{i+1}}-(i-1)\geq 0,
\end{align}
and
\begin{align}\label{eq_the_c_k_term}
    \frac{iu^2_i}{(i+1)u^2_{i+1}}-\frac{i-1}{i}\geq 0.
\end{align}
As $np^2=o(1)$ we have $0>u_{i}>u_{i+1}$. Therefore, inequality \eqref{eq_the_c_k_term_1} is equivalent to
\begin{align*}
    i\log(p)-\log(np^2)\leq (i-1)\log(p)-\frac{i-1}{i+1}\log(np^2).
\end{align*}
By algebraic manipulation the above is equivalent to
\begin{align*}
    \log(p)\leq \frac{2}{i+1}\log(np^2).
\end{align*}
This always holds in our settings provided $n$ is sufficiently large. Indeed, the above is equivalent to $p^{i-3}\leq n^{2}$ which holds as $n^{-1}\leq p\ll 1$.
It is left to show inequality \eqref{eq_the_c_k_term}. We prove this in the following claim.

\begin{claim}\label{claim_technical}
    \[
        i^2u_i^2\geq (i+1)(i-1)u^2_{i+1} \iff p\geq n^{-1+c_i}. 
    \]
\end{claim}

\begin{proof}
The inequality in the left-hand side is equivalent to $iu_i\leq \sqrt{(i+1)(i-1)}u_{i+1}$ as $0>u_{i}>u_{i+1}$.
Recall that $u_\ell =\log(p)-\log(np^2)/\ell$ and plug this in the above inequality to obtain
\[
    i\log(p)-\log(np^2)\leq \sqrt{{(i-1)}{(i+1)}}\log(p)-\sqrt{\frac{i-1}{i+1}}\log(np^2).
\]
Equivalently, 
\[
    \left(i-2-\sqrt{(i+1)(i-1)}+2\sqrt{\frac{i-1}{i+1}}\right)\log(p)\leq \left(1-\sqrt{\frac{i-1}{i+1}}\right)\log(n)
\]
Thus we conclude that our assumption is equivalent to the following:
\[
    \log(p)\geq \frac{\sqrt{i+1}-\sqrt{i-1}}{(i-2)\sqrt{i+1}-(i-1)^{3/2}}\log(n)=\left(\frac{1}{2 + \sqrt{\frac{i+1}{i-1}}}-1\right)\log(n)=(-1+c_i)\log(n).\qedhere
\]
\end{proof}

Since we assumed $ p\geq n^{-1+c_{k-1}}$ and for all $i<k$ we have $ c_{i}\leq c_{k-1}$ we also have $p\geq n^{-1+c_{k-1}}\geq n^{-1+c_{i}}$. Therefore, for any $x\in \mathbb{R}^{R+1}$ we may iterate this process for $2\leq i\leq k-1$ and obtain that $f\left(x^{(2,k)(3,k)\ldots (k-1,k)}\right)\geq f(x)$. This conclude the push of the mass of $x$ from the $i$-th coordinates for $i<k$ to the $k$-th coordinate.

Now provided what we showed so far we prove a similar statement for $x^{(j,k)}$. More specifically, we prove iteratively the following: Let $x\in \mathbb{R}^{R+1}$ and suppose that $x_r=0$ for all $k<j<r\leq R+1$ and $x_\ell=0$ for all $\ell<k$. Then $f\left(x^{(j,k)}\right)-f(x)\geq 0$ provided $n_0$ is large enough. To see this let 

\begin{align*}
    f_3(x_1,x_2,\ldots,x_{R+1})=\sum ^{j-2}_{i=k}{(i-1)x_i}\left(\sum_{i<\ell\leq j-2}\frac{x_\ell}{2\ell}+\frac{x_i}{4i}+\varepsilon m_2\right),
\end{align*}
which depends only on $x_i$ for $k\leq i\leq j-2$, and let 
\begin{align*}
    f_4(x_1,x_2,\ldots,x_{R+1})=&\sum ^{j-2}_{i=k}{(i-1)x_i}\sum_{\ell=j-1}^{j}\frac{x_\ell}{2\ell}+\sum ^{j}_{i=j-1}{(i-1)x_i}\left(\sum_{i<\ell\leq j}\frac{x_\ell}{2\ell}+\frac{x_i}{4i}+\varepsilon m_2\right),
\end{align*}
which depends only on $x_i$ for $k\leq i\leq j$.

Note that for all $y$ with $y_i=0$ for all $i\leq k$ or $i\geq j+1$ we have $f(y)=f_3(y)+f_4(y)$. Note also that $x^{(j,k)}_i=x_i$ for all $i\leq j-2$ thus, $f(x^{(j,k)})-f(x)=f_4(x^{(j,k)})-f_4(x)$.
This implies that
\begin{align*}
    f\left(x^{(j,k)}\right)-f(x)=&\sum _{i=k}^{j-1}\frac{(i-1)x_ix_j}{2}\left(\frac{u_j}{u_{j-1}(j-1)}-\frac{1}{j}\right)+\frac{x_j^2}{4}\left(\frac{(j-2)u_{j}^2}{(j-1)u_{j-1}^2}-\frac{j-1}{j}\right).
\end{align*}
Thus, to show that $f\left(x^{(j,k)}\right)-f(x)\geq 0$ it is suffices to show the following:
\[
    \frac{u_j}{u_{j-1}}\geq \frac{j-1}{j},
\]
and
\begin{align*}
    \frac{u^2_j}{u^2_{j-1}}\geq \frac{(j-1)^2}{j(j-2)}.
\end{align*}
Note that these two inequalities are respectively equivalent to the following inequalities:
\begin{align}\label{eq_first_cond}
    j\log(p)-\log(np^2)\leq(j-1)\log(p)-\log(np^2),
\end{align}
\begin{align}\label{eq_second_cond}
    \frac{u^2_{j-1}}{u^2_{j}}\leq \frac{j(j-2)}{(j-1)^2}.
\end{align}
Inequality \eqref{eq_first_cond} holds always as $p<1$. Moreover, by Claim \ref{claim_technical} inequality \eqref{eq_second_cond} holds if and only if $p\leq n^{-1+c_{j-1}}$.This indeed holds as we assume $p\leq n^{-1+c_{k}}$ and we also have $c_{k}\leq c_{j-1}$ for all $j> k$. This conclude the proof of the lemma.
\end{proof}

Now we deduce the proposition from the above claims. Assuming $n_0$ is large enough and using Claim \ref{claim_the_right_center} it is sufficient to bound from above
\[
   s\coloneqq \max\{\alpha u_k:f(\alpha e_k)\geq t\}.
\]
Let $\alpha\in \mathbb{R}$ be a witness for that. We have, 
\begin{align*}
    t \leq f(\alpha e_k) &= \alpha ^2 \frac{k-1}{4k}+\varepsilon \alpha (k-1) m_2\\
    &= \frac{(\delta+\varepsilon)\E[X]\alpha^2}{m_k^2}+2\varepsilon\alpha(k-1)\sqrt{(\delta+\varepsilon)\E[X]}\\
    &=\frac{\delta+\varepsilon}{\delta-\varepsilon}\cdot \left(\left(\frac{\alpha}{m_k}\right)^2t+4\varepsilon\sqrt{k(k-1)}\left( \frac{\alpha}{m_k}\right)t\right).
\end{align*}
Denoting $x=\alpha/m_k$ implies 
\[
    x\geq -2\varepsilon\sqrt{k(k-1)}+2\sqrt{\varepsilon^2k(k-1)+\frac{\delta-\varepsilon}{\delta+\varepsilon}} \text{ or }x\leq -2\varepsilon\sqrt{k(k-1)}-2\sqrt{\varepsilon^2k(k-1)+\frac{\delta-\varepsilon}{\delta+\varepsilon}}.
\]
Hence, for every positive $\eta$ and sufficiently small $\varepsilon$ we obtain
\[
    x\geq 1-\eta \text{ or }x\leq -1+\eta.
\] 
The second option is not possible as $\alpha$ is nonnegative. 
Thus we have, $\alpha/m_k\geq 1-\eta$ and hence
\[
    s=\alpha u_k\leq (1-\eta)m_ku_k.
\]
This is as claimed.
\end{proof}

\section{The solution to the variational problem}\label{sec_7}

In this section we solve the naive mean field variational problem (when $\sqrt{\log(n)}/n\ll p\ll n^{-1/2-\varepsilon}$) for the upper tail of the number of induced copies of $C_4$ in $G_{n,p}$. Let us recall the definition of the variational problem.

Let $N$ be a positive integer and let $Y=(Y_1,Y_2,\ldots, Y_N)$ be a sequence of independent Bernoulli random variables with mean $p$. Further, let $X$ be a function from the hypercube to $\mathbb{R}$. Then the naive mean field variational problem associated to the upper tail of $X(Y)$ is the function $\Psi_X\colon \mathbb{R} _{\geq 0}\to \mathbb{R}_{\geq 0}\cup\{ \infty \}$ such that for every $\delta\geq 0$:
    \[
        \Psi_{X}(\delta)=\inf\left\{\sum_{i\in N}I_p(z_i):\E[X(Z)]\geq (1+\delta)\E[X(Y)]\right\},
    \]
    where $I_{p}(q)=D_{KL}(\Ber(q)\|\Ber(p))=q\log\left(\frac{q}{p}\right)+(1-q)\log\left(\frac{1-q}{1-p} \right)$, and the infimum is taken over all $Z=(Z_1,Z_2,\ldots,Z_N)$ sequences of $N$ Bernoulli random variables with means $z_i$ respectively.

We will be interested in the case where $N=\binom{n}{2}$, the random variables $Y_i$ correspond to the edges of $G_{n,p}$, and $X$ is the random variable counting the number of induced copies of $C_4$. The main proposition of this section is the following.

\begin{prop}\label{prop_solution_to_var_problem}
    Suppose $\varepsilon,\delta$ are positive reals. Further, suppose that $\sqrt{\log(n)}/n\ll p\ll n^{-1/2-\varepsilon}$. Then, for large enough $n$, we have 
    \[
        (1-\varepsilon) \sqrt{\frac{\delta}{2}}\leq \frac{\Psi_{X}(\delta)}{n^2p^2\log(1/p)}\leq (1+\varepsilon) \sqrt{\frac{\delta}{2}}.
    \]
\end{prop}

The upper bound follows immediately from
\[
    \Psi_X(\delta)\leq \log(1/p)\min\{e(C):\E[X\mid C\subseteq G_{n,p}]\geq (1+\delta)\E[X]\}.
\]
In Section \ref{sec_4} we this minimum. It is attained when $C$ is the complete bipartite graph $K_{m_0,m_0}$ where $m_0^2\approx \sqrt{\frac{\delta}{2}}n^2p^2$. For more details the reader is referred to Section \ref{sec_4}. 

For the lower bound we start with the following reduction.

\begin{lemma} \label{lem_q_i>p}
    Suppose $\varepsilon,\delta$ are positive reals. Furthermore, suppose $p\ll 1$. Then, for large enough $n$, we have
    \begin{align*}
        \Psi_X(\delta)\geq \inf_{\bar{q}}\left\{\sum_{i\in \binom{[n]}{2}}I_p(q_i):\E[N_{ind}(C_4,G_{n,\bar{q}})]\geq (1+\delta-\varepsilon)\E[X]\text{ and } \forall i\in \binom{[n]}{2}\, q_i\geq p \right\}.
    \end{align*}
\end{lemma}

To prove the lemma we define the set $\mathcal{C}_4$ as follows:
Let $\mathcal{C}_4$ be a set of representatives for the set of all tuples $(x,y,z,w)\in [n]^4$ with distinct coordinates and the equivalence relation, $(x,y,z,w)\sim(x',y',z',w')$ if and only if $\{xy,yz,zw,wx\}=\{x'y',y'z',z'w',w'x'\}$; i.e.\ both $(x,y,z,w)$ and $(x',y',z',w')$ induce the same 4-cycle in $K_n$.

\begin{proof}
    Let $\bar{q}\in [0,1]^{\binom{n}{2}}$ be such that $\E[N_{ind}(C_4,G_{n,\bar{q}})]\geq (1+\delta)\E[X]$. Define $\bar{r}$ by setting $r_i=\max\{p,q_i\}$. Since $I_p(x)$ is monotone decreasing between $0$ and $p$, we obtain that
    \[
        \sum I_{p}(q_i)\geq \sum I_{p}(r_i).
    \]
    To conclude the proof, we prove that $\E[N_{ind}(C_4,G_{n,\bar{r}})]\geq (1+\delta-\varepsilon)\E[X]$. Indeed, provided $n$ is sufficiently large, we have
    \begin{align*}
        \E[N_{ind}(C_4,G_{n,\bar{r}})] &= \sum_{(x,y,z,w)\in\mathcal{C}_4} r_{xy}r_{yz}r_{zw}r_{wx}(1-r_{xz})(1-r_{yw})\\
                                       &\geq \sum_{(x,y,z,w)\in\mathcal{C}_4} q_{xy}q_{yz}q_{zw}q_{wx}(1-q_{xz})(1-q_{yw})(1-p)^2\\
                                       & \geq (1+\delta)(1-p)^2\E[X]\geq (1+\delta-\varepsilon)\E[X]. \qedhere 
    \end{align*}    
\end{proof}

Next we present a lemma about the structure of $G_{n,\bar{q}}$ when $\sum I_{p}(q_i)$ is close to the infimum in Lemma \ref{lem_q_i>p}.

\begin{restatable}{lemma}{lemmaone}
    \label{lem_only_C_4}
    Suppose $\varepsilon,\delta$ are positive reals with $\varepsilon$ sufficiently small. Suppose also that $\sqrt{\log(n)}/n \ll p\ll n^{-1/2}$. Then, for any sequence $\bar{q}$ satisfying:
    \begin{enumerate}
        \item $q_i=p+u_i$ for some $u_i\geq 0$,
        \item $\sum_{i\in \binom{[n]}{2}}I_p(q_i) \leq \sqrt{2\delta}n^2p^2\log(1/p)$, and
        \item $\E[N_{ind}(C_4,G_{n,\bar{q}})]\geq (1+\delta-\varepsilon)\E[X]$,
    \end{enumerate}  
    we have $\E[N_{ind}(C_4,G_{n,\bar{u}})]\geq (\delta-2\varepsilon)\E[X]$ provided $n$ is large enough.
\end{restatable}

This lemma is a variant of the methods in the papers of Lubetzky and Zhao \cite{LubZha17} and in Bhattacharya, Ganguly, Lubetzky and Zhao \cite[Section 5.2]{BhaGanLubZha17}. In these papers the authors use the language of graph limits or `graphons'; we recreate their proof in the language of graphs in the Appendix.
Further, Lubetzky and Zhao \cite{LubZha17} and Bhattacharya, Ganguly, Lubetzky and Zhao \cite{BhaGanLubZha17}, analysed the function $I_p(p+x)$ and proved several lemmas in the language of graphons. We use these lemmas in the proof of Proposition \ref{prop_solution_to_var_problem}, once again in the language of graphs and not graphons. For a proof of the first lemma we refer the reader to \cite{LubZha17} and \cite{BhaGanLubZha17}; the second lemma will be proven in the Appendix with some extensions.

\begin{restatable}[\cite{LubZha17}]{lemma}{lemmatwo}\label{lem_exact_estimations}
    The following holds, 
    \[
       I_{p}(p+x)=\begin{cases}
                    (1+o(1))\frac{x^2}{2p}, & 0\leq x\ll p,\\
                    
                  (1+o(1))x\log(x/p), & p\ll x\leq 1-p.  
                  \end{cases}
    \] 
\end{restatable}



\begin{restatable}[\cite{LubZha17}]{lemma}{lemmathree}\label{lem_estimation of B}
    Suppose $\varepsilon,\delta,C$ are positive reals and $\sqrt{\log(n)}/n\ll p\ll n^{-1/2}$. Suppose also that $\bar{u}\in [0,1-p]^{\binom{n}{2}}$ such that $\sum_{i\in \binom{[n]}{2}}I_{p}(p+u_i)\leq Cn^2p^2\log(1/p)$.
    Let $b=b(n)$ be such that $\max\{np^2,\sqrt{p\log(1/p)}\}\ll b\leq 1-\varepsilon$. Then, provided $n$ is large enough there is a constant $D>0$ such that the following holds:
        \begin{enumerate}[label=(\roman*)]
            \item \label{eq_estimation_2_pre_Appendix} $\sum_{x\in [n]}\left(\sum_{y\neq x}u_{xy}\right)^2\leq Dn^{3}p^{2}b$ and
            \item \label{eq_estimation_3_pre_Appendix} $\sum_{i\in \binom{[n]}{2}}u_i\leq Dn^2 p^{3/2}\sqrt{\log(1/p)}$.
        \end{enumerate}
\end{restatable}

Now we are ready to prove Proposition \ref{prop_solution_to_var_problem}.

\begin{proof}[Proof of Proposition \ref{prop_solution_to_var_problem}]
    We wish to prove the following:
    \[
        \sum_{i\in \binom{[n]}{2}} I_{p}(q_i)\geq (1-\eta)\sqrt{\frac{\delta}{2}}n^2p^2\log(1/p)
    \]
    for sufficiently small $\eta>0$ and for $\bar{q}\in [0,1]^{\binom{n}{2}}$ such that 
    \[
            \E[N_{ind}(C_4,G_{n,\bar{q}})]\geq (1+\delta)\E[X].
    \]
    We may assume that
    \[
        \sum_{i\in \binom{[n]}{2}} I_{p}(q_i) \leq  Cn^2p^2\log(1/p),
    \]
    for some absolute constant $C>0$ as otherwise there is nothing to prove.
    By Lemma \ref{lem_q_i>p} and Lemma \ref{lem_only_C_4} we may assume that $q_i=p+u_i$ for some $u_i\geq 0$ which satisfies
    \[
        \E[N(C_4,G_{n,\bar{u}})] \geq \E[N_{ind}(C_4,G_{n,\bar{u}})]\geq (\delta-\varepsilon)\E[X]
    \]
    where $N(C_4,G_{n,\bar{u}})$ is the number of copies of $C_4$ in $G_{n,\bar{u}}$. Note that
    \begin{align*}
         \E[N(C_4,G_{n,\bar{u}})]&={\sum_{(x,y,z,w)\in \mathcal{C}_4}u_{xy}u_{yz}u_{zw}u_{wx}}\geq (\delta-\varepsilon)\E[X].
    \end{align*}
    In the next three claims we collect some properties of $(x,y,z,w)\in \mathcal{C}_4$ with non-negligible contribution to the sum above. Afterwards, we use these properties to show that the set of all 4-cycles with non-negligible contribution to the sum above is a subset of 
    \[
        \mathcal{C}_4^{\gamma}\coloneqq\{(a_0,a_1,a_2,a_3)\in \mathcal{C}_4:u_{a_ia_{i+1}}> \gamma \text{ for all } i=0,1,2,3\},
    \]
    where the summation is modulo $4$ and $\gamma$ is some positive constant depending on $\varepsilon$.
    We now explain briefly the proof strategy:
    First, we show that to have a non-negligible contribution a 4-cycle $(a_0,a_1,a_2,a_3)\in \mathcal{C}_4$ must contain a $K_{1,2}$ with both edge weights at least $\sqrt{p}/\log(1/p)$; for convenience assume that these edges are $a_0a_1,a_1a_2$. Afterwards, we show that, among such 4-cycles, the only ones contributing a non-negligible amount must additionally satisfy $u_{a_2a_3}\geq p^{1-2\varepsilon}$ or $u_{a_3a_0}\geq  p^{1-2\varepsilon}$, assume $u_{a_3a_0}\geq  p^{1-2\varepsilon}$. Then, we prove that a 4-cycle with the above properties contributes a non-negligible amount only if $u_{a_0a_1},u_{a_2a_3}\geq \gamma$ (in the complementary case we show $u_{a_1a_2},u_{a_3a_1}\geq \gamma$). Using symmetries of the 4-cycle, we may use this argument again from a different point of view. We then obtain that the only 4-cycles contributing a non-negligible amount to the sum are ones with $u_{a_ia_{i+1}}\geq \gamma$ for all $i$. This explanation can also be viewed in Figure \ref{fig_C_4_flow}.

    \begin{figure}[htb]

  \begin{subfigure}[b]{0.2\linewidth}
    \begin{tikzpicture}
    \node [align=right] at (0.3,2.7) {Densities:};
    \fill[fill=black] (-0.4,0) circle (0.1);
	\fill[fill=black] (0.6,0) circle (0.1);
	\draw [line width=1.1] (-0.4,0) -- (0.6,0);
	\node [align=right] at (1.1 ,0) {$\ge \gamma$};
	\fill[fill=black] (-0.4,1) circle (0.1);
	\fill[fill=black] (0.6,1) circle (0.1);
	\draw[dashed, line width=1.1] (-0.4,1) -- (0.6,1);
	\node [align=right] at (1.6 ,1) {$\ge \frac{\sqrt{p}}{\log (1/p)}$};
	\fill[fill=black] (-0.4,2) circle (0.1);
	\fill[fill=black] (0.6,2) circle (0.1);
	\draw[loosely dotted, line width=1.1] (-0.4,2) -- (0.6,2);
	\node [align=right] at (1.5 ,2) {$\ge p^{1-2\varepsilon}$};
    \end{tikzpicture}%
  \end{subfigure}
\begin{subfigure}[b]{0.7\linewidth}
  \begin{tikzpicture}  
\def \r{2.2};
\def \t{1};
\node (a) at (\t + 0,\r + 1) {$a_0$};
\node (b) at (\t + 1.15,\r + 1) {$a_1$};
\node (c) at (\t + 1.15,\r + 0) {$a_2$};
\node (d) at (\t + 0, \r + 0) {$a_3$};
\draw [-to, line width=1.5] (\t + 2, \r + 0.5)-- node [text width=2.5cm,midway,above=0.2em,align=center ]{Claim \ref{claim_7.6}} (\t + 3.15, \r + 0.5);

\def \t{5};
\node (a) at (\t + 0,\r + 1) {$a_0$};
\node (b) at (\t + 1.15,\r + 1) {$a_1$};
\node (c) at (\t + 1.15,\r + 0) {$a_2$};
\node (d) at (\t + 0, \r + 0) {$a_3$};
\draw [-to, line width=1.5] (\t + 2, \r + 0.5)-- node [text width=2.5cm,midway,above=0.2em,align=center ]{Claim \ref{claim_three_edges}} (\t + 3.15, \r + 0.5);

\draw [dashed, line width=1.1] (a)--(b);
\draw [dashed, line width=1.1] (b)--(c);

\def \t{9};
\node (a) at (\t + 0,\r + 1) {$a_0$};
\node (b) at (\t + 1.15,\r + 1) {$a_1$};
\node (c) at (\t + 1.15,\r + 0) {$a_2$};
\node (d) at (\t + 0, \r + 0) {$a_3$};
\draw [-to, line width=1.5] (\t + 2, \r + 0.5)-- node [text width=2.5cm,midway,above=0.2em,align=center ]{Claim \ref{claim_7.8}} (\t + 3.15, \r + 0.5);

\draw [dashed, line width=1.1] (a)--(b);
\draw [dashed, line width=1.1] (b)--(c);
\draw [loosely dotted, line width=1.1] (a)--(d);

\def \r{0};
\def \t{5};

\draw [-to, line width=1.5] (\t-2, \r + 0.5) -- node [text width=2.5cm,midway,above=0.2em,align=center ]{Claim \ref{claim_7.8}} (\t-0.85, \r + 0.5);

\node (a) at (\t + 0,\r + 1) {$a_0$};
\node (b) at (\t + 1.15,\r + 1) {$a_1$};
\node (c) at (\t + 1.15,\r + 0) {$a_2$};
\node (d) at (\t + 0, \r + 0) {$a_3$};
\draw [-to, line width=1.5] (\t + 2, \r + 0.5)-- node [text width=2.5cm,midway,above=0.2em,align=center ]{Claim \ref{claim_7.8}} (\t + 3.15, \r + 0.5);

\draw [line width=1.1] (a)--(b);
\draw [dashed, line width=1.1] (b)--(c);
\draw [loosely dotted, line width=1.1] (a)--(d);
\draw [line width=1.1] (d)--(c);

\def \t{9};
\node (a) at (\t + 0,\r + 1) {$a_0$};
\node (b) at (\t + 1.15,\r + 1) {$a_1$};
\node (c) at (\t + 1.15,\r + 0) {$a_2$};
\node (d) at (\t + 0, \r + 0) {$a_3$};

\draw [line width=1.1] (a)--(b)--(c)--(d)--(a);

\end{tikzpicture}
\end{subfigure}

\caption{}\label{fig_C_4_flow} 
\end{figure}

    We now execute this plan rigorously.
    
    \begin{claim}\label{claim_7.6}
    For all $\eta>0$ the following holds for large enough $n$
        \[
            \sum_{(x,y,z,w)\in \mathcal{A}}u_{xy}u_{yz}u_{zw}u_{wx}\leq \eta^2 n^4p^4,
        \]
        where $\mathcal{A}$ is the set of all tuples $(x,y,z,w)\in \mathcal{C}_4$ with 
        \[
            u_{xy},u_{zw}\leq \sqrt{p}/\log(1/p)\quad \text{or}\quad u_{xw},u_{yz}\leq \sqrt{p}/\log(1/p).
        \]
    \end{claim}
    
    \begin{proof}
        Suppose otherwise, meaning,
        \[
            \sum_{(x,y,z,w)\in \mathcal{A}}u_{xy}u_{yz}u_{zw}u_{wx}> \eta^2 n^4p^4.
        \]
        By the definition of $\mathcal{A}$ we have 
        \begin{align*}
            u_{xy},u_{zw}\leq \sqrt{p}/\log(1/p)\quad \text{or}\quad u_{xz},u_{yw}\leq \sqrt{p}/\log(1/p).
        \end{align*}
        This implies that 
        \[
            \eta ^2 n^4p^4 < \sum_{(x,y,z,w)\in \mathcal{A}}u_{xy}u_{yz}u_{zw}u_{wx}\leq \left(\sum_{x,y}u_{xy}\right)^2\frac{p}{\log^2(1/p)},
        \]
        and therefore,
        \[
            \eta n^2p^{3/2}\log(1/p) \leq \sum_{x,y}u_{xy}.
        \]
        By Lemma \ref{lem_estimation of B} we have $\sum_{x,y}u_{xy}=O\left(n^2 p^{3/2}\sqrt{\log(1/p)}\right)$. This implies,
        \[
            \eta n^2p^{3/2}\log(1/p)\leq O\left(n^2 p^{3/2}\sqrt{\log(1/p)}\right),
        \]
        contradicting the assumption that $p=o(1)$.
    \end{proof}
    
    \begin{claim}\label{claim_three_edges}
    For all $\eta>0$ the following holds for large enough $n$
        \[
        \sum_{(a_0,a_1,a_2,a_3)\in \mathcal{B}\setminus \mathcal{A}}u_{a_0a_1}u_{a_1a_2}u_{a_2a_3}u_{a_3a_0}\leq \eta^2 n^4p^4,
        \]
        where $\mathcal{B}=\cup_{i=0}^{3}\mathcal{B}_i$, and $\mathcal{B}_i$ is the set of all $(a_0,a_1,a_2,a_3)\in \mathcal{C}_4$ with
        \[
            u_{a_ia_{i+1}},u_{a_{i+1}a_{i+2}}\leq p^{1-2\varepsilon}.
        \]
    \end{claim}
    \begin{proof}
        Assume towards contradiction that
        \[
            \sum_{(a_0,a_1,a_2,a_3)\in \mathcal{B}\setminus \mathcal{A}}u_{a_0a_1}u_{a_1a_2}u_{a_2a_3}u_{a_3a_0}> \eta^2 n^4p^4.
        \]
        We claim that for any $(a_0,a_1,a_2,a_3)\in \mathcal{B}\setminus \mathcal{A}$ there is $i$ satisfying
        \[
            u_{a_{i}a_{i+1}},u_{a_{i+1}a_{i+2}}>\sqrt{p}/\log(1/p)\quad \text{and}\quad u_{a_{i+2}a_{i+3}},u_{a_{i+3}a_{i}}\leq p^{1-2\varepsilon},
        \]    
        where summation is taken modulo $4$.
        Indeed, let $(a_0,a_1,a_2,a_3)\in \mathcal{B}\setminus\mathcal{A}$. Since $(a_0,a_1,a_2,a_3) \not \in \mathcal{A}$  we have 
        \[
            u_{a_{0}a_{1}}>\sqrt{p}/\log(1/p) \quad \text{or}\quad u_{a_{2}a_{3}}>\sqrt{p}/\log(1/p)   
        \]
        and 
        \[
            u_{a_{0}a_{3}}>\sqrt{p}/\log(1/p) \quad\text{or}\quad  u_{a_{1}a_{2}}>\sqrt{p}/\log(1/p).
        \]
        Without loss of generality, assume that $u_{a_{0},a_{1}},u_{a_{1},a_{2}}>\sqrt{p}/\log(1/p)$. Note that as $(a_0,a_1,a_2,a_3)\in \mathcal{B}$ there is $i$ such that $u_{a_ia_{i+1}},u_{a_{i+1}a_{i+2}}\leq p^{1-2\varepsilon}$. For $\varepsilon<1/4$ we have $\sqrt{p}/\log(1/p)\geq p^{1-2\varepsilon}$ implying that $i\neq 0,1,3$ and therefore, $i=2$.
        We continue by letting $\mathcal{L}_1$ be the set of all $\{x,y\}\in \binom{[n]}{2}$ such that $u_{xy}>\sqrt{p}/\log(1/p)$, and we claim that
        \begin{align*}
            \sum_{(a_0,a_1,a_2,a_3)\in \mathcal{B}\setminus \mathcal{A}}u_{a_0a_1}u_{a_1a_2}u_{a_2a_3}u_{a_3a_0}\leq np^{2(1-2\varepsilon)}\left(\sum_{\{x,y\}\in\mathcal{L}_1}u_{xy}\right)^2.
        \end{align*}
        This follows from the fact that any $K_{1,2}$ in $K_n$ can be extended to a $C_4$ in at most $n$ ways. Therefore, we obtain the following:
        \[
            \sum_{\{x,y\}\in\mathcal{L}_1}u_{xy} > \eta n^{3/2}p^{1+2\varepsilon}.
        \]
        By Lemma \ref{lem_exact_estimations} we have,
        \[
            \sum_{x,y}I_{p}(p+u_{xy})\geq \sum_{\{x,y\}\in \mathcal{L}_1}I_{p}(p+u_{xy})=\sum_{\{x,y\}\in \mathcal{L}_1}(1+o(1))u_{xy}\log(u_{xy}/p)=\Omega\left(n^{3/2}p^{1+2\varepsilon}\log(1/p)\right).
        \]
        This is a contradiction to our assumption that $\sum_{x,y}I_{p}(p+u_{xy})=O(n^2p^2\log(1/p))$ as $p\ll n^{-1/2-\varepsilon}$ and therefore, $n^{3/2}p^{1+2\varepsilon}=\omega(n^2p^2)$.
    \end{proof}

    \begin{claim}\label{claim_7.8} For any $\eta>0$ there exists $\gamma>0$ such that the following holds provided $n$ is large enough
        \[
        \sum_{(a_0,a_1,a_2,a_3)\in \mathcal{D(\gamma)}}u_{a_0a_1}u_{a_1a_2}u_{a_2a_3}u_{a_3a_0}\leq \eta ^2 n^4p^4,
        \]
        where $\mathcal{D}(\gamma)=\mathcal{D}_0(\gamma)\cup \mathcal{D}_1(\gamma)$, and $\mathcal{D}_i(\gamma)$ is the set of all tuples $(a_0,a_1,a_2,a_3)\in \mathcal{C}_4$ with 
        \[
            u_{a_{i}a_{i+1}}\leq \gamma \quad \text{or}\quad u_{a_{i+2}a_{i+3}}\leq \gamma,
        \]
        and
        \[
            u_{a_{i+1}a_{i+2}}, u_{a_{i}a_{i+3}}\geq p^{1-2\varepsilon}.
        \]
    \end{claim}

     \begin{proof}
        Let $\gamma = \left({\eta \varepsilon}/{C}\right)^2$ and assume towards contradiction that
        \[
            \sum_{(a_0,a_1,a_2,a_3)\in \mathcal{D(\gamma)}}u_{a_0a_1}u_{a_1a_2}u_{a_2a_3}u_{a_3a_0}> \eta^2 n^4p^4.
        \]
        Note that any $(a_0,a_1,a_2,a_3)\in \mathcal{D}(\gamma)$ satisfies
        \[
            u_{a_{i}a_{i+1}}\leq \gamma \quad\text{or}\quad u_{a_{i+2}a_{i+3}}\leq \gamma 
        \]    
        and 
        \[
            u_{a_{i}a_{i+3}},u_{a_{i+1}a_{i+2}} \geq p^{1-2\varepsilon},
        \]
        for some $i$ and where summation is taken modulo $4$. Letting $\mathcal{L}_2$ be the set of all $\{x,y\}\in \binom{[n]}{2}$ such that $u_{xy}\geq p^{1-2\varepsilon}$ we have the following by the definition of $\mathcal{D}(\gamma)$:
        \begin{align*}
            \sum_{(a_0,a_1,a_2,a_3)\in \mathcal{D}(\gamma)}u_{a_0a_1}u_{a_1a_2}u_{a_2a_3}u_{a_3a_0}\leq \gamma \left(\sum_{\{x,y\}\in \mathcal{L}_2}u_{xy}\right)^2.
        \end{align*}
        This implies that
        \[
            \sum_{\{x,y\}\in \mathcal{L}_2}u_{xy}> \eta n^{2}p^{2}/\sqrt{\gamma}.
        \]
        Therefore, by the definition of $\mathcal{L}_2$ and by Lemma \ref{lem_exact_estimations} we have,
        \begin{align*}
            \sum_{x,y}I_{p}(p+u_{xy})&\geq \sum_{\{x,y\}\in \mathcal{L}_2}I_{p}(p+u_{xy})\geq  \sum_{\{x,y\}\in \mathcal{L}_2}u_{xy}\log(u_{xy}/p)\geq \frac{2\eta\varepsilon }{\sqrt{\gamma}}n^2p^2\log(1/p).
        \end{align*}
        This is a contradiction to  the choice of $\gamma$ as by our assumptions we have $\sum_{x,y}I_{p}(p+u_{xy})\leq Cn^2p^2\log(1/p)$.
    \end{proof}
    
    We now proceed with the proof of Proposition \ref{prop_solution_to_var_problem}. By the our assumptions and the above three claims we obtain the following for any $\eta>0$ and small enough $\gamma>0$:
    \[
        \sum_{(a_0,a_1,a_2,a_3)\in \mathcal{C}_4\setminus\left(\mathcal{A}\cup \mathcal{B}\cup \mathcal{D}(\gamma)\right)} u_{a_0a_1}u_{a_1a_2}u_{a_2a_3}u_{a_3a_0}(1-u_{a_0a_2})(1-u_{a_1a_3})\geq \left(\delta-\eta \right)\E[X].
    \]
    
    We claim that provided $n$ is large enough we have:
    \[  
        \mathcal{C}_4\setminus\left(\mathcal{A}\cup \mathcal{B}\cup \mathcal{D}(\gamma)\right)\subseteq \mathcal{C}_4^\gamma.
    \]
    
    Indeed, as $(a_0,a_1,a_2,a_3)\not \in \mathcal{A}$, there exists $i$ such that provided $\varepsilon<1/4$ and $n$ is large enough we have:
    \[
        u_{a_{i}a_{i+1}},u_{a_{i+1}a_{i+2}}>\sqrt{p}/\log(1/p)> p^{1-2\varepsilon}.
    \]
    Further, as $(a_0,a_1,a_2,a_3)\not \in \mathcal{B}_{i+2}$ we have
    \[
        u_{a_{i+2}a_{i+3}}>p^{1-2\varepsilon} \quad \text{or} \quad u_{a_{i}a_{i+3}}>p^{1-2\varepsilon}.
    \]
    Without loss of generality, assume the latter holds. In particular we have,
    \[
        u_{a_{i}a_{i+3}},u_{a_{i+1}a_{i+2}}>p^{1-2\varepsilon}.
    \]
    Since $(a_0,a_1,a_2,a_3)\not \in \mathcal{D}_i(\gamma)$, we have
    \[
        u_{a_{i}a_{i+1}},u_{a_{i+2}a_{i+3}}>\gamma.
    \]
    Finally, as $(a_0,a_1,a_2,a_3)\not \in \mathcal{D}_{i+1}(\gamma)$, for large enough $n$ we have $p^{1-2\varepsilon}<\gamma$ implying that 
    \[
        u_{a_{i}a_{i+3}},u_{a_{i+1}a_{i+2}}>\gamma.
    \]
    
    Letting $\mathcal{L}_3$ be the set of all $\{x,y\}\in \binom{[n]}{2}$ with $u_{xy}>\gamma$, we claim the following holds:
    \begin{equation}\label{eq_induced_bound}
         \sum_{(a_0,a_1,a_2,a_3)\in \mathcal{C}_4^{\gamma}} u_{a_0a_1}u_{a_1a_2}u_{a_2a_3}u_{a_3a_0}(1-u_{a_0a_2})(1-u_{a_1a_3})\leq \frac{\left(\sum_{\{a_0,a_1\}\in \mathcal{L}_3}u_{a_0a_1}\right)^2}{4}.
    \end{equation}
    Indeed, 
    \begin{align*}
        2\sum_{(a_0,a_1,a_2,a_3)\in\mathcal{C}_4^{\gamma}} u_{a_0a_1}u_{a_1a_2}u_{a_2a_3}u_{a_3a_0}(1-u_{a_0a_2})(1-u_{a_1a_3})
    \end{align*}
    is at most
    \[
        \sum u_{a_0a_1}u_{a_2a_3}(\underbrace{u_{a_1a_2}u_{a_3a_0}(1-u_{a_0a_2})(1-u_{a_1a_3})+u_{a_0a_2}u_{a_1a_3}(1-u_{a_1a_2})(1-u_{a_3a_0})}_{(\star)}),
    \]
    where the sum ranges over unordered pairs, $\{\{a_0,a_1\},\{a_2,a_3\}\}\subseteq \mathcal{L}_3$, such that $a_0,a_1,a_2,a_3$ are all distinct.
    We also have
    \[
       (\star)\leq (1-u_{a_0a_2})+u_{a_0a_2}=1.
    \]
    This implies that 
    \[
         (\delta -\eta)\E[X]\leq \sum_{(a_0,a_1,a_2,a_3)\in \mathcal{C}_4^{\gamma}} u_{a_0a_1}u_{a_1a_2}u_{a_2a_3}u_{a_3a_0}(1-u_{a_0a_2})(1-u_{a_1a_3})\leq \frac{\sum u_{a_0a_1}u_{a_2a_3}}{2},
    \]
    where the summation in the right-hand side is over unordered pairs, $\{\{a_0,a_1\},\{a_2,a_3\}\}\subseteq \mathcal{L}_3$, such that $a_0,a_1,a_2,a_3$ are all distinct.
    This implies \eqref{eq_induced_bound} as
    \[
         \frac{\left(\sum_{\{a_0,a_1\}\in \mathcal{L}_3}u_{a_0a_1}\right)^2}{2}=\frac{\sum_{\{a_0,a_1\}\in \mathcal{L}_3}u^2_{a_0a_1}+2\sum u_{a_0a_1}u_{a_2a_3}}{2}\geq {\sum u_{a_0a_1}u_{a_2a_3}},
    \]
    where the unlabeled summations are over unordered pairs $\{\{a_0,a_1\},\{a_2,a_3\}\}\subseteq \mathcal{L}_3$.
    We conclude that:
    \[
         2\sqrt{(\delta-\eta)\E[X]}\leq \sum_{\{a_0,a_1\}\in \mathcal{L}_3}u_{a_0a_1}.
    \]
    Now we can bound from below the `cost' using Lemma \ref{lem_exact_estimations} as follows:
    \begin{align*}
        \sum_{\{x,y\}\in \binom{[n]}{2}}I_{p}(p+u_{xy}) &\geq \sum_{\{x,y\}\in \mathcal{L}_3}I_{p}(p+u_{xy}) = (1+o(1))\sum_{\{x,y\}\in \mathcal{L}_3}u_{xy}\log(u_{xy}/p)\\
        &\geq  (1+o(1))\sum_{\{x,y\}\in \mathcal{L}_3}u_{xy}\log(1/p)\geq  (1+o(1))2\sqrt{(\delta-\eta)\E[X]}\log(1/p).
    \end{align*}    
    This finishes the proof of the proposition as $\E[X]=(1+o(1))\frac{n^2p^2}{8}.$\qedhere
    
\end{proof}

\bibliography{bibliography.bib}
\bibliographystyle{amsplain}

\appendix
\section{Translating graphons' language into graphs' language}\label{app_A}

The aim of this Appendix is to add a proof (in the language of graphs instead of graphons) for Lemmas \ref{lem_only_C_4} and \ref{lem_estimation of B} which were proven by Lubetzky--Zhao \cite{LubZha17} and Bhattacharya, Ganguly, Lubetzky and Zhao \cite{BhaGanLubZha17} in the language of graphons.

We start with a discussion of how one would prove Lemma \ref{lem_only_C_4}, then we state some lemmas and an extension of Lemma \ref{lem_estimation of B}. To prove Lemma \ref{lem_only_C_4} we will use Lemma \ref{lem_estimation of B}, and therefore it will be proven before Lemma \ref{lem_only_C_4}.
Recall Lemma \ref{lem_only_C_4}:

\lemmaone*
    
    We start with by analysing the term $\E[N_{ind}(C_4,G_{n,\bar{q}})]$. Suppose $\bar{q}\in [0,1]^{\binom{n}{2}}$ satisfies the conditions of Lemma \ref{lem_only_C_4}. Note that
    \[
        \E[N_{ind}(C_4,G_{n,\bar{q}}))]={\sum_{(x,y,z,w)\in \mathcal{C}_4}q_{xy}q_{yz}q_{zw}q_{wx}(1-q_{xz})(1-q_{yw})}.
    \]
    Since $q_i=p+u_i\geq u_i$ we have the following for large enough $n$: 
    \begin{align*} \E[N_{ind}(C_4,G_{n,\bar{q}}))]&\leq {\sum_{(x,y,z,w)\in \mathcal{C}_4}q_{xy}q_{yz}q_{zw}q_{wx}(1-u_{xz})(1-u_{yw})}\\
    &={\sum_{(x,y,z,w)\in \mathcal{C}_4}(p+u_{xy})(p+u_{yz})(p+u_{zw})(p+u_{wx})(1-u_{xz})(1-u_{yw})}\\
    &=\E[N(C_4,G_{n,p})]+\E[N_{ind}(C_4,G_{n,\bar{u}})]+\sum_{\emptyset \not = H\subsetneq^* C_4}\E[N(H,G_{n,\bar{u}})]p^{4-e_H}\\
    &=\E[X]/(1-p)^2+\E[N_{ind}(C_4,G_{n,\bar{u}})]+\sum_{\emptyset \not = H\subsetneq^* C_4}\E[N(H,G_{n,\bar{u}})]p^{4-e_H}\\
    &=(1+\varepsilon/2)\E[X]+\E[N_{ind}(C_4,G_{n,\bar{u}})]+\sum_{\emptyset \not = H\subsetneq^* C_4}\E[N(H,G_{n,\bar{u}})]p^{4-e_H},
    \end{align*}
    where $\subsetneq ^*$ stands for a spanning subgraph which is not equal to the host graph.
    Recalling that $\E[N_{ind}(C_4,G_{n,\bar{q}}))]-\E[X]\geq (\delta-\varepsilon) \E[X]$, we obtain,
        \begin{align*}
        (\delta-3\varepsilon/2)\E[X]\leq & \, \E[N_{ind}(C_4,G_{n,\bar{u}})]+\E[N(P_4,G_{n,\bar{u}})]p+\E[N(M_2,G_{n,\bar{u}})]p^2\\
                                        &\, +\E[N(K_{1,2}\sqcup K_1,G_{n,\bar{u}})]p^2+\E[N(K_2\sqcup 2K_1,G_{n,\bar{u}})]p^3.\\
                                        \leq&\, \E[N_{ind}(C_4,G_{n,\bar{u}})]+\E[N(P_4,G_{n,\bar{u}})]p+\E[N(M_2,G_{n,\bar{u}})]p^2\\ &\,+\E[N(K_{1,2},G_{n,\bar{u}})]np^2+\E[N(K_2,G_{n,\bar{u}})]n^2p^3,
    \end{align*}
    where $P_4$ is the path with three edges, $M_2$ is the matching of size two, $K_{1,2}\sqcup K_1$ is the complete bipartite graph with sides of size $1$ and $2$ and an extra isolated vertex and $K_2\sqcup 2K_1$ is the disjoint union of an edge and an independent set of size two.
    Therefore, to prove the lemma, it is enough to show that all of the terms in the right-hand side of the above inequality except for $\E[N_{ind}(C_4,G_{n,\bar{u}})]$ are negligible in comparison to $n^4p^4$.
    
    For this we remind the reader Lemma \ref{lem_exact_estimations} and cite two useful lemmas from \cite{LubZha17} and \cite{BhaGanLubZha17}:
    
    \lemmatwo*
    
    \begin{lemma}[\cite{LubZha17}]\label{lem_not_exact_estimations_Appenxid}
        There exists $p_0>0$ such that for all $0<p\leq p_0$ and $0\leq x\leq b\leq 1-p-1/\log(1/p)$ we have,
        \[
            I_{p}(p+x)\geq \frac{x^2I_{p}(p+b)}{b^2}.
        \]
    \end{lemma}
    
    \begin{cor}[\cite{LubZha17}]\label{cor_not_exact_estimations_Appenxid}
        There exists $p_0 > 0$ such that for all $0 < p \leq p_0$ and all $0 \leq  x \leq 1-p$ we have,
        \[
        I_p(p + x) \geq x^2 I_p(1-1/\log(1/p)) = (1 + o(1))x^2I_p(1)
        \]
        where the $o(1)$-term goes to zero as $p \to 0$.
    \end{cor}
    
    Another useful tool we use is the generalized H\"{o}lder inequality:
    
    \begin{lemma}[Generalized H\"{o}lder inequality \cite{BhaGanLubZha17}]
        Let $\mu_1,\mu_2,\ldots ,\mu_n$ be probability measures on $\Omega_1,\Omega_2,\ldots,\Omega_n$ respectively, and let $\mu=\prod_{i=1}^n \mu_i$. Let $A_1,A_2,\ldots,A_m$ be non-empty subsets of $[n]$ and for $A\subseteq [n]$ put $\mu_A =\prod_{j\in A}\mu_j$ and $\Omega_A=\prod_{j\in A}\Omega_j$. Let $f_i\in L^{p_i}(\Omega_{A_i},\mu_{A_i})$ for each $i\in[m]$, and suppose that $\sum_{i:A_{i}\ni j}\frac{1}{p_i}\leq 1$ for all $j\in [n]$. Then,
        \[
            \int \prod_{i=1}^m|f_i|d\mu \leq \prod_{i=1}^m \left(\int|f_i|^{p_i}d\mu_{A_i}\right)^{1/p_i}.
        \]
    \end{lemma}
    
    In particular, if every element of $[n]$ is contained in at most two sets $A_j$, then we can take all $p_i=2$ and obtain:
    \[
        \int \prod_{i=1}^m f_i d\mu \leq \prod_{i=1}^m \left(\int |f_i|^2d\mu_{A_i} \right)^{1/2}.
    \]
    
    Next, we recreate a proof of Lubetzky--Zhao \cite{LubZha17} and Bhattacharya, Ganguly, Lubetzky and Zhao \cite{BhaGanLubZha17} of an extension of Lemma \ref{lem_estimation of B}. Then, we will derive the required bounds.

    \begin{lemma}\label{lem_estimation of B_Appendix}
        Suppose $\varepsilon,\delta,C$ are positive reals and $\sqrt{\log(n)}/n\ll p\ll n^{-1/2}$. Suppose also that $\bar{u}\in [0,1-p]^{\binom{n}{2}}$ such that $\sum_{i\in \binom{[n]}{2}}I_{p}(p+u_i)\leq Cn^2p^2\log(1/p)$.
        Let $b=b(n)$ be such that $\max\{np^2,\sqrt{p\log(1/p)}\}\ll b\leq 1-\varepsilon$. Then, provided $n$ is large enough there is a constant $D>0$ such that the following holds:
        \begin{enumerate}[label=(\roman*)]
            \item \label{eq_estimation_2_Appendix} $\sum_{x\in [n]}\left(\sum_{y\neq x}u_{xy}\right)^2\leq Dn^{3}p^{2}b,$
            \item \label{eq_estimation_3_Appendix} $\sum_{i\in \binom{[n]}{2}}u_i\leq Dn^2 p^{3/2}\sqrt{\log(1/p)}$, and
            \item \label{eq_estimation_4_Appendix} $\sum_{i\in \binom{[n]}{2}}u_i^2 \leq Dn^2 p^{2}.$
        \end{enumerate}
    \end{lemma}

\begin{proof}
    We first show that the set $B=\{x\in [n]:\sum_{y\neq x }u_{xy}\geq bn\}$ is empty, provided $n$ is large enough. If this were not true, then, as $I_{p}(p+x)$ is a convex function of $x$, we have the following by Jensen's inequality
    \begin{align*}
        \sum_{x,y}I_{p}(p+u_{xy})&\geq (n-1)\sum_{x}I_{p}\left(p+\frac{\sum_{y\neq x }u_{xy}}{n-1}\right)\geq (n-1)\sum_{x\in B} I_{p}\left(p+\frac{\sum_{y\neq x }u_{xy}}{n-1}\right)\\
        &\geq (n-1)\sum_{x\in B}I_{p}(p+b)\geq {(1+o(1))(n-1)|B|b\log(b/p)}
    \end{align*}
    where the last inequality follows from Lemma \ref{lem_exact_estimations}. Since $b\gg \sqrt{p\log(1/p)}$ and $\sum_{i\in \binom{[n]}{2}}I_{p}(p+u_i)\leq Cp^2n^2\log(1/p)$ we obtain the following:
    \[
        |B|\leq (1+o(1))\frac{Cp^2n^2\log(1/p)}{(n-1)b\log(\sqrt{\log(1/p)/p})} \ll 1,
    \]
    and therefore, $B$ is empty.
    To prove \ref{eq_estimation_2_Appendix}, we use the convexity of $I_{p}(p+x)$ and Lemma \ref{lem_not_exact_estimations_Appenxid},
    \begin{align*}
        \sum_{x,y}I_p(p+u_{xy})&\geq (n-1)\sum_{x}I_{p}\left(p+\frac{\sum_{y\neq x }u_{xy}}{n-1}\right)= (n-1)\sum_{x\not \in B}I_{p}\left(p+\frac{\sum_{y\neq x }u_{xy}}{n-1}\right)\\
        &\geq (n-1)I_{p}(p+b)\sum_{x\not \in B}\left(\frac{\left(\sum_{y\neq x }u_{xy}\right)^2}{b^2n^2}\right).
    \end{align*}
    
    Since, $\sum_{i\in \binom{[n]}{2}}I_{p}(p+u_i)\leq Cp^2n^2\log(1/p)$ we obtain:
    
    \[
        \sum_{x}\left(\sum_{y\neq x }u_{xy}\right)^2\leq \frac{Cp^2n^4b^2\log(1/p)}{(n-1)I_{p}(p+b)}.
    \]
    As $b\gg \sqrt{p\log(1/p)}\gg p$ we may use Lemma \ref{lem_exact_estimations} and we obtain:
    \[
         \sum_{x\in [n]}\left(\sum_{y\neq x}u_{xy}\right)^2 \leq \frac{Cn^4p^{2}b^2\log(1/p)}{(1+o(1))(n-1)b\log(b/p)} \leq  2Cn^{3}p^{2}b.
    \]
    
    Now we prove \ref{eq_estimation_3_Appendix}. By convexity we have,
    \[
        \binom{n}{2}I_{p}\left(p+\frac{\sum_{i\in \binom{[n]}{2}}u_i}{\binom{n}{2}}\right)\leq \sum_{x,y}I_{p}(p+u_{xy})\leq Cn^2p^2\log(1/p).
    \]
    Since $p\gg \sqrt{p^3\log(1/p)}$ we may use Lemma \ref{lem_exact_estimations}, and obtain the following for large enough $n$: 
    \[
        I_{p}\left(p+\sqrt{12Cp^{3}\log(1/p)}\right)\geq \frac{12Cp^3\log(1/p)}{3p}=4Cp^2\log(1/p).
    \]
    This implies the following provided $n$ is large enough,
    \[
        I_p\left(p+\frac{\sum_{i\in \binom{[n]}{2}}u_i}{\binom{n}{2}}\right)\leq 4Cp^2\log(1/p)\leq I_{p}\left(p+\sqrt{12Cp^{3}\log(1/p)}\right).
    \]
    By the monotonicity of $I_{p}(p+x)$ for $x>0$ we obtain \ref{eq_estimation_3_Appendix}. That is 
    \[
        \sum_{i\in\binom{[n]}{2}}u_i\leq \sqrt{3Cn^4p^{3}\log(1/p)}.
    \]
    Lastly, we prove \ref{eq_estimation_4_Appendix}. By Corollary \ref{cor_not_exact_estimations_Appenxid},
    \[
      \sum_{i\in \binom{[n]}{2}}I_{p}(p+u_i)\geq (1+o(1))\sum_{i\in \binom{[n]}{2}}u_i^2\log(1/p).
    \]
    This and the assumption of the lemma implies \ref{eq_estimation_4_Appendix}.
    \end{proof}

    Now we can finish the proof of Lemma \ref{lem_only_C_4}.

    \begin{lemma}\label{lem_P_4_and_K_1,2_are_small}
        Suppose $\delta,C$ are positive reals and $\sqrt{\log(n)}/n\ll p\ll n^{-1/2}$. Suppose also that $\bar{u}$ is a sequence of $\binom{n}{2}$ reals between $0$ and $1$ such that $\sum_{i\in \binom{[n]}{2}}I_{p}(p+u_i)\leq Cp^2n^2\log(1/p)$. Then,
        \[
            \E[N(M_2,G_{n,\bar{u}})]p^2=o(n^4p^4),
        \]
        \[
            \E[N(K_2,G_{n,\bar{u}})]n^2p^3=o(n^4p^4),
        \]
        \[
            \E[N(P_4,G_{n,\bar{u}})]p=o(n^4p^4),
        \]
        \[
            \E[N(K_{1,2},G_{n,\bar{u}})]np^2=o(n^4p^4).
        \]
        
    \end{lemma}

    \begin{proof}
    By item \ref{eq_estimation_3_Appendix} in Lemma \ref{lem_estimation of B_Appendix} we have $\E[e(G_{n,\bar{u}})]=o(n^2p^2)$, and therefore,
    \[
        \E[N_{ind}(M_{2},G_{n,\bar{u}})]p^2\leq \E[N(M_2,G_{n,\bar{u}})]p^2\leq \E[e(G_{n,\bar{u}})]^2p^2=o(n^4p^4),
    \]
    \[
        \E[N_{ind}(K_2,G_{n,\bar{u}})]n^2p^3=\E[e(G_{n,\bar{u}})]n^2p^3=o(n^4p^4),
    \]
    where $N(M_2,G_{n,\bar{u}})$ is the number of $M_2$ in $G_{n,\bar{u}}$. 
    Let $\max\{np^2,\sqrt{p\log(1/p)}\}\ll b= b(n)\ll 1$. Note that each $P_4\subseteq K_{n}$ with vertices $a_0,a_1,a_2,a_3$ can be represented in exactly two ways as a tuple $(a_0,a_1,a_2,a_3)$ where edges are consecutive vertices in this tuple. Let $\mathcal{P}_4$ be a collection of exactly one such representative for each $P_4$ in $K_n$. Observe the following:
     \begin{align*}
         \E[N(P_4,G_{n,\bar{u}})]&= \sum_{(x,y,z,w)\in \mathcal{P}_4 }u_{xy}u_{yz}u_{zw}\leq \sum_{y,z,w\in[n]} \left(\sum_{x\neq y}u_{xy}\right) u_{yz}u_{zw},
     \end{align*}
    where $N(P_4,G_{n,\bar{u}})$ is the number of $P_4$ in $G_{n,\bar{u}}$. 
    By the generalized H\"{o}lder inequality we have:
    \[
        \sum_{y,z,w\in[n]} \left(\sum_{x\neq y}u_{xy}\right) u_{yz}u_{zw}\leq \left(\sum_{y}\E[\deg(y)]^2\right)^{1/2}\left(\sum_{x,y}u_{xy}^2\right).
    \]
    Applying items \ref{eq_estimation_2_Appendix} and \ref{eq_estimation_4_Appendix} in Lemma \ref{lem_estimation of B_Appendix} we obtain
    \[
        \E[N(P_4,G_{n,\bar{u}})]p\leq D^{3/2}n^{7/2}p^{4}\sqrt{b}=o(n^4p^4).
    \]
    This establishes the third assertion of the lemma.
    The fourth assertion of the lemma follows immediately from Lemma \ref{lem_estimation of B_Appendix} item \ref{eq_estimation_2_Appendix} and the definition of $b$:
    \[
        \E[N(K_{1,2},G_{n,\bar{u}})] \leq \sum_{x,y,z\in [n]} u_{xy}u_{yz}\leq  \sum_{y\in [n]}\left(\sum_{x\in [n]}u_{xy}\right)^2\leq Dn^2p^{3/2}b=o(n^3p^2).\qedhere
    \]
    \end{proof}

\end{document}